\documentclass{article}
\usepackage{amsmath,amssymb,latexsym,amsthm,mathrsfs}

\newtheorem{theo}{Theorem}[section]
\newtheorem{pro}[theo]{Proposition}
\newtheorem{lem}[theo]{Lemma}
\newtheorem{cor}[theo]{Corollary}

\newcommand{\ra}{\rightarrow}

\theoremstyle{definition}
\newtheorem{defin}[theo]{Definition}
\newtheorem{nota}[theo]{Notation}
\newtheorem{exa}{Example}[section]

\theoremstyle{remark}
\newtheorem{rem}[theo]{Remark}

\title{Random walks on nilpotent groups driven by measures 
supported on powers of 
generators}
\author{Laurent Saloff-Coste\thanks{Both authors partially supported by NSF
grant DMS 1004771}\\
{\small Department of Mathematics}\\
{\small Cornell University} 
\and 
Tianyi Zheng\\
{\small  Department of Mathematics}\\
{\small Cornell University}
}

\begin{document}
\maketitle

\begin{abstract} We study the decay of convolution powers of
a large family $\mu_{S,a}$ of measures
on finitely generated nilpotent groups. 
Here, $S=(s_1,\dots,s_k)$
is a generating $k$-tuple of group elements  and $a=
(\alpha_1,\dots,\alpha_k)$ is a $k$-tuple of reals in the interval 
$(0,2)$.
The symmetric measure $\mu_{S,a}$ is supported by $S^*=\{s_i^{m}, 
1\le i\le k,\,m\in \mathbb Z\}$ and gives probability proportional to 
$$(1+m)^{-\alpha_i-1}$$
to $s_i^{\pm m}$, $i=1,\dots,k,$ $m\in \mathbb N$.
We determine the behavior of the probability of return 
$\mu_{S,a}^{(n)}(e)$ as $n$ tends to 
infinity. 
This behavior depends in somewhat 
subtle ways on  interactions between the $k$-tuple $a$ and 
the positions of the generators $s_i$ within the lower central 
series $G_{j}=[G_{j-1},G]$, $G_1=G$. 
\end{abstract}

\section{Introduction}\setcounter{equation}{0}
Generating sets play an essential role in the theory of countable groups. 
This is obvious when a group is defined by generators 
and relations or when a group is defined as the subgroup generated 
by a given finite subset of elements in a much larger group.
In this context, the 
larger ambient group serves as a sort of ``black box'' that encodes 
the law of the group.

Given a group $G$ with finite symmetric generating set $A$, 
the simple random walk on $G$ 
can be interpreted as a way to randomly explore the group $G$. Starting 
at the identity element $e$, the position of the walk at time $n$ is the product 
$\xi_1\dots\xi_n$ where the $G$-valued random variables $\xi_i$
are independent equidistributed with law given by 
the uniform probability on the set $A$.  
More generally, given a probability measure $\mu$ on $G$, the random
walk driven by $\mu$ corresponds to taking the sequence $(\xi_i)$ to be
i.i.d.\ with law $\mu$ and the position at time $n$ has distribution 
$\mu^{(n)}$, the $n$-fold convolution product of $\mu$ with itself.
In particular, $\mathbf P_e (\xi_1\dots \xi_n=e)=\mu^{(n)}(e)$. 
In the case of the simple random walk based on the generating set $A$, 
$\mu=|A|^{-1}\mathbf 1_A$.  

Not surprisingly, many aspects of the behavior of these random processes  
are closely related to the algebraic and geometric property of the 
underlying group $G$. Harry Kesten introduced this question
in his Ph.D. thesis published in 1958. One of Kesten's fundamental results
states that, for a random walk driven by a symmetric measure with generating 
support, the probability of return, 
$\mathbf P_e(\xi_1\dots\xi_n=e)$, decays exponentially fast if and only 
the group $G$ is non-amenable. See \cite{Kest1,Kest2}.

\subsection{The measures $\mu_{S,a}$}

This is the first of a series of papers where we study  
a natural family of random walks  
driven by measures $\mu_{S,a}$ which are defined as follows. 
The letter $S$ represents a finite generating tuple, i.e.,  
a list $S=(s_1,s_2,\dots,s_k)$ of generators 
(repetitions are permitted). 
In addition, we are given a $k$-tuple $a$ of (extended) positive reals 
$a=(\alpha_1,\alpha_2,\dots,\alpha_k)$, $\alpha_i\in (0,\infty]$.
The measure $\mu_{S,a}$ allows long steps along any of the one-parameter group 
$\langle s_i\rangle=\{s_i^n: n\in \mathbb Z\}$, $1\le i\le k$. 
The probability of such a long step along $\langle s_i\rangle$ is given by a 
power law whose exponent $\alpha_i$ is the $i$-th entry of the tuple $a$. 
Namely, we set,
\begin{equation}\label{muS}
\mu_{S,a}(g)=\frac{1}{k}\sum_{i=1}^k
c(\alpha_i) \sum_{m\in \mathbb Z}  (1+|m|)^{-\alpha_i-1}\mathbf 1_{s_i^m}(g)
\end{equation}
where 
$$c(\alpha)^{-1}= \sum_{\mathbb Z}(1+ |m|)^{-\alpha-1}.$$
We make the somewhat arbitrary convention that if $\alpha=\infty$ then 
$(1+|m|)^{-\alpha-1}=0$ unless 
$m=0,\pm 1$ in which case $(1+|m|)^{-\alpha-1}=1$. Note that $\mu_{S,a}$ is 
symmetric, that is, satisfies $\mu_{S,a}(g^{-1})=\mu_{S,a}(g)$.
 We can also describe $\mu_{S,a}$ as the push-forward of the probability 
measure $\mu_{a}$ 
on the free group $\mathbf F_k$ on $k$ generators $\mathbf s_i$, $1\le i\le k$, 
which gives probability 
$$\mu_{a}(\mathbf s_i^{\pm m})=k^{-1}c(\alpha_i)(1+|m|)^{-\alpha_i-1} \mbox{ to }
\mathbf s_i^{\pm m}.$$
Indeed, if $\pi$ is the projection from $\mathbf F_k$ onto $G$ which sends
$\mathbf s_i$ to $s_i$, $$\mu_{S,a}(g)=\mu_{a}(\pi^{-1}(g)).$$

On $\mathbb Z$, the power laws  $\mu_\alpha(\pm k)=c(\alpha) (1+|k|)^{-\alpha-1}$
are very natural probability measures. For $\alpha\in (0,2)$, $\mu_\alpha$ 
can be viewed as a discrete version of the symmetric stable laws which is the probability distribution on $\mathbb R$ whose Laplace transform is 
$e^{-|y|^\alpha}$.

The main result of this paper, Theorem \ref{th-main1} below, 
describes the behavior of 
$$n\mapsto \mu^{(n)}_{S,a}(e)$$ 
when $G$ is any given finitely generated nilpotent 
group, $S$ is any given finite generating tuple of elements of $G$ and 
the entries of the tuple  $a$ are in $(0,2)$. What makes 
this problem  interesting 
is the interaction between the nature of the 
long jumps allowed in the directions of 
each generators and the non-commutative structure of 
the group. As we shall see, the behaviors of the random walks
driven by the measures $\mu_{S,a}$ capture a wealth of information on 
the algebraic structure of $G$.

Because of the results of \cite{PSCstab} --- in particular, Theorem \ref{th-PSC1}
stated below --- the very precise form of the measure $\mu_{S,a}$ 
defined at (\ref{muS}) is not really essential in determining the behavior of 
$n\mapsto \mu_{S,a}^{(n)}(e)$. Indeed, any symmetric measure $\nu$ on $G$ 
such that $c\nu \le \mu_{S,a} \le C \nu $ will satisfy
$$\nu^{(kn)}(e)\le K \mu_{S,a}^{(n)}(e) \mbox{ and }\;\mu_{S,a}^{(kn)}(e)\le 
K\nu^{(n)}(e)$$ for some $k,K$ independent of $n$. 

\subsection{The case of $\mathbb Z^d$}
In the simplest non-trivial case 
where $G=\mathbb Z^2=\{(x,y): x,y\in \mathbb Z\}$, 
$S=\{(1,0),(0,1)\}$ and $a=(\alpha_1,\alpha_2)\in (0,\infty]^2$, 
it is not hard to see that
$\mu^{(n)}_{S,a}(e)$, $e=(0,0)$, behaves as follows. Set
$$\tilde{\alpha}=\min\{\alpha,2\},\;\;\frac{1}{\beta}= 
\frac{1}{\tilde{\alpha}_1} +\frac{1}{\tilde{\alpha}_2} \mbox{ and }
\gamma=\#\{i: \alpha_i=2\}.$$ 
\begin{enumerate}
\item If $2\not\in \{\alpha_1,\alpha_2\}$,
$\mu^{(n)}_{S,a}(e) \sim c(\alpha_1,\alpha_2) n^{-1/\beta};$
\item If $2\in \{\alpha_1,\alpha_2\}$,
$\mu^{(n)}_{S,a}(e) \simeq 
 n^{-1/\beta}(\log n)^{-\gamma/2}.$
\end{enumerate}

Here and in the rest of this paper $\sim$ and $\simeq$ are used with the 
following meaning. For two functions $f,g$ 
defined either over the  positive reals or the natural numbers, 
we say that $f\sim g$ 
(usually, at $0$ or infinity), if $\lim f/g= 1$. We say that $f\simeq g$ if 
there are constants $c_1$ 
such that $$c_1f(c_2t)\le g(t)\le c_3f(c_4 t)$$ (in a neighborhood of the 
relevant value, usually $0$ or infinity). We recommend to restrict the use of 
$\simeq$ to cases where one of the two functions $f$ or $g$ is monotone.

Next, let us review briefly what happens when $G=\mathbb Z^d$ and 
$S=(s_1,\dots,s_k)$, $k\ge d$.  By hypothesis, $S$ is generating.
Given $a=(\alpha_1,\dots,\alpha_k)$, we extract from $S$ a $d$-tuple
$\Sigma=(\sigma_1,\dots,\sigma_d)$ using the following algorithm.
Set $\Sigma_1=\{\sigma_1=s_{i_1}\}$ where 
$\alpha_{i_1}=\min\{\alpha_i: 1\le i\le k\}$.
For $t\ge 1$, if
$$\Sigma_t=(\sigma_1,\dots,\sigma_t),\;\;
\sigma_1=s_{i_1},\dots,\sigma_t=s_{i_t}$$ 
have been chosen,
pick $\sigma_{t+1}=s_{i_{t+1}}$ 
in $\{s_i: 1\le i\le k\}$ with the properties that  $\alpha_{i_{t+1}}=
\min\{\alpha_j: j \not\in \{i_1,\dots,i_t\}\}$ and the rank of the lattice 
generated by $\Sigma_{t+1}=\Sigma_t\cup\{\sigma_{t+1}\}$ is (strictly) greater than the rank of the lattice generated by $\Sigma_t$. Note that the 
final $d$-tuple 
$\Sigma$ might not generates $\mathbb Z^d$ but 
does generate a lattice of finite index in $\mathbb Z^d$. Set 
$a(\Sigma)=(\alpha_{i_1},\dots,\alpha_{i_d})$.

\begin{theo} Let $G=\mathbb Z^d$. Let $S=(s_i)_1^k$ 
be a generating $k$-tuple. Let $a=(\alpha_i)_1^k\in (0,\infty]^k$. Let  
$\Sigma=(\sigma_i)_1^d$ and $ a(\Sigma)$ be obtained from 
$(S,a)$ by the algorithm described above. 
Set 
$$\gamma=\#\{j\in\{1,\dots,d\}: \alpha_{i_j}=2\} 
\mbox{ and }\frac{1}{\beta}=\sum_{s=1}^d \frac{1}{\tilde{\alpha}_{i_s}}$$ 
where $\tilde{\alpha}=\min\{\alpha,2\}$.
Then we have
$$\mu^{(n)}_{S,a}(e)\simeq \mu^{(n)}_{\Sigma,a(\Sigma)}(e)\simeq n^{-1/\beta}[\log n]^{-\gamma/2}$$
\end{theo}
With some work, this result can be extracted from \cite{Griffin1986}.

\subsection{The main result in its simplest form}\label{sec-main}

The goal of this paper is to prove the following theorem together with more sophisticated assorted results.
\begin{theo}\label{th-main1}
Let $G$ be a
nilpotent group equipped with a generating $k$-tuple 
$S=(s_i)_1^k$ and $a=(\alpha_i)_1^k\in (0,\infty]^k$. Assume
that the subgroup generated by $\{s_i: \alpha_i<2\}$ is of finite index in $G$.
Then there exists a real $D\ge 0$ depending on $(G,S,a)$ such that
$$\mu_{S,a}^{(n)}(e)\simeq n^{-D}.$$
\end{theo}
This statement suggests further questions
including the following three: 
\begin{itemize}
\item  Can we compute $D$? how does it depends on $S$, $a$ and $G$?
\item What happen if the subgroup generated by $\{s_i: \alpha_i<2\}$ is not
of finite index in $G$?
\item What happens on other groups? In particular, how does Theorem \ref{th-main1}
generalize to finitely generated groups of polynomial volume growth?
\end{itemize}
The first question will be answer completely in this paper. Indeed, we would 
not be able to prove the above theorem without a detailed understanding of 
how to compute the real $D$. The exact value of $D$ depends in 
an intricate and  interesting way on (a) the 
commutator structure of $G$, (b) the position of the generators $s_i$ 
in the commutator structure of $G$ and (c)  
the values of the parameters  $\alpha_i$. See Theorem \ref{th-main2} 
in the next subsection.

The second question is rather subtle and will not be completely elucidated 
in this paper although some partial results will be obtain in this direction.

In its full generality, the third question is too wide ranging to be discussed 
here in details. Partial results for various classes of groups 
(e.g., some classes of solvable groups and free groups) will be discussed 
elsewhere. The question regarding  groups of polynomial growth is tantalizing but
appears surprisingly difficult to attack.

\subsection{Weight systems and the value of $D$}\label{sec-wD}
The goal of this section is to give the reader a clear idea of the key 
ingredients that enter
the exact computation of the real $D$ governing the behavior of 
$\mu_{S,a}^{(n)}(e)$ in Theorem \ref{th-main1}.

Consider $S=(s_1,\dots,s_k)$ as a formal alphabet equipped with a weight 
system $\mathfrak w$ which assigns weight $w_i\in (0,\infty)$ to the letter 
$s_i$, $1\le i\le k$.  We extend our alphabet by adjoining to each $s_i$ 
its formal inverse $s_i^{-1}$. Using this alphabet, we build the set 
$\mathfrak C(S,m)$ of all formal commutators of length $m$ by induction on $m$. 
Commutators of length $1$ are the letters  in $S^{\pm 1}$. 
Commutators of length $m$
are the formal expression $c$ of the form $c=[c_1,c_2]$ where $c_1,c_2$ 
are commutators of length $m_1, m_2\ge 1$ with $m_1+m_2=m$. 

The commutators of length 2 are (the $\pm 1$ must be understood 
here as independent of each other)
$$[s^{\pm 1}_i,s^{\pm 1}_j], \;\;1\le i, j\le k.$$
The commutators of length 3 are 
$$[[s^{\pm 1}_i,s^{\pm 1}_j], s_\ell^{\pm 1}],\;
[s^{\pm 1}_i ,[s^{\pm 1}_j,s^{\pm 1}_\ell]], \;\;1\le i, j,\ell \le k.$$
For $1\le i_1, i_2,i_3, i_4 \le k$, the commutators of length 4 are 
$$[[[s^{\pm 1}_{i_1},s^{\pm 1}_{i_2}], 
s_{i_3}^{\pm 1}],s_{i_4}^{\pm 1}],\;\;
[[s^{\pm 1}_{i_1},[s^{\pm 1}_{i_2}, 
s_{i_3}^{\pm 1}]],s_{i_4}^{\pm 1}],\;\;
[[s_{i_1}^{\pm 1},s^{\pm 1}_{i_2}],[s^{\pm 1}_{i_3},s^{\pm 1}_{i_4}]]$$
$$[s_{i_1}^{\pm 1},[[s^{\pm 1}_{i_2},s^{\pm 1}_{i_3}],s^{\pm 1}_{i_4}]],\;\; 
[s_{i_1}^{\pm 1} ,[s^{\pm 1}_{i_2},[s^{\pm 1}_{i_3},s^{\pm 1}_{i_4}]]].$$
To any formal commutators we can associate its build-word and its group-word.
The build-word of a commutator $c$ is the word over $S$ that list the entries of 
$c$ in order after one removes brackets and $\pm 1$.  So, the build-word of 
$c=[[s_{i_1}^{\pm 1},s^{\pm 1}_{i_2}],[s^{\pm 1}_{i_3},s^{\pm 1}_{i_4}]]$
is $s_{i_1}s_{i_2}s_{i_3}s_{i_4}$. The group word is the word on $S^{\pm1}$
obtained by applying repeatedly the group rules 
$$[c_1,c_2]^{-1}=[c_2,c_1] \mbox{ and }[c_1,c_2]=c_1^{-1}c_2^{-1}c_1c_2.$$
So the group-word of 
$c=[[s_i,s^{- 1}_{j}],s_{\ell}]$ is  $s_js_i^{-1}s_j^{-1}s_is_\ell^{-1}s_i^{-1}s_js_i s_j^{-1} s_\ell$. 

\begin{defin}[Power weight systems] \label{def-pow}
Given a $k$-tuple $(s_1,\dots,s_k)$
of formal letters and a $k$-tuple $(w_1,\dots, w_k)$ of positive reals, define 
the weight system $\mathfrak w$ on $\mathfrak C(S)$ by setting (inductively)
$$ w(c)= w(c_1)+w(c_2) \mbox{ if }  c=[c_1,c_2].$$
Let
$$\bar{w}_1<\bar{w}_2<\dots <\bar{w}_j<\cdots$$
be the increasing sequence of the weight values of the weight 
system $\mathfrak w$. For $j=1,2,\dots$, 
let $\mathfrak C^\mathfrak w_j$  be the set of all commutators $c$ with $w(c)\ge \bar{w}_j$.
\end{defin}
Clearly, the weight of a formal commutator is the sum of the weights of 
the letters appearing in its build-word. If $S=(s_1,s_2)$ and $w_1=3, w_2=13/2$, 
then the  weight-value sequence is 
$$\bar{w}_1= 3,\bar{w}_2=6,\bar{w}_3=13/2,\bar{w}_4=9,\bar{w}_5=12,\bar{w}_6= 
25/2,\bar{w}_7=13,\dots$$

Given a group $G$ generated by a $k$-tuple $S=(s_1,\dots,s_k)$, any finite
word $\omega$ on the alphabet $S^{\pm 1}$ has a well defined 
image $\pi_G(\omega)$ in $G$. Similarly, any formal
commutator $c$ on the alphabet $S^{\pm 1}$ has an image in $G$ given 
by its group-word representation.

\begin{defin}[Group filtration associated to $\mathfrak w$] \label{def-Gw}
Let $G$ be a nilpotent group equipped with
a generating $k$-tuple  $S=(s_1,\dots,s_k)$ and a weight system $\mathfrak w$
generated by 
$(w_1,\dots,w_k)\in (0,\infty)^k$.  Set 
$$G^\mathfrak w_j= \langle \mathfrak C^\mathfrak w_j\rangle .$$
That is, $G^\mathfrak w_j$ is the subgroup of $G$ generated by the images 
of all formal commutators of weight greater or equal to $\bar{w}_j$.
Let $j_*=j_*(G,S,\mathfrak w)$ be the smallest integer such that 
$G^\mathfrak w_{j_*+1}=\{e\}$.
\end{defin}

\begin{exa} \label{basic-exa}
Let $G$ be the discrete Heisenberg group
$$G=\left\{ \left(\begin{array}{ccc} 1& x & z\\
0&1& y\\
0&0&1\end{array}\right): x,y,z\in \mathbb Z\right\}.$$
Let 
$$s_1=X=\left(\begin{array}{ccc} 1& 1 & 0\\
0&1& 0\\
0&0&1\end{array}\right),\;
s_2=Y=\left(\begin{array}{ccc} 1& 1 & 0\\
0&1& 0\\
0&1&1\end{array}\right),\; s_3=Z^5=\left(\begin{array}{ccc} 1& 0 & 5\\
0&1& 0\\
0&0&1\end{array}\right),$$
and $$w_1=1,\;w_2=3/2,\; w_3=3.$$   In this case, the increasing 
sequence $\bar{w}_j$ is given by $\bar{w}_1=1,\bar{w}_2=3/2,\bar{w}_3=2,\bar{w}_4=5/2,\bar{w}_5=3,\bar{w}_6=7/2,\dots$ and we have
$$G^{\mathfrak w}_6=\{e\},\; G^{\mathfrak w}_5=\{s_3^k: k\in \mathbb Z\},\; 
G^{\mathfrak w}_4= G^{\mathfrak w}_3 = 
\left\{ \left(\begin{array}{ccc} 1& 0 & z\\
0&1& 0\\
0&0&1\end{array}\right): z\in \mathbb Z\right\},$$
$$ G^{\mathfrak w}_2 =
\left\{ \left(\begin{array}{ccc} 1& 0 & z\\
0&1& y\\
0&0&1\end{array}\right): y,z\in \mathbb Z\right\},\;\;
G^{\mathfrak w}_1= G.$$
\end{exa}

\begin{pro}\label{pro-Gw1}
Referring to the setting and notation of {\em Definition \ref{def-Gw}},
for all $j=1,2\dots$, we have $G^\mathfrak w_j\subset G^\mathfrak w_{j+1}$
and $[G,G^\mathfrak w_j]\subset G^\mathfrak w_{j+1}$. In particular,
$$G=G^\mathfrak w_1\supseteq G^\mathfrak w_2\supseteq \cdots \supseteq 
G^\mathfrak w_j\supseteq \cdots\supseteq G^\mathfrak w_{j_*}
\supset G^\mathfrak w_{j_*+1}=\{e\}$$
is a descending normal series 
with $[G^\mathfrak w_j,G^\mathfrak w_j]\subset G^\mathfrak w_{j+1}$.
\end{pro}
\begin{proof}
Recall that if  $X,Y$ are subsets of $G$, $[X,Y]$ 
denotes the subgroup generated by $\{[x,y]: x\in X,y\in Y\}$. 
Recall further that
$$[<X>,<Y>]= [X,Y]^{<X><Y>}$$ 
where the right-hand side is the group generated by all conjugates of $[X,Y]$ 
by elements of the form $g=xy$, $x\in <X>, y\in <Y>$.
Since  $[f_1,f_j]\in\mathfrak C^\mathfrak w_{j+1}$ for all 
$f_1\in \mathfrak C^\mathfrak w_1$, 
$f_j\in \mathfrak C^\mathfrak w_j$ and 
$$[G,G^\mathfrak w_j]= [\mathfrak C^\mathfrak w_1,\mathfrak C^\mathfrak w_j]^G$$
it follows that
$$[G,G^\mathfrak w_j] \subset  (G^\mathfrak w_{j+1})^G$$
Thus a descending induction on $j$ shows  that the groups $G^\mathfrak w_j$ 
are all normal subgroups of $G$ and that 
$$[G,G^\mathfrak w _j]\subset G^\mathfrak w_{j+1}.$$
Note that it may happen that $G^\mathfrak w_j=G^\mathfrak w_{j+1}$ 
for some values of $j$,
$1<j<j_*$. For instance, it may happen that all formal commutators of a 
certain weight are trivial in $G$.  In Example \ref{basic-exa}, $G^{\mathfrak 
w}_3=G^{\mathfrak w}_4$ 
because all commutators of weight $\bar{w}_3=2$ are obviously trivial. 
\end{proof}

\begin{defin}Referring to the setting and notation of Definition \ref{def-Gw},
let
$$R^\mathfrak w_j= \mbox{rank}(G^\mathfrak w_j/G^\mathfrak w_{j+1})$$
be the torsion free rank of the abelian group $G^\mathfrak w_j/G^\mathfrak w_{j+1}$.
\end{defin}
By construction, the images of the formal commutators of weight $\bar{w}_j$ 
form a generating subset of $G^\mathfrak w_j/G^\mathfrak w_{j+1},\;\; j=1,2,\dots,j_*.$ By definition, the torsion free rank of this abelian group 
is the minimal number of elements needed to generates 
$G^\mathfrak w_j/G^\mathfrak w_{j+1}$ modulo torsion. 

\begin{defin}\label{def-DSa}
Referring to the setup and notation of Definition \ref{def-Gw}, set
$$D(S,\mathfrak w)= \sum_1^{j_*} \bar{w}_j 
\mbox{ rank}(G^\mathfrak w_j/G^\mathfrak w_{j+1})
.$$
\end{defin}
Note that  $D(S,\mathfrak w)$ 
depends on the weights values $\bar{w}_j$ 
as well as on the algebraic relations between elements of $S$ in $G$ 
(via the rank of the group $G^\mathfrak w_j$).\vspace{.08in}

\noindent{\bf Example 1.1}(continued)
In Example \ref{basic-exa},
we have $j_*=5$, 
$$G^{\mathfrak w}_5/G^{\mathfrak w}_6= \mathbb Z,\;
G^{\mathfrak w}_4/G^{\mathfrak w}_5=\mathbb Z/5\mathbb Z,\;
G^{\mathfrak w}_3/G^{\mathfrak w}_4=\{0\},
\;G^{\mathfrak w}_2/G^{\mathfrak w}_3=\mathbb Z \mbox{ and }
G^{\mathfrak w}_1/G^{\mathfrak w}_2=\mathbb Z.$$ Hence
$\mbox{rank}(G^{\mathfrak w}_5/G^{\mathfrak w}_6)=1$, 
$\mbox{rank}(G^{\mathfrak w}_4/G^{\mathfrak w}_5)=
\mbox{rank}(G^{\mathfrak w}_3/G^{\mathfrak w}_4)=0,$
$\mbox{rank}(G^{\mathfrak w}_2/G^{\mathfrak w}_3)=\mbox{rank}(
G^{\mathfrak w}_1/G^{\mathfrak w}_2)=1$ and 
$D(S,\mathfrak w)=1+ 3/2+3=11/2$ since 
$\bar{w}_1=1,\bar{w}_2=3/2,\bar{w}_3=2,\bar{w}_4=5/2,\bar{w}_5=3,
\bar{w}_6=7/2,\dots$.

\begin{exa} Assume that the weight $w_i$ are all equal, namely, 
 $w_i=v$, $i=1,\dots,k$.
Then the weight-value sequence is  given by $\bar{w}_j=jv$ and $j_*$
is equal to the nilpotency class of $G$. In this case, the
descending normal series $G^\mathfrak w_j$ is the lower central series defined 
inductively by $G_1=G$, $G_j=[G,G_{j-1}]$, $j\ge 2$, and 
$D(S,\mathfrak w)=vD(G)$ where 
\begin{equation}\label{def-DG}
D(G)=\sum_1^{j_*} j \mbox{ rank}(G_j/G_{j+1}).
\end{equation}
\end{exa}

\begin{theo}\label{th-main2} 
Let $G$ be a
nilpotent group equipped with a generating $k$-tuple 
$S=(s_i)_1^k$ and $a=(\alpha_i)_1^k\in (0,\infty]^k$. Assume
that the subgroup generated by $\{s_i: \alpha_i<2\}$ 
is of finite index in $G$.
Consider the weight system  $\mathfrak w(a)=\mathfrak w$ induced by setting 
$w_i=1/\widetilde{\alpha}_i$ where $\widetilde{\alpha}=\min\{2,\alpha\}$.
Then 
$$\mu_{S,a}^{(n)}(e)\simeq n ^{-D(S,\mathfrak w)}$$
with $D(S,\mathfrak w)$ as in {\em Definition \ref{def-DSa}}.
\end{theo}

\begin{exa} Let $G$ be the discrete Heisenberg group equipped with the 
generating triple $S=(s_i)_1^3$ has in Example \ref{basic-exa}. Let 
$a=(\alpha_i)_1^3$.  In this case, the condition 
that $\{s_i: \alpha_i<2\}$  generates a subgroup of finite index 
is equivalent to $\alpha_1,\alpha_2\in (0,2)$. 
Let $\mathfrak w$  be as defined  in Theorem \ref{th-main2}. Then
$$D(S,\mathfrak w)=
\frac{1}{\alpha_1}+\frac{1}{\alpha_2} +
\max\left\{ \frac{1}{\alpha_1}+\frac{1}{\alpha_2}, \frac{1}{\alpha_3}
\right\}.$$
\end{exa}

\subsection{Some background on random walks} 
Given a finite symmetric generating set $A$, we set $|x|_A=\inf \{k: x\in A^k\}$
(since $A^0=\{e\}$, by convention, $|e|=0$). 
This is called the word-length of $x$ (w.r.t.\ the generating set $A$).
With some abuse of notation, if $S=(s_1,\dots,s_k)$ is a generating $k$-tuple,
we write $|\cdot|_S$ for the word-length
associated with the symmetric generating set $\{s_i^{\pm 1}, 1\le i\le k\}$. 
The volume growth of $G$ (with respect to $A$) is the function $V_A(r)=\#\{g: |g|_A\le r\}$. The $\simeq$-equivalence class of the function $V_A$ is independent 
of the choice of $A$. It is a group invariant called the growth function of $G$.

We say that a probability measure $\phi$ is symmetric if  
$\check{\phi}=\phi$ where $\check {\phi}(x)=\phi(x^{-1})$, $x\in G$.
The Dirichlet form associated with $\phi$ is the quadratic form
$$\mathcal E_\phi(f,f)=\frac{1}{2}\sum_{x,y\in G}|f(xy)-f(x)|^2\phi(y).$$
This form is fundamental in the study of random walks because of the following  
basic result.   
\begin{theo}\label{th-PSC1}[\cite{PSCstab}] 
Assume that $\phi,\psi$ are two symmetric probability measures
on a countable group $G$. If $\mathcal E_\phi\le C \mathcal E_\psi$ then
$$ \psi^{(2kn)}(e)\le 2\phi^{(2n)}(e)+ 2e^{-2kn},  \;k=[C]+2.$$
\end{theo}
This theorem will be use extensively in the present paper. In
\cite{PSCstab}, it is  used to prove that the long time asymptotic behavior
of the probability of return is roughly the same for all random walks 
driven by  symmetric measures with generating support and finite second moment. 
\begin{theo}[\cite{PSCstab}] \label{stab}
Assume that $\phi$ is a symmetric probability measure
on a finitely generated group $G$ with finite symmetric generating set $A$.
Let $u_A$ be the uniform probability measure on $A$. If $\phi$ satisfies
\begin{equation}\label{2m}
\sum_{g\in G}|g|_A^2\phi(g)<\infty
\end{equation}
then there are constants $k, C$ such that
$$u_A^{(2kn)}(e)\le C \phi^{(2n)}(e).$$
Further, if $\phi$  satisfies {\em (\ref{2m})} 
and $\phi>0$ on a finite generating set then
$$\phi^{(2n)}(e)\simeq u_A^{(2n)}(e).$$
\end{theo}
This theorem implies that, if $A$ and $B$ are two symmetric finite 
generating sets
of the group $G$, we have $u^{(2n)}_A(e)\simeq u_B^{(2n)}(e)$. 
Further, for any symmetric $\phi$ with 
finite second moment and generating support, 
$\phi^{(2n)}(e)\simeq u_A^{(n)}(e)$. In this sense, the equivalence class of 
the function $n\mapsto u_A^{(2n)}(e)$ under 
the equivalence relation $\simeq$ is a group invariant. 
This group invariant, which we denote by $\Phi_G$, i.e.,
\begin{equation}\label{Phi}
\Phi_G(n)\simeq u_A^{(2n)}(e),\end{equation}
 has been studied extensively (\cite{PSCstab} shows that $\Phi_G$ 
is invariant under quasi-isometries). In particular,  
$$\Phi_G(n)\simeq \left\{\begin{array}{ll} n^{-D/2} &\mbox{if } G \mbox{ has volume growth } V(r)\simeq r^D,\\
\exp(-n^{1/3}) &\mbox{if } G \mbox{ is polycyclic with exponential volume growth},\\
\exp(-n) &\mbox{if } G \mbox{ is non-amenable.}\end{array}\right.$$
Nilpotent groups belong to the first category and have 
$D=D(G)$ given explicitly by (\ref{def-DG}).
Many other behaviors beyond the three mentioned above are known to occurs
and their are many groups for which
$\Phi_G$ is unknown. See, e.g., \cite{SCnotices,SCsurveydiff} and the references 
therein.

To explain how  Theorem \ref{stab} applies to 
the measures $\mu_{S,a}$ defined at (\ref{muS}), 
we need the following definition.
\begin{defin} \label{def-comlength}
Let $G$ be a nilpotent group with 
descending lower central series $G_j$. The commutator length
$\ell(g)$ of an element $g$ of $G$ is the supremum of 
the integers $\ell$ such that $g^m\in G_\ell$ for some integer $m$.
In particular, by definition,  torsion elements have infinite commutator length.
\end{defin}

\begin{cor} \label{cor-PhiG}
On any finitely generated group $G$ equipped 
with a generating $k$-tuple $S$, we have
$$\mu_{S,a}^{(n)}(e)\simeq \Phi_G(n)\simeq n^{-D(G)/2}$$
for all $k$-tuple $a=(\alpha_1,\dots,\alpha_k)$ such that 
$\alpha_i \ell(s_i)>2$ for all $i=1,\dots, k.$  
\end{cor}
\begin{proof} 
It is well known that for any fixed $g\in G$, we have
$|g^n|_S \simeq n^{1/\ell(g)}$ 
(see also Proposition \ref{pro-gn} where a more general version of this fact
is proved).  It follows that, as long as the $k$-tuple $a$
satisfies the condition stated in the corollary, $\mu_{S,a}$ 
has finite second moment. Hence, Theorem \ref{stab}
implies $\mu_{S,a}^{(n)}(e)\simeq \Phi_G(n)$ as desired.
\end{proof}
As a consequence of the more detailed results proved in this paper, 
we can state the following complementary result.
\begin{theo}\label{th-faster}  
Let $G$ be a nilpotent group equipped 
with a generating $k$-tuple $S$. Let $a\in (0,\infty]^k$. 
If there exists $i\in \{1,\dots,k\}$ such that 
$(\alpha_i,\ell(s_i))=(2,1)$ or  $\alpha_i \ell(s_i)< 2$
then we have  
\begin{equation}\label{faster}
 \lim_{n\ra \infty} [n^{D(G)/2}\mu_{S,a}^{(n)}(e)]=0.
\end{equation}
\end{theo}
Regarding (\ref{faster}), we conjecture but are not able to prove 
that the sufficient condition
provided by Theorem \ref{th-faster} is also necessary. 
See Theorems \ref{th-low3}--\ref{th-low32}.

\subsection{Radial stable laws}

Let $G$ be a finitely generated group with symmetric finite generating set $A$.
Set $B_m=\{g:|g|_A\le m\}$.
Define the radially symmetric ``stable law'' on $G$ with index $\alpha\in (0,2)$
to be probability measure
$$\mu_\alpha(g)= c_\alpha\sum_{m=0}^\infty (1+m)^{-\alpha-1}
\frac{\mathbf 1_{B_m}(g)}{V_A(m)}, \;\;c_\alpha^{-1}= 
\sum_0^\infty (1+m)^{-\alpha-1}
$$
Note that $\mu_\alpha$ is well defined for all $\alpha>0$ and  that
$$\forall\,0<\beta<\alpha<\infty,\;\;\sum_g|g|_A^\beta\mu_\alpha(g)<\infty .$$
It is observed in \cite{SCconv,SClisb,Varconv} that  
$$\forall\,n,\;\;V_A(n)\ge cn^D  \Longrightarrow 
\forall\,n,\;\;\mu_\alpha^{(n)}(e)\le Cn^{-D/\alpha}.$$
In addition, by \cite{HSC,BSClmrw}, for a given group $G$ and for some/any  
$\alpha\neq 2$,
\begin{equation}\label{polequ}
V_A(n)\simeq  cn^D  \Longleftrightarrow 
\mu_\alpha^{(n)}(e)\simeq Cn^{-D/\tilde{\alpha}}, 
\;\;\tilde{\alpha}=\min\{2,\alpha\}.\end{equation}
In fact, if we assume that the group $G$ has polynomial 
volume growth $V(n)\simeq n^D$ then
 $$\mu_\alpha(g) \simeq (1+ |g|_A)^{-D-\alpha}.$$
Further, it follows from \cite{HSC} 
that, for any $\alpha\in (0,2)$, 
there are constants $c_1(\alpha),c_2(\alpha)$ 
such that
$$  c_1(\alpha)\mu_\alpha\le \nu_\alpha\le c_2(\alpha)\mu_\alpha$$
where $\nu_\alpha$ denotes the measure that is $\alpha$-subordinated to $u_A$
in the sense of (\cite{BSCsubord}), that is,
$$\nu_\alpha = 
\sum_1^\infty \frac{\Gamma(n-\alpha)}{\Gamma(1-\alpha)\Gamma(n+1)}u_A^{(n)}.$$ 
Moreover, for any $\alpha\in (0,2)$,
$$\forall\,n\in \mathbb N,\;\;\;\mu_\alpha^{(n)}(e)\simeq \nu_\alpha^{(n)}(e)\simeq n^{-D/\alpha}.$$
In \cite{SCZ0}, motivated by applications given below, the authors
prove the following complementary statement regarding the behavior of $\mu_2$.
\begin{pro}[{\cite{SCZ0}}] 
Assume that $G$ has polynomial volume growth $V_S(n)\simeq n^D$.
Then we have
$$\mu_2^{(n)}(e)\simeq  (n \log n)^{-D/2} .$$
\end{pro}

The lower bounds on $\mu_{S,a}^{(n)}(e)$ obtained in this paper are proved by 
establishing Dirichlet form comparisons involving appropriate 
generalization of the above radially symmetric stable measures and 
using Theorem \ref{th-PSC1}.

\subsection{Background on nilpotent groups}
The classical setting 
for the study of random walks is the lattice $\mathbb Z^d$. 
See \cite{Spitzer}.  Since this work is concerned with random walks on 
nilpotent groups, we briefly discuss some of the similarities and 
differences between
the lattice $\mathbb Z^d$ and  finitely generated nilpotent groups. We also 
describe three basic examples.

The most fundamental similarity between a finitely generated nilpotent group
$G$ and the lattice $\mathbb Z^d$ is that,  assuming that $G$ is torsion free,
there exists a real nilpotent Lie group $\mathbb G$ such that $G$ can be 
identified with a discrete subgroup of $\mathbb G$ with compact quotient 
$\mathbb G/G$. In other words, $G$ is a (co-compact) lattice in $\mathbb G$ in 
exactly the same way that $\mathbb Z^d$ is a lattice in 
$\mathbb R^d$ (except that the 
quotient is not a group, in general).
This is a fundamental result of Malcev. See, e.g., 
Philip Hall famous notes \cite{PHall}. 
However, simply connected real nilpotent Lie groups and their lattices 
are classified only in very small dimensions. See \cite{deGraaf}. 
For instance, there are essentially $5$ distinct ``irreducible''
simply connected real nilpotent Lie groups of dimension 5. 
In dimension 6, there are 34.
No one knows the list of all simply connected nilpotent real Lie groups
of dimension 9, let alone higher dimensions. 

From a technical viewpoint, the study of random walks 
on abelian groups is mostly based on the use of the  Fourier 
transform (see \cite{Spitzer}). Although the representation theory of 
(real) nilpotent Lie groups is well developed, 
it has proved very hard to use this theory to 
study random walks (except in some very particular cases). For these reasons, 
the study of random walks on nilpotent groups is often based on techniques that 
are rather different from the classical techniques used in the abelian case.
This is certainly the case for the present work.

\begin{exa} Let $U(d)$ be the group 
of all upper triangular $d\times d$ matrices over 
$\mathbb Z$ with diagonal entries  
equal to $1$. This group is a lattice in the nilpotent real Lie 
group $\mathbb U(d)$ of all  upper triangular $d\times d$ matrices over 
the reals with diagonal entries equal to $1$. Let $E_{i,j}$, 
$1\le i<j\le d$, be the matrix
in $\mathbb U(d)$ with all non-diagonal entries equal to $0$ 
except for the entry 
in the $i$-th row and $j$-th column which equals $1$.  
These elements are  related by
$E_{i,j}E_{\ell,m}= \delta_{j,\ell}E_{i,m}$.
Further,
$$E_{i,j}=[E_{i,i+1},[E_{i+1,i+2},\dots,[E_{j-2,j-1},E_{j-1,j}]\cdots]].$$
In particular, the $(d-1)$-tuple
$S=(E_{i,i+1})_1^{d-1}$ is generating.
For any $m=1,\dots, d-1$, the elements $\{E_{i,i+m}: 1\le i\le d-m\}$ 
can be expressed as 
commutators of length $m$ on $S^{\pm 1}$ and form a minimal generating set
for the subgroup $U(d)_m=[U(d),U(d)_{m-1}]$ in the lower central 
series of $U(d)$.  The nilpotency class of $U(d)$ is $d-1$, that is, 
any commutator of length greater than $d-1$ equals the identity in $U(d)$.

Any matrix $M=(m_{i,j})$ in $U(d)$ can (obviously) be written uniquely
(order matters!)
$$M= \prod_{k=1}^{d-1}\left(\prod_{i=0}^{k-1}E_{k-i,d-i} ^{m_{k-i,d-i} }\right)
$$where the $m_{i,j}$ are simply the entry of the matrix $M$. 
Much less trivially,
there is also a unique expression of the form
$$M= \prod_{k=1}^{d-1}\left(\prod_{i=k}^{d-1}E_{i-k+1,i+1} ^{m'_{i-k+1,i+1} }\right)
$$
where  $(m'_{i,j})_{1\le i<j\le n}$ is obtained from 
$(m_{i,j})_{1\le i<j\le n}$ by a polynomial bijective transformation with 
polynomial inverse.

Since $A=\{E^{\pm 1}_{i,i+1}, 1\le i\le d-1\}$ generates $U(d)$, it is of great 
interest to describe the word length $|M|_A$ of a matrix $M\in U(d)$ in terms 
of the  coordinate systems $(m_{i,j})_{1\le i<j\le d}$ 
and $(m'_{i,j})_{1\le i<j\le d}$. The answer is essentially the same in 
both cases, namely, 
$$|M|_A \simeq \sum_{1\le i<j\le d}|m_{i,j}|^{1/|j-i|}\simeq 
\sum_{1\le i<j\le d}|m'_{i,j}|^{1/|j-i|}.$$
This well known (but non-trivial) result is the key to the volume 
growth estimate
$$V_{U(d),A}(r) \simeq r^{D(U(d))},\;\;D(U(d))=\sum_{i=1}^{d-1} i(d-i)$$
and to the assorted random walk result (see, e.g.,  \cite{VSCC})
$\Phi_{U(d)}(n)\simeq n^{-D(U(d))/2}.$
If we set $S=(s_i=E_{i,i+1})_1^{d-1}$ then for any
 $a=(\alpha_i)_1^{d-1}\in (0,2)^{d-1}$ our main result yields
$$\mu_{S,a}^{(n)}(e)\simeq n^{-D} ,\;\;D=\sum_{1\le i<j\le d}\sum_{m=i}^{j-1}
\frac{1}{\alpha_m}.$$
\end{exa}

\begin{exa} The free nilpotent group of nilpotency class 
$\ell$ on $k$ generators , $N(k,\ell)$, can be defined as the 
quotient of the free group on $k$ generators by the normal subgroup
generated by the images of all formal commutators of length greater 
than $\ell$. This group has the (universal) property that it covers 
any $k$ generated nilpotent group $G$ of nilpotency class $\ell$ 
with a covering homomorphism sending the canonical generating $k$-tuple 
of $N(k,\ell)$ to the given generating $k$-tuple of $G$.

Marshal Hall gave a description of $N(k,\ell)$ in terms of the 
so-called ``basic commutators''. See \cite[Chapter 11]{MHall}. 
Let $(s_1,\dots, s_k)$ 
be the canonical generators of $N(k,\ell)$. Define the ordered set of 
all basic commutators $c_1<\dots < c_t$ using the following 
inductive procedure. 

(1) $s_1,\dots,s_k$ are the basic commutators of 
length $1$ and, by definition $s_1<s_2<\dots<s_k$; (2) for each $m$
the basic commutators of length $m$ are all commutators of the form 
$c=[c',c'']$ with $c',c''$ basic commutators of length $m',m''$ with $m'+m''=m$ 
such that  $c'>c''$ and, if $c'=[d',d'']$ ($d,d'$ basic commutators) then 
$c''\ge d''$; (3) commutators of length $m$ come after commutators of length 
$m-1$ and are ordered arbitrary with respect to each other.
By a theorem of Witt (e.g., \cite[Theorem 11.2.2]{MHall}), 
the number of basic commutators of length $m$
on $k$ generators is
$M_k(m)=m^{-1}\sum_{d|m}\mu(d)k^{m/d}$ where $\mu$ denotes the classical 
M\"obius function.  Marshall Hall proved that the basic commutators of length $m$
form a basis of the abelian group $N(k,\ell)_m/N(k,\ell)_{m+1}$ for 
$1\le m\le \ell$ and that
any element $g$ of $N(k,\ell)$ 
can be written uniquely
$$g=\prod_1^t c_i^{x_i},\;\;x_i\in \mathbb Z.$$
Moreover, the length of $g$ with respect to the generating set 
$A=\{s_i^{\pm 1}\}$ satisfies $|g|_A\simeq \sum_1^t |x_i|^{1/m_i}$
where $m_i$ is the commutator length of $c_i$.
This gives the volume group estimate
$$V_A(r)\simeq r^{D(N(k,\ell))},\;\;D(N(k,\ell))=\sum _{m=1}^\ell m M_k(m)
=\sum_{m=1}^\ell \sum_{d|m}\mu(d)k^{m/d} $$
and the assorted random walk estimate
$\Phi_{N(k,\ell)}(n)\simeq n^{-D(N(k,\ell))/2}.$

In this case, the main result of the present work,
together with Witt's theorem (e.g., \cite[Theorem 11.2.2]{MHall}), gives 
that for any 
$k$-tuple $a=(\alpha_i)_1^k\in (0,2)^k$, we have
$$\mu_{S,a}^{(n)}(e)\simeq n^{-D}$$
where $$D= \sum_{m=1}^\ell\sum_{(m_1,\dots,m_k)\vdash m}
\frac{1}{m}\left(\sum_1^k\frac{m_i}{\alpha_i}\right)
\sum_{d|m_1,\dots,m_k}\mu(d) \binom{m/d}{m_1/d,\cdots,m_k/d}.$$
\end{exa}

\begin{exa} Let $G$ be the group 
$$G=\langle u_1,\dots, u_\ell,t|
[u_i,u_{j}]=1; [u_i,t]=u_{i+1},i<\ell;
[u_\ell,t]=1\rangle$$  
defined by generators and relations. This group is nilpotent of nilpotency 
class $\ell$ and it is generated by $S=(s_1=u_1,s_2=t)$ with $G_m$ generated by
$\{u_i: i\ge m\}$. In this case, we have $\Phi_G(n)\simeq n^{-D(G)/2}$
with $D(G)= 1+\ell(\ell+1)/2$.  
If we let $a=(\alpha_1,\alpha_2)\in (0,2)^2$, our main result yields $\mu_{S,a}^{(n)}(e)\simeq n^{-D}$ with 
$$D=\frac{\ell}{\alpha_1}+\frac{1+ (\ell-1)\ell/2}{\alpha_2}
.$$

In any of the above examples, we can also consider other choices 
of generating tuples. For instance, in the current example, we can fix 
$j\in \{1,\dots,\ell-1\}$ and consider the generating $3$-tuple 
$S_j=(s_1=u_1,s_2=t,s_3=u_{j+1})$ with $a'=(\alpha_1',\alpha_2',\alpha_3')
\in (0,2)^3$. In this case, our main result yields
$\mu_{S_j,a'}^{(n)}(e)\simeq n^{-D}$ with 
$$D=\left\{\begin{array}{cc} \frac{\ell}{\alpha'_1}+
\frac{1+(\ell-1)\ell/2}{\alpha'_2} 
& \mbox{ if } \frac{1}{\alpha'_3}\le \frac{1}{\alpha'_1}+ 
\frac{j}{\alpha'_2}\\  
\frac{j}{\alpha'_1}+\frac{1+j(j+1)/2}{\alpha'_2} 
+\frac{\ell-j}{\alpha_3'} +\frac{(\ell-j)(\ell-j+1)/2}{\alpha'_2} 
& \mbox{ if } \frac{1}{\alpha'_3}> \frac{1}{\alpha'_1}+ \frac{j}{\alpha'_2}. 
 \end{array}\right. 
$$
\end{exa}

\section{Quasi-norms and approximate coordinates} \setcounter{equation}{0}
\label{sec-norm}
This section describes results of an algebraic and geometric nature that play 
a key role in our study to the random walks driven by the measures $\mu_{S,a}$
defined at (\ref{muS}). One of the basic idea in the study of simple 
random walks on groups (i.e., the collection of random walks driven by 
the uniform probability measures $u_A$ where $A$ is a finite symmetric 
generating set)
is that the  notion of ``volume growth'' of the group leads to
basic upper bounds on $u_A^{(2n)}(e)$: the faster the volume growth, the faster
the decay of the probability of return. In the case of nilpotent group, this
heuristic leads to sharp bounds. Indeed, for any given $D\ge 0$, 
 $V_A(n)\simeq n^ D$ 
if and only if $u_A^{(2n)}(e)\simeq
n^{-D/2}$. See \cite{VSCC}.

The estimates of $\mu_{S,a}^{(n)}(e)$ obtained in this work are 
based on a similar heuristic which requires us to define appropriate 
geometries associated with the different choices of $S$ and $a$. 
This section defines these geometries and develop the needed key results.

\subsection{Weight systems and weight-functions systems}\label{sec-wF}
We refer the reader to subsection \ref{sec-wD} for notation regarding
words and formal commutators over a finite alphabet $S^{\pm 1}$, 
$S=(s_1,\dots,s_k)$.

\begin{defin}[Multidimensional weight system]  \label{def-wd}
Given  a $k$-tuple
$(w_1,\dots, w_k)$ with $w_i\in (0,\infty)\times \mathbb R^{d-1}$, $1\le i\le k$,
let $\mathfrak w$ be the weight system 
$$\mathfrak w:  \mathfrak C(S) \ni c\mapsto w(c)\in (0,\infty)\times 
\mathbb R^{d-1}$$
on the set $\mathfrak C(S)$ of all formal commutators on $S^{\pm 1}$ defined by 
$w(s_i^{\pm 1})= w_i$ and $w(c)=w(c_1)+w(c_2)$ if $c=[c_1,c_2]$.
Let
$$\bar{w}_1<\bar{w}_2<\cdots<\bar{w}_j<\dots$$
be the ordered sequence of the values  $w(c)$ when $c$ runs over all formal 
commutators and
 $(0,\infty)\times \mathbb R^{d-1}$ is given the usual lexicographic order.
\end{defin}
Note that we always have $w([c_1,c_2])> \max\{w(c_1),w(c_2)\}$.

\begin{defin} \label{def-Gw2}
For each $j=1,\dots$, let $\mathfrak C_j(S)$ be the 
set of all formal commutators of weight at least $\bar{w}_j$. 
If $G$ is a group generated by a $k$-tuple $S=(s_1,\dots, s_k)$, let
$G^\mathfrak w_j=\langle \mathfrak C_j(S)\rangle$ be the subgroup of $G$
generated by the image in $G$ of $\mathfrak C_j(S)$. Assuming that 
$G$ is nilpotent, let $j_*=j_*(\mathfrak w)$
be the smallest integer such that $G^\mathfrak w_{j_*+1}=\{e\}$.
\end{defin}
The proof of the following proposition is the same as that of Proposition 
\ref{pro-Gw1}.
\begin{pro}\label{pro-Gw2} Referring to the 
setting and notation of {\em Definition \ref{def-Gw2}}, 
assume that $G$ is nilpotent.
Then, for all $j=1,2\dots$, we have 
$G^\mathfrak w_j\subset G^\mathfrak w_{j+1}$
and $[G,G^\mathfrak w]\subset G^\mathfrak w_{j+1}$. In particular,
$$G=G^\mathfrak w_1\supseteq G^\mathfrak w_2\supseteq \cdots \supseteq 
G^\mathfrak w_j\supseteq \cdots\supseteq G^\mathfrak w_{j_*}
\supset G^\mathfrak w_{j_*+1}=\{e\}$$
is a descending normal series 
with $[G^\mathfrak w_j,G^\mathfrak w_j]\subset G^\mathfrak w_{j+1}$.
We let 
$R^\mathfrak w_j$ be the torsion free rank of the abelian 
group $G^\mathfrak w_j/G^\mathfrak w_{j+1}$.
\end{pro}

\begin{defin}[Weight-function system] \label{def-F}
Given increasing 
functions 
$$F_i:[1,\infty)\rightarrow [1,\infty),$$ we define the weight-function 
system $\mathfrak F$ to be the collection of functions
$$F_c: [1,\infty)\rightarrow [1,\infty), \;c\in \mathfrak C(S),$$   
by setting inductively $F_{s^{\pm 1}_i}=F_i$, $1\le i\le k$, and 
$F_c=F_{c_1}F_{c_2}$ if $c=[c_1,c_2]$.
\end{defin}

\begin{rem} According to Definitions \ref{def-wd}-\ref{def-F}, 
if the build-sequence of the commutators $c$ of length $\ell $  is $(u_1,\dots,u_\ell)\in S^\ell$
then
$$w(c)= \sum_1^\ell w_i,\;\;F_c(r)=\prod_1^\ell F_i(r).$$
\end{rem}
\begin{rem} A key collection of examples of weight systems are the 
(one-dimensional) power-weight systems introduced in \ref{def-pow} 
where $w_i\in (0,\infty)$. Such a weight system is naturally associated with the 
weight-function system of power functions where $F_i(r)=r^{w_i}$.
In the context of the study of the 
 random walks driven by the measures $\mu_{S,a}$, these power weight systems 
and associated power function systems are relevant to the case when 
$a=(\alpha_i)_1^k\in (0,2)^k$. 
\end{rem}

\begin{exa} In order to study the measures $\mu_{S,a}$ with tuples $a$ with 
$\alpha_j =2$ for some $j$, 
it is necessary to introduce weight functions of the type  $r^2\log r$.
To allow for such functions, one can consider the two-dimensional weight systems
build on 
$$w_i=(u_i,v_i) \mbox{ with } u_i>0 \mbox{ and }
v_i\in \mathbb R, 1\le i\le k.$$ 
In this case a natural compatible  weight-function system  would be 
$$F_i(r)= r^{u_i} [\log (e+r)]^{v_i}, 1\le i\le k.$$
\end{exa} 

\begin{exa}
When dealing with more general measures than $\mu_{S,a}$,
it makes sense to consider multiparameter weight 
functions such that
 $$f_{v_1,v_2,v_3}(r)=r^{v_1}[\log (e+r)]^{v_2}[\log(e+\log (e+r))]^{v_3}
, \;\;v_1\in (0,\infty),\;\; 
v_2,v_3\in \mathbb R,$$
together with the natural associated lexicographical order on the parameter 
space $(v_1,v_2,v_3)$.
\end{exa}

In what follows we will mostly use weight-function systems $\mathfrak F$
such that
\begin{equation}\label{F1}
\exists\, C\ge 1,\forall i\in \{1,\dots,k\},\;\forall\,r\ge 1,\;\;
2F_i(r)\le F_i(Cr),\; F(2r)\le CF(r).
\end{equation}
Further, we will often make the assumption that we are given a weight system 
$\mathfrak w$ 
and a weight-function system $\mathfrak F$ that are compatible in the sense that
\begin{equation}\label{F2} \exists\,C\ge 1,\;\forall\,
c,c',\;\;
w(c)\preceq w(c') \Longleftrightarrow  \forall r,\;\;F_{c}(r)\le C F_{c'}(r).
\end{equation}
Note that under these two hypotheses, $w(c)=w(c')$ is equivalent to
 $F_c\simeq F_{c'}$. In this case, except for notational convenience, 
it is obviously somewhat redundant to use 
both $\mathfrak w $ and $\mathfrak F$ since they contain more or less  
the same information.

\begin{defin}\label{defin-bfF}
Referring to the setting and notation introduced above, assume that the 
weight-function system $\mathfrak F$ and the weight system $\mathfrak w$
satisfy (\ref{F1})-(\ref{F2}). For any $j=1,\dots,j_*$,
let $\mathbf F_j$ be a function such that for any commutator $c$ 
with $w(c)=\bar{w}_j$, we have
$$\mathbf F_j \simeq  F_ c .$$
(The function $\mathbf F_j$ corresponding to commutators $c$ with $w(c)=\bar{w}_j$ should not be confused $F_i=F_{s_i}$).
\end{defin}

In the following definition, given a finite tuple  
$\Sigma$ of elements of a nilpotent group $G$,
we let $\Omega(\Sigma)$  be the set of all finite words with formal letters 
in $\Sigma \cup \Sigma^{-1}$. For $\omega\in \Omega(\Sigma)$, we write 
$\pi(\omega)$ to denote the corresponding element of $G$. 
For $\omega\in \Omega(\Sigma)$ and $\sigma\in \Sigma$,
let $\mbox{deg}_{\sigma}(\omega)$ is the number of occurrences of 
$\sigma^{\pm 1}$ in $\omega$.

\begin{defin} \label{defin-FQN} Let $G$ be a nilpotent group generated by the 
$k$-tuple $S=(s_1,\dots,s_k)$.
Let $\mathfrak w,\mathfrak F$ be a weight system and associated  
weight function system on a generating $k$-tuple $S$ which satisfy 
(\ref{F1})-(\ref{F2}). 
For any tuple $\Sigma$ of elements 
in $\mathfrak C(S)$, set 
$F_{\Sigma}= F_c$ where $w(c)=\min\{w(\sigma): \sigma\in \Sigma\}$.   
For $g\neq e$, set
$$\|g\|_{\Sigma,\mathfrak F}=
\min\{ r\ge 1: 
g=\pi(\omega): \omega\in \Omega(\Sigma),\; 
\mbox{deg}_{c}(\omega)\le  F_{c}\circ F^{-1}_{\Sigma}(r) , c\in \Sigma\}.$$
By convention, $\|e\|_{\Sigma,\mathfrak F}=0$.
Set also
$$Q(\Sigma,\mathfrak F,r)=\{g\in G: F_{\Sigma}^{-1}(\|g\|_{\Sigma,\mathfrak F})
\le r\}.$$
Further, when $S$ and $\mathfrak w,\mathfrak F$ are fixed, set
$$\|g\|_{\mbox{\tiny com}}=\|g\|_{\mathfrak F,\mbox{\tiny com}}=
\|g\|_{\mathfrak C(S),\mathfrak F},\;\; 
\|g\|_{\mbox{\tiny gen}}=\|g\|_{\mathfrak F,\mbox{\tiny gen}}=
\|g\|_{S,\mathfrak F}$$
and 
$$Q_{\mbox{\tiny com}}(r)=Q(\mathfrak C(S),\mathfrak F,r),
\;\; Q_{\mbox{\tiny gen}}(r)=Q(S,\mathfrak F,r).$$
Note that $F_S=F_{\mathfrak C(S)}$.
\end{defin}
\begin{rem}\label{rem-norm}
If $\Sigma$ generates $G$ then $\|\cdot\|_{\Sigma,\mathfrak F}$ 
is a quasi-norm on $G$ (see \ref{radial} below for a precise definition).  
It is a norm on $G$ (i.e., satisfies the triangle 
inequality) if each of
the functions $\{F_{c}\circ F^{-1}_{\Sigma}$ , $c\in \Sigma\}$,
defined on $[1,\infty)$ can be extended to a convex function on 
$[0,\infty)$ that vanishes at $0$.
\end{rem}

\begin{exa}The simplest example is when the weight system $\mathfrak w$ is one 
dimensional, generated by $w(s_i)=w_i\in [2,\infty)$, and the associated 
weight function system $\mathfrak F$ is generated by $F_i(r)=r^{w_i}$.
In this case, it will sometimes be convenient to 
write $\|\cdot\|_{S,\mathfrak w}$
for $\|\cdot\|_{S,\mathfrak F}$ (resp. 
$\|\cdot\|_{\Sigma,\mathfrak w}$ for $\|\cdot\|_{\Sigma,\mathfrak F}$). 
\end{exa}

\begin{exa} \label{exa-illustr}
For further illustration, consider the  groups $\mathbb Z^3$
equipped with its natural generating $3$-tuple $S=(s_i)_1^3$ and the
discrete Heisenberg group (see Example \ref{basic-exa})
equipped with the generating $3$-tuple $S=(s_1=X,s_2=Y,s_3=Z)$ 
where $X$ is the matrix with 
$x=1,y=z=0$ and $Y,Z$ are  defined similarly.  Set $F_1(r)= r^{3/2}$, 
$F_2(r)= r^{2}\log (e+r)$, $F_3(r)= r^{\gamma}$, $\gamma>3/2$,
and let $\mathfrak F$ 
be the associated weight-function system (we let the reader define 
the natural $2$-dimensional weight system $\mathfrak w$ that is compatible with $\mathfrak F$). 

On $\mathbb Z^3$, it is clear from the definition that
$$\|(x,y,x)\|_{\mathfrak F,\mbox{\tiny gen}}\simeq \max\left\{ |x|,
\frac{|y|^{3/4}}{\log(e+|y|)^{3/4}},|z|^{3/(2\gamma)}\right\}.$$   

On the Heisenberg group, it is not immediately obvious how to compute
the $\|\cdot\|_{\mathfrak F,\mbox{\tiny gen}}$-norm 
of the element  $$g_{x,y,z}= \left(\begin{array}{ccc} 1& x & z\\
0&1& y\\
0&0&1\end{array}\right).$$
Theorem \ref{th-coord2F} below 
(and the fact that the matrix representation of $g_{x,y,z}$ is unique) 
leads to the conclusion that
$$\|g_{x,y,z}\|_{\mathfrak F,\mbox{\tiny gen}}\simeq
\max\left\{ |x|,
\frac{|y|^{3/4}}{\log(e+|y|)^{3/4}},|z|^{3/(2\gamma)}\right\}
\mbox{ if }\gamma> 7/2$$ and 
$$\|g_{x,y,z}\|_{\mathfrak F,\mbox{\tiny gen}}\simeq
\max\left\{ |x|,
\frac{|y|^{3/4}}{\log(e+|y|)^{3/4}},\frac{|z|^{3/7}}{[\log(e+|z|)]^{3/7}}\right\}
\mbox{ if } 3/2\le \gamma \le 7/2.$$

One can check (without much trouble) that 
$\|\cdot\|_{\mathfrak F, \mbox{\tiny gen}}$ satisfies the triangle inequality
in this case (on either $\mathbb Z$ or the Heisenberg group). 
We shall see that this choice of weight-function system is relevant to the 
study of the probability measure $\mu$ on $G$ such that
$$ \mu(s_i^{n}) \mbox{ is proportional to }\frac{1}{1+|n|F^{-1}_i(|n|)}, \;n\in \mathbb Z.$$
We will use this example to illustrate some of our main results in the 
rest of the paper. 
\end{exa}

The following theorem contains some of the key geometric results 
we will need to study the walk driven by measures of the type $\mu_{S,a}$.

\begin{theo}[$\mathfrak w$-$F$-adapted coordinates]\label{th-coord2F}
Let $G$ be a nilpotent group equipped with a generating $k$-tuple 
$S=(s_1,\dots,s_k)$. Let $\mathfrak w$, $\mathfrak F$ 
be weight and weight-function systems on $S$ satisfying 
{\em (\ref{F1})-(\ref{F2})}. 

Let $\Sigma=(c_1,\dots,c_t)$ be a tuple of formal commutators in 
$\mathfrak C(S)$ with  non-decreasing weights  $w(c_1)\preceq 
\dots\preceq  w(c_t)$. Let $m_j$, $j=0,\dots, j_*$ be defined by
$$\{c_i: w(c_i)=\bar{w}_j\}
= \{c_i: m_{j-1}<i\le m_j\}.$$
Assume  that (the image of) $\{c_i: w(c_i)=\bar{w}_j\}$
generates $G^\mathfrak w_j$ modulo $G^\mathfrak w_{j+1}$ and
that $ \{c_i: m_{j-1}<i\le m_{j-1} +R^\mathfrak w_j\}$
is free in $G^\mathfrak w_j/G^\mathfrak w_{j+1}$. Then the following 
properties hold:

\begin{itemize}
\item  There exists a constant $C=C(G,S,\mathfrak F)$ such that for
any $r\ge 1$, if $g\in G$ can be expressed as a word $\omega$
over $\mathfrak C(S)$
with $\mbox{\em deg}_{c}(\omega)\le F_{c}(r)$ for all 
$c\in \mathfrak C(S)$ then
$g$ can be expressed in the form
$$g=\prod_{i=1}^t c_i^{x_i}
\mbox{ with }|x_i|\le C\times \left\{\begin{array}{ll}\mathbf F_j(r) 
&\mbox{ if } m_{j-1}+1\le i\le R^\mathfrak w_j\\
1 &\mbox{ if } R^\mathfrak w _j+1\le i\le m_j.\end{array}\right.$$
\item There exist an integer $p=p(G,S,\mathfrak F)$, 
a constant $C=C(G,S,\mathfrak F)$ and a sequence
$(i_1,\dots,i_p)\in \{1,\dots, k\}^p$ such that 
if $g$ can be expressed as a word $\omega$ 
over $\mathfrak C(S)$
with $\mbox{\em deg}_{c}(\omega)\le F_{c}(r)$ for some $r\ge 1$  and all $c\in \mathfrak C(S)$ then
$g$ can be expressed in the form
$$g=\prod_{j=1}^p s_{i_j}^{x_j}
\mbox{ with } |x_j|\le CF_{{i_j}}(r).$$
\end{itemize}
\end{theo}
This important theorem will be proved in the last section of this article.
See also Theorem \ref{th-coord2Fimp} for an additional improvement of the 
the last statement of Theorem \ref{th-coord2F}. Note that in the decomposition 
$g=\prod_{j=1}^p s_{i_j}^{x_j}$, the sequence
$(i_j)_1^p$ is independent of the group element $g$.  

The proof of the following simple corollary is omitted.

\begin{cor} \label{cor-coord2F}
Referring to {\em Definition \ref{defin-FQN}}, the quasi-norms
$\|\cdot\|_{\mbox{\em\tiny com}}$ and 
$\|\cdot\|_{\mbox{\em\tiny gen}}$ defined on $G$ satisfy
$$  \|\cdot\|_{\mbox{\em \tiny gen}} \simeq 
\|\cdot\|_{\mbox{\em \tiny com}} \mbox{ over }  G.$$
Further, referring to the $t$-tuple $\Sigma=(c_1,\dots,c_t)$ of
{\em Theorem \ref{th-coord2F}}, we have
$$F_\Sigma^{-1}(\|\cdot\|_{\Sigma,\mathfrak F}) \simeq 
F^{-1}_{S}(\|\cdot\|_{\mbox{\em \tiny com}}) \mbox{ over }  G.$$
\end{cor}

\begin{rem}In the case when the generators $s_i$ 
are given equal weight-functions, i.e., $F_i=F_j$, $1\le i\le j\le k$, 
the quasi-norms $\|\cdot\|_{S,\mathfrak F}, \|\cdot\|_{\Sigma,\mathfrak F}$ 
and $\|\cdot\|_{\mathfrak C(S),\mathfrak F}$ 
are all comparable to
the usual word-norm  $|\cdot|_S$. 
\end{rem}

\subsection{Norm equivalences}\label{sec-neq}
In this section, we briefly discuss how changing weight functions 
affect the quasi-norms $\|\cdot\|_{\mbox{\tiny com}}$ and 
$\|\cdot\|_{\mbox{\tiny gen}}$ introduced in Definition
\ref{defin-FQN}.

\begin{defin}\label{def-jw} 
Let $G$ be a countable nilpotent group equipped with a 
generating $k$-tuple $S=(s_1,\dots,s_k)$ 
and a (possibly multidimensional) weight system $\mathfrak w$ as above.  
For each $g\in G$, let 
$$j_\mathfrak w(g)=\max \{j:\exists u\in \mathbb N
,\;g^{u}\in G_{j}^{\mathfrak{w}}\}.$$ 
Let $\mbox{core}(\mathfrak w,S)$ be the  sub-sequence of $S$ obtained by 
keeping only those  $s_i$ such that
$w(s_i)=
\overline{w}_{j_\mathfrak w(s)}.$
\end{defin}
By construction,  we always have $w(s)\le \bar{w}_{j_\mathfrak w(s)}$.
Those generators $s\in S$  with $w(s)< \bar{w}_{j(s)}$ are, in some sense, 
inefficient. The following proposition makes this precise and motivates 
this definition.
\begin{pro} \label{pro-core}
Any formal commutator $c\in \mathfrak C(S)$ 
whose image in $G$ is free in 
$G^\mathfrak w_j/G^\mathfrak w_{j+1}$ must only use letters in 
$\mbox{\em core}(\mathfrak w,S)$. In particular,
referring to the sequence of commutators $c_1,\dots, c_t$
in {\em Theorem \ref{th-coord2F}}, any formal commutator $c_i$
with $i\in m_{j-1}+1,\dots,m_{j-1}+R^{\mathfrak w} _j$ 
must only use letters in $\mbox{\em core}(\mathfrak w,S)$.
\end{pro}

\begin{proof} Assume that the image of $c$ is in the torsion free part of 
$G_{j}^{\mathfrak{w}}/G_{j+1}^\mathfrak{w}$ and 
involves  $s\notin \mbox{core}(S)$, say $c=[c',[s,c'']]$.
Then $\exists u\in 
\mathbb{N},\;s^{u}\in G_{j(s)}^{\mathfrak{w}}$ with $\overline{w}_{j(s)}>w(s)$
(where we write $j(s)=j_\mathfrak w(s)$). 
From the linearity of brackets, we have
\[
c^{u}\equiv \lbrack c',[s^{u},c'']]\mbox{ }\mbox{ mod }G_{j+1}^{%
\mathfrak{w}} 
\]%
while $[c',[s^{u},c'']]\in G_{j+1}^{\mathfrak{w}}$ since $s^{u}\in
G_{j(s)}^{\mathfrak{w}}$ with $\overline{w}_{j(s)}>w(s).$ Therefore 
\[
c^{u}\equiv 0\mbox{ }\mbox{ mod }G_{j+1}^{\mathfrak{w}}. 
\]%
This contradicts the assumption that $c$ is free 
in  $G_{j}^{%
\mathfrak{w}}/G_{j+1}^{\mathfrak{w}}$.  The proposition follows.
\end{proof}

\begin{defin} Let $G$ be a countable nilpotent group equipped with a 
generating $k$-tuple $S=(s_1,\dots,s_k)$ 
and a (possibly multidimensional) weight system $\mathfrak w$ as above.  
Let $\Sigma=(c_1,\dots,c_t)$ be a sequence of formal
commutators as in Theorem \ref{th-coord2F}.
Let $\mbox{core}(\mathfrak w,S,\Sigma)$ be the sub-sequence of $S$ 
of those letters 
$s_\delta$ that appear in the build-sequence of 
one or more of the formal commutators $c_i\in \Sigma$ with
$i\in \cup_{j=1}^{q+1} \{m_{j-1}+1,\dots,m_{j-1}+R^{\mathfrak w}_j\}$.   
\end{defin}

\begin{rem}Proposition \ref{pro-core} shows that, for any sequence $\Sigma$
of formal commutators as in Theorem \ref{th-coord2F}, we have
$$\mbox{core}(\mathfrak w,S, \Sigma)\subset 
\mbox{core}(\mathfrak w,S).$$
\end{rem}
In what follows, given two tuples  $S=\{s_1,\dots,s_k)$, 
$\Theta=(\theta_1,\dots,\theta_\kappa)$ of elements of $G$ 
(possibly of different length $k,\kappa$), 
we write $S \subset \Theta$ if there is a 
one to one map $J: \{1,\dots,k\}\ra \{1,\dots,\kappa\}$ such that
$s_{J(i)}=\theta_i$ in $G$.  This applies, for instance, to 
the ``inclusion'' $\mbox{core}(\mathfrak w,S, \Sigma)\subset 
\mbox{core}(\mathfrak w,S)$  in the previous remark. 
Abusing notation, we will sometimes use 
the same letter $s$ 
to denote an element of $S$ and the associated element in $\Theta$.

\begin{pro}\label{pro-gn} Referring to the setting and notation of
{\em Theorem \ref{th-coord2F}},
for each $g\in G$ either $G$ is a torsion element and 
$\|g^n\|_{\mbox{\em\tiny com}}\simeq 1$ for all $n$ or
\begin{equation} \label{Ngn}
\forall\,n,\;\;\|g^n\|_{\mbox{\em \tiny com}} \simeq  F_{S}\circ 
\mathbf F^{-1}_{j}(n) \mbox{ where }j=j_{\mathfrak w}(g).\end{equation}
\end{pro}
\begin{proof} The upper bound is very easy. 
Let $\kappa$ be such that $g^\kappa\in G^\mathfrak w_j$, $j=j_\mathfrak w(g)$.
Since $g^\kappa$ is in $G^\mathfrak w_j$
it can be written as word $\omega$ using formal commutators of weight at least 
$\bar{w}_j$.  Hence, $g^{\kappa n}$ can be written as a word 
$\omega_n$, namely, $\omega $ repeated $n$ times. Obviously, 
if $w(c)\ge \bar{w}_j$,
$\mbox{deg}_{c}(\omega_n)\le  \mbox{deg}_c(\omega) n$. By definition, this 
implies 
$\|g^{\kappa n}\|_{\mbox{\tiny com}}\le C F_S\circ \mathbf F_j^{-1}(n)$.
The estimate $\|g^{n}\|_{\mbox{\tiny com}}\le C'  F_S\circ \mathbf F_j^{-1}(n)$
easily follows.

The lower bound is more involved. 
Using Theorem \ref{th-coord2F},
it suffices to show that any  writing of $g^{\kappa n}$ as a product
\begin{equation}\label{writting}
g^{\kappa n}=\prod_1^t c_i^{x_i} \mbox{ with }|x_i| \le C \mbox{ for } 
i\in \cup_h\{m_{h-1}+R^\mathfrak w_h+1,\dots,m_h\}
\end{equation}
must have 
$\max _{i\in \{m_{j-1}+1,\dots,m_{j-1}+R^\mathfrak w_j\}} \{|x_i|\} \ge  cn$.   
First, we claim that there exists a constant $T$ (independent of $g$ 
but depending
on the structure of $G$, $S$, the weight system $\mathfrak w$ 
and the constant $C$ appearing in the above displayed equation) 
such that for any $n$ and any writing of $g^{\kappa n}$ as above we have
\begin{equation}\label{bounded}
|x_i|\le T \mbox{ for all } i\le m_{h-1}, h\le j.
\end{equation}
The proof is by induction on $h\le j$. There is nothing to prove for $h=1$.
Assume that $h+1\le j$ and that we have proved that $|x_i|\le T$ for all 
$i\le m_{h-1}$.
Since $g^\kappa, g^{\kappa n} \in G^\mathfrak w_h$, the product
$\sigma=\prod _1^{m_{h-1}}c_i^{x_i}$ is in $G^\mathfrak w_h$. Since $|x_i|\le T$,
$i\le m_{h-1}$, $\sigma =\prod_{i> m_{h-1}} c_i^{z_i}$ with $|z_i|\le T'$
where $T'$ depends only on $G, S, \mathfrak w, T$ but not on $g,n$. Computing in 
$G^\mathfrak w_h$
modulo $G^\mathfrak w_{h+1}$, we have
$$g^{\kappa n}= \prod_{m_{h-1}+1}^{m_h} c_i^{x_i+z_i}=e\;\mbox{ mod } 
G^\mathfrak w_{h+1} .$$ The last equality 
holds because $g^{\kappa n}\in G^\mathfrak w_h$ and $h+1\le j$.
Since 
$$\{c_{m_{h-1}+1},\dots, c_{m_{h-1}+R^\mathfrak w_h}\}$$ is free
in $G^\mathfrak w_h/G^\mathfrak w_{h+1}$ and $\sup_i |z_i|\le T'$,
$\sup\{|x_i|:  m_{h-1}+R^\mathfrak w_h+1\le i\le m_h\}\le C$, 
there is a constant $T''$ depending only on $G,S,\mathfrak w, C$ 
and $T'$ such that
$|x_i|\le T''$ for $i\in \{m_{h-1}+1,\dots, m_{h-1}+R^\mathfrak w_h\}$.
This proves (\ref{bounded}).

On the one hand, since $j$ is the largest integer such that 
$g^u\in G^\mathfrak w_{j}$ 
for some $u$, it follows that for any $n$ we can write 
$$g^{\kappa n}= \prod_{i=m_{j-1}+1}^{m_j} 
c_i^{y_i} \mbox{ mod } G^\mathfrak w_{j+1}  \mbox{ with }
\sum_{i=m_{j-1}+1}^{m_{j-1}+R^\mathfrak w_j} 
|y_i| \ge cn$$ 
and  
$$\max\{|y_i|: m_{j-1}+R^\mathfrak w_j+1\le i\le m_j\}
\le C'.
$$
On the other hand, since any writing of $g^{\kappa n}$ as in 
(\ref{writting}) satisfies (\ref{bounded}), the same reasoning as in the 
induction step for (\ref{bounded}) gives
$$g^{\kappa n}= \prod_{m_{j-1}+1}^{m_j} c_i^{y_i-x_i-z_i}=e\;\mbox{ mod } 
G^{\mathfrak w}_{j+1} $$ 
with $|z_i|\le T$.
Since $\{c_i: m_{j-1}+1\le i \le m_{j-1}+R^\mathfrak w_j\}$ is free,  
the facts that
$$\sum_{i=m_{j-1}+1}^{m_{j-1}+R^\mathfrak w_j}  
|y_i| \ge cn ,\;\;\max\{|y_i|: m_{j-1}+R^\mathfrak w_j+1\le i\le m_j\}
\le C'$$ and $|z_i|\le T$ together imply that  
$$\sum_{i=m_{j-1}+1}^{m_{j-1}+R^\mathfrak w_j} |x_i|\ge c'n.$$ 
Hence, $\|g^{\kappa n}\|_{\mbox{\tiny com}}\simeq 
F_S\circ \mathbf F_j^{-1}(n)$. \end{proof}

\begin{theo} \label{th-normeq}
Let $G$ be a countable nilpotent group equipped with two 
generating tuples $S,S'$ and associated multidimensional weight systems  
$\mathfrak w,\mathfrak w'$ as well as weight function systems 
$\mathfrak F,\mathfrak F'$  satisfying {\em (\ref{F1})-(\ref{F2})}. By definition, $F_S$ and $F'_{S'}$ are the weight functions associated with the smallest 
weights in $\mathfrak w$ and $\mathfrak w'$, respectively. 
Let $\Sigma=(c_1,\dots,c_t)$ be a sequence of formal
commutators as in {\em Theorem \ref{th-coord2F}} applied to 
$(S,\mathfrak w,\mathfrak F)$.
\begin{enumerate}
\item Assume that $S'\supset \mbox{\em core}(\mathfrak w,S,\Sigma)$ 
and $F'_s\ge F_s$ 
for all $s\in \mbox{\em core }(\mathfrak w,S,\Sigma)$.
Then $$ \forall\,g\in G,\;\;
(F'_{S'})^{-1}(\|g\|_{S',\mathfrak F'}) \le C F_S^{-1}(\|g\|_{S,\mathfrak F})$$
\item Assume that, for all $s\in S'$,  
$F'_s\le \mathbf F_{j_\mathfrak w(s)}$. 
Then $$ \forall\,g\in G,\;\;
(F'_{S'})^{-1}(\|g\|_{S',\mathfrak F'}) \ge c  F_S^{-1}(\|g\|_{S,\mathfrak F})$$
\end{enumerate}
\end{theo}
\begin{proof} To prove the first statement, referring to the notation used
in Theorem \ref{th-coord2F},  Set 
$$I_1=
\cup_j \{m_{j-1}+1,\dots,m_{j-1}+R^\mathfrak w_j\},\;\; 
I_2=\{1,\dots, t\}\setminus I_1$$  and recall that
any any $g\in G$ can be written as 
$$g=\prod_1^tc_i^{x_i}, \;\;\; |x_i|\le C\left\{\begin{array}{cc}
F_{c_i}(F_S^{-1}(\|g\|_{\mbox{\tiny com}})) &\mbox{ if } i\in I_1\\
1 &\mbox{ if } i\in I_2.\end{array}\right. $$
By hypothesis, 
$F'_{c_i}\ge F_{c_i}$ for $i\in I_1$. Further, each $c_i$, $i\in I_2$, 
is a product of elements in $S'$. Hence, we obtain an expression for $g$
as a word $\omega$ on formal commutators on $S'$ with 
$\mbox{deg}_{c}(\omega)\le C F'_c(F_S^{-1}(\|g\|_{\mbox{\tiny com}}))$. 
This proves that 
$(F'_{S'})^{-1}(\|g\|_{S',\mathfrak F'}) \le C F_S^{-1}(\|g\|_{S,\mathfrak F})$
as desired.

To prove the second statement, apply Theorem \ref{th-coord2F}(iii) to 
$(S',\mathfrak w',\mathfrak F')$ to write any $g\in G$ as a product
$$g=\prod_1^p (s'_{i_j})^{x_j} \mbox{ with }|x_j|\le  F'_{s'_{i_j}}\circ 
(F'_{S'})^{-1}(\|g\|_{S',\mathfrak F'})$$
where $s'_{i,j}\in S'$ (note that the sequence $(i_j)$ and the integer $p$ are fixed and independent of $g$).
By Proposition  \ref{pro-gn} and the hypothesis  
$\mathbf F_{j_{\mathfrak w}(s)} \ge F'_s$ for all $s\in S'$, we obtain that
$F^{-1}_S(\|g\|_{S,\mathfrak F})\le C (F'_{S'})^{-1}(\|g\|_{S',\mathfrak F'})$
as desired.
\end{proof}

\begin{cor}\label{cor-normeq}
Let $G$ be a countable nilpotent group equipped with two 
generating tuple $S,S'$ and associated  multidimensional weight systems  
$\mathfrak w,\mathfrak w'$  with function systems 
$\mathfrak F,\mathfrak F'$  
satisfying {\em (\ref{F1})-(\ref{F2})}. 
Let $\Sigma=(c_1,\dots,c_t)$ be a sequence of formal
commutators as in {\em Theorem \ref{th-coord2F}} applied to 
$(S,\mathfrak w,\mathfrak F)$.
Assume that there exists $C\in (0,\infty)$ such that
the following two conditions are satisfied:
\begin{itemize}
\item[(i)] $\mbox{\em core}(\mathfrak w,S,\Sigma)\subset S'$ and, 
$\forall\, s\in \mbox{\em core}(\mathfrak w,S,\Sigma),\;\; CF'_s\ge F_s.$
\item[(ii)]
$\forall\,s\in S',\;\; F'_s\le C\mathbf F_{j_\mathfrak w(s)}$.
\end{itemize}
Then
$$ \forall\,g\in G,\;\;
(F'_{S'})^{-1}(\|g\|_{S',\mathfrak F'}) 
\simeq  F_S^{-1}(\|g\|_{S,\mathfrak F}).$$
In particular,
$$\forall\,r>0,\;\;\#Q(S',\mathfrak F',r)
\simeq \#Q(S,\mathfrak F,r).$$
\end{cor}

\begin{exa}[Continuation of Example \ref{exa-illustr}] \label{exa-illustr2}
Consider the 
discrete Heisenberg group as in  Example \ref{exa-illustr}
equipped with the generating $3$-tuple $S=(s_1=X,s_2=Y,s_3=Z)$
and $S'=(s'_i=X,s'_2=Y)$. Set $F_1(r)=F'_1(r)= r^{3/2}$, 
$F_2(r)= F'_2(r)=r^{2}\log (e+r)$, $F_3(r)= r^{\gamma}$, $\gamma>3/2$,
and let $\mathfrak F,\mathfrak F'$ be the associated weight-function systems. 
The natural $2$ dimensional weight systems $\mathfrak w,\mathfrak w'$ 
are generated by $w_1=w'_1=(3/2,0)$, $w_2=w'_2=(2,1)$, $w_3= (\gamma,0)$.
The first observation is that $\mbox{core}(\mathfrak w,S)=(s_1,s_2,s_3)$
is $\gamma>7/2$ and $\mbox{core}(\mathfrak w,S)=(s_1,s_2)$ if 
$3/2<\gamma\le 7/2$. It follows that, 
$\forall\, g\in G$, $\|g\|_{S', \mathfrak F'}\simeq \|g\|_{S,\mathfrak F}$
if $\gamma\in (3/2, 7/2]$ whereas these norms are not equivalent if $\gamma>7/2$.
\end{exa}

\section{Volume estimates}
\label{sec-vol}
This section gathers some of the main results we will need 
regrading volume estimates for the balls $Q(S,\mathfrak F,r)$ introduced in Definition \ref{defin-FQN}. It also addresses the question of how changes 
in the weight-function system affect these volume estimates.

We start with a  
general and very flexible result which admits a rather simple proof. 
In this theorem, 
the weight-function system $\mathfrak F$  is not
necessarily tightly related to the weight system $\mathfrak w$.
The proof of this theorem will be given in the last section of this paper.
\begin{theo} \label{th-coord3}
Let $\mathfrak w$ be a multidimensional weight system as in 
{\em Section \ref{sec-wF}}. Assume that we are given weight functions
$F_i$, $1\le i\le k$ satisfying {\em (\ref{F1})}. 
Let $\Sigma=(c_1,\dots, c_s)$ be a $s$-tuple of formal commutators on 
$\{s_i^{\pm 1}: 1\le i\le k\}$.
Assume that, for any $h$, the family $\{c_i: w(c_i)=\bar{w}_h\}$ projects to 
a free family in the abelian group $G^\mathfrak w_h/G^\mathfrak w_{h+1}$.  
Then there exist an integer $M=M_\Sigma$ and a sequence 
$(i_1,\dots,i_M)\in \{1,\dots,k\}^M$, depending on $\Sigma$
such that for any $r>0$ there exists a subset $K_\Sigma(r)\subset G$
satisfying the following two properties: 
\begin{enumerate}
\item $\displaystyle \# K_\Sigma(r) \ge \prod_{i=1}^s (2F_{c_i}(r)+1) $
\item $g\in K_\Sigma(r)\Longrightarrow g=\prod_{j=1}^Ms_{i_j}^{x_{j}}$, 
$|x_{j}|\le F_{i_j}(r).$
\end{enumerate}
Further, every $s_{i_j}$, $1\le j\le M$, belongs to the build-sequence
of at least one $c_h\in \Sigma$.
\end{theo}
Theorem \ref{th-coord3} is very useful for comparing the volume growth 
associated with different ``weight-function systems''. See the proof of 
Theorem \ref{th-Vcomp} below.

Next we state and prove sharp volume estimates related to 
Theorem \ref{th-coord2F}.

\begin{theo} \label{th-VolQ}
Referring the setting and notation of {\em
Theorem \ref{th-coord2F}}, we have
$$\#Q(\mathfrak C(S),\mathfrak F,r)\simeq \#Q(\Sigma, \mathfrak F,r)
\simeq \#Q(S,\mathfrak F,r)\simeq \prod_{j=1}^{j_*}
\mathbf F_j(r)^{R^\mathfrak w_j}.
$$
\end{theo}
\begin{rem}Assume that the weight system $\mathfrak w$ is unidimensional,
generated by $(w_i)_1^k\in (0,\infty)^k$,
and the weight-functions $F_i$ are power functions
$F_i(r)= r^{\mathfrak w_i}$, $i=1,\dots, k$. Then
$$Q(S,\mathfrak F,r)\simeq r^{D(S,\mathfrak w)}$$
with $D(S,\mathfrak w)$ as in Definition \ref{def-DSa}.
\end{rem}

\begin{proof} The equivalences 
$\#Q(\mathfrak C(S),\mathfrak F,r)\simeq \#Q(\Sigma, \mathfrak F,r)
\simeq \#Q(S,\mathfrak F,r)$ and the
upper bound $\#Q(\Sigma, \mathfrak F,r)\le C \prod_{j=1}^{j_*}
\mathbf F_j(r)^{R^\mathfrak w_j}$
follows immediately from Theorem \ref{th-coord2F} and inspection.

The lower bound $\#Q(\Sigma, \mathfrak F,r)\ge c\prod_{j=1}^{j_*}
\mathbf F_j(r)^{R^\mathfrak w_j}$
requires an additional argument.
Note that $Q(\Sigma,\mathfrak F,r)$ contains the image in $G$ of 
$$\prod _{j=1}^{j_*} \prod_{i=m_{j-1}+1}^{m_{j-1}+R_j} c_i^{x_i},\;\;|x_i|\le F_{c_i}(r).$$
Further, it is not hard to check that 
$$\prod _j \prod_{i=m_{j-1}+1}^{m_{j-1}+R_j} c_i^{x_i}=
\prod _j \prod_{i=m_{j-1}+1}^{m_{j-1}+R_j} c_i^{y_i}$$
implies
$$  x_i=y_i,\;\;
i\in \bigcup_{j=1}^{j_*}\{m_{j-1}+1,\dots,m_{j-1}+R_j\}.$$  
The desired lower bound follows.
\end{proof}

\begin{theo} \label{th-Vcomp}
Let $G$ be a countable nilpotent group equipped with two 
generating tuples $S,S'$ and associated multidimensional weight systems  
$\mathfrak w,\mathfrak w'$ as well as weight function systems 
$\mathfrak F,\mathfrak F'$  satisfying {\em (\ref{F1})-(\ref{F2})}. 
Let $\Sigma=(c_1,\dots,c_t)$ be a sequence of formal
commutators as in {\em Theorem \ref{th-coord2F}} applied to 
$(S,\mathfrak w,\mathfrak F)$.
Assume that  $S'\supset \mbox{\em core}(\mathfrak w,S,\Sigma)$ and that 
$$F'_s\ge F_s  \mbox{ for all }s\in \mbox{\em core}(\mathfrak w,S,\Sigma).$$  
Then
$$\#Q(S',\mathfrak F',r)\simeq \prod_{j=1}^{j_*(\mathfrak w')} 
\mathbf F'_j(r)^{R^{\mathfrak w'}_j}
\ge \#Q(S,\mathfrak F,r)\simeq \prod _{j=1}^{j_*(\mathfrak w)} 
 \mathbf F_j(r)^{R^{\mathfrak w}_j}
.$$
Assume 
further that there exists $\sigma \in S'$ such that 
$F'_\sigma \ge  \mathbf F_{j_\mathfrak w(\sigma)}$. Then
$$\#Q(S',\mathfrak F',r) 
\ge c\left(\frac{F'_\sigma(r)}{\mathbf F_{j_\mathfrak w(\sigma)}(r)}\right) 
\#Q(S,\mathfrak F,r).$$
\end{theo}
\begin{proof}
Since $\mbox{core}(\mathfrak w,S,\Sigma)\subset S'$ it follows 
that, for any  $c_i\in \Sigma$, $F'_{c_i}$ is well defined as
the product of $F'_{s}$ with $s\in \mbox{core}(\mathfrak w,S,\Sigma)\subset S'$. 
Use the 
collection of commutators $c_i$, 
$i\in \{m_{j-1}+1,\dots, m_{j-1}+R^{\mathfrak w}_j\}$, $j=1,\dots,j_*$ in 
Theorem \ref{th-coord2F} 
with the weight system $\mathfrak w$ and 
weight-function system  $\mathfrak F'$. For each $r$, 
Theorem \ref{th-coord3} provides a set $K(r)\in G$ such that
\begin{equation} \label{cruxK}
\# K(r) \ge  
\prod_{j=1}^{j_*(\mathfrak w)}\prod_{i=m_{j-1}+1}^{m_{j-1}+R^\mathfrak w_i} F'_{c_i}(r) \end{equation}
and, by Theorem \ref{th-coord2F} , Theorem \ref{th-coord3} and the definition of 
$\mbox{core}(\mathfrak w,S,\Sigma)$,
$$ K(r)\subset \{g\in G: \|g\|_{S',\mathfrak F'}\le F'_{S'}(r)\}.$$
By Theorem
\ref{th-VolQ}, it follows that 
$$\forall\,r,\;\;\#K(r)\le \#Q(S',\mathfrak F',r).$$

By hypothesis, $F'_s\ge F_s$ if $s\in \mbox{core}(\mathfrak w,S,\Sigma)$. Hence 
$F'_{c_i}\ge F_{c_i}$ (i.e., $w'(c_i)\ge w(c_i)$). By (\ref{cruxK}) and 
Theorem \ref{th-VolQ}, this implies $\#K(r)\ge c\prod_{j=1}^{j_*(\mathfrak w)}
\mathbf F_j^{R^\mathfrak w_j}$.
This proves the first statement.

Suppose now that there exists $\sigma\in S'$ such that
$w'(s)>\bar{w}_{j_\mathfrak w(\sigma)}$.  Set $j_0=j_\mathfrak w(\sigma)$.
In the sequence of commutators $c_1,\dots,c_t$ used above, consider the
the free family 
$$\{c_i: i\in \{m_{j_0-1}+1,\dots,m_{j_0-1}+R^\mathfrak w_{j_0}\}\} 
\mbox{ in }G^\mathfrak w_{j_0}/G^\mathfrak w_{j_0+1}.$$ 
By hypothesis,  there exists an integer $u$
such that $\sigma^u\in G^\mathfrak w_{j_0}$ is free in 
$G^\mathfrak w_{j_0}/G^{\mathfrak w}_{j_0+1}$. Since a maximal free subset of
$\{\sigma^u\}\cup
\{c_i: i\in \{m_{j_0-1}+1,\dots,m_{j_0-1}+R^\mathfrak w_{j_0}\}\}$ in 
$G^\mathfrak w _{j_0}/G^\mathfrak w_{j_0+1}$
containing $\sigma^u$ must contain $R^\mathfrak w_{j_0}$ elements, 
we can replace one of the $c_i$, say $c_{i_*}$  
by $\sigma^u$ so that the $R^\mathfrak w_{j_0}$-tuple  so obtained is free in $G^\mathfrak w _{j_0}/G^\mathfrak w_{j_0+1}$. 
Let $b_i= c_i$ if $i\neq i_*$, $b_{i_*}=\sigma^u$, $\widetilde{F}^i=F'_{c_i}$ 
if $i\neq i_*$, $\widetilde{F}^{i_*}(r)=F'_{\sigma}(r/|u|)$, 
and apply Theorem \ref{th-coord32}.
The desired result follows.
\end{proof}

\section{Random walk upper bounds} \label{sec-rwub}
\setcounter{equation}{0}
This section is devoted to obtaining upper bounds on 
the return probability of a large collection of random walks 
including those driven by the measures $\mu_{S,a}$. Generalizing one of the 
approaches developed in \cite{VSCC} for simple random walks, we will make 
use of  appropriate volume growth estimates and of the notion of 
pseudo-Poincar\'e inequality. 

\subsection{Pseudo-Poincar\'e inequality} Let $G$ be a group generated by a 
finite symmetric set $A$. Then it holds that for any finitely supported 
function $f$ on $G$,
\begin{equation}\label{PP}
\|f_g-f\|^2_2\le C_A |g|^2_A \mathcal E_A(f,f)
\end{equation}
where $$\mathcal E_A(f,f)=\frac{1}{2|A|}\sum_{x\in G, y\in A}|f(xy)-f(x)|^2.$$
This expression is the Dirichlet form 
associated with the simple random walk based on $A$. Inequality (\ref{PP}) 
captures a fundamental universal property of Cayley graphs. 
In \cite{VSCC}, it is proved that this simple property implies 
interesting upper-bounds on $u_A^{(2n)}(e)$ in terms of the 
volume growth function $V_A$.

The main result of this section is a pseudo-Poincar\'e inequality 
adapted to probability measure of the form
\begin{equation}\label{def-mutrunc}
\mu(g)=k^{-1}\sum_{j=1}^k \sum_{n\in \mathbb Z}\mu_i(n)\mathbf 1_{s_i^n}(g).
\end{equation}
where $(s_1,\dots,s_k)$ is a generating $k$-tuple in $G$ and  
the $\mu_i$'s are probability measures on $\mathbb Z$ with 
truncated second moment
\begin{equation}\label{def-Gi}
\mathcal G_i(n):=\sum_{|m|\le n}m^2 \mu_i(n)
\end{equation}
satisfying
\begin{equation} \label{tsmhyp}
\mathcal G_i(n)\ge c n^{2-\tilde{\alpha}_i} L_i(n),\;\; 
\tilde{\alpha}_i\in (0,2],
\end{equation}
for some slowly positive varying functions $L_i$, $1\le i\le k$.
Under these circumstances, we let $F_i$ denote the inverse function of 
$n\mapsto n^{\tilde{\alpha}_i}/L_i(n)$. 
The function $F_i$ is a regularly varying 
function of positive index $1/\tilde{\alpha}_i\in [2,\infty)$.
In addition, we assume that the $\mu_i$'s are essentially decreasing in the sense that there is a constant $C_1$ such that
\begin{equation}\label{decreasing}
\forall\, i=1,\dots, k, 0\le m\le n,\;\;\mu_i(n)\le C_1 \mu_i(m).
\end{equation}

\begin{exa} The measure $\mu_{S,a}$ with $a=(\alpha_i)_1^k\in (0,\infty)^k$
satisfies
$$\mathcal G_i(n) \simeq \left\{\begin{array}{ll}n^{2-\alpha_i} & \mbox{ if } 
\alpha_i\in (0,2),\\
\log n & \mbox{ if } \alpha_i=2,\\
1 &\mbox{ if } \alpha_i>2.
   \end{array}\right.$$ 
Hence, in this case, we have $\tilde{\alpha}_i=\min\{\alpha_i,2\}$ and
$L_i=1$ except if $\alpha_i=2$ in which case $L_i(n)=\log n$.
\end{exa}

We will make use of the following general result 
(which is essentially well-known). We let $\mathcal C_c(G)$ be the 
set of all finitely supported function on $G$ and set $f_g(x)=f(xg)$.
\begin{theo} \label{th-Nash}
Let $G$ be a finitely generated group. Let $\mu$ be a symmetric 
probability measure on $G$. Assume that for each $r\ge 1$ there is a subset 
$K(r)$ of $G$ such that
\begin{equation}\label{KPP1}
\forall\, g\in K(r),\;\;\|f_g-f\|^2_2\le C_0\; r \mathcal E_\mu(f,f).
\end{equation}
and 
\begin{equation}\label{VK}
\forall \, r\ge 1,\;\; \#K(r)\ge v(r)
\end{equation} where $v$ is increasing and 
regularly varying of positive index.  
Let $\psi$ be the right-continuous inverse of $v$. 
Then there is a function $\Psi\simeq \psi$ such that the Nash inequality
\begin{equation} \label{KNash}
\forall \, f\in \ell^1(G),\;\;
\|f\|_2 ^2\le \Psi( \|f\|^2_1/\|f\|^2_2) \mathcal E_\mu(f,f)
\end{equation}
is satisfied. Moreover
\begin{equation}\label{KUB}
\mu^{(2n)}(e)\le C_1 \eta (n)
\end{equation}
where $\eta$ is defined implicitly by 
$$ \tau=\int_1^{1/\eta(\tau)} \Theta(s)\frac{ds}{s},\;\; \tau>0.$$
\end{theo}
\begin{proof} Assuming (\ref{KPP1}) and $\#K(r)\ge v(r)$, 
the Nash  inequality (\ref{KNash})
easily follows from writing
$$ \|f\|_2\le \|f-f_{K(r)}\|_2 +\|f_{K(r)}\|_2\le  C_0 r\mathcal E(f,f)+ 
v(r)^{-1/2}\|f\|_1$$
and optimizing in $r$. Here $f_{K(r)}(x)$ is the average of $f$ over $xK(r)$
so that, obviously, $\|f_{K(r)}\|_2\le (\#K(r))^{-1/2}\|f\|_1$ and  
(\ref{KPP}) gives
$\|f-f_{K(r)}\|_2\le C_0 r\mathcal E_\mu(f,f)$ with $C_0=CMk$.
The return probability estimate (\ref{KUB}) is a well-known consequence of 
(\ref{KNash}). See \cite{CNash,PSCnote}.
\end{proof}
\begin{rem} In this theorem, the parametrization of the set $K(r)$ is chosen so 
that $r$ appears on the right-hand side of (\ref{KPP1}) instead of $r^2$.
\end{rem}

\begin{theo}  \label{th-PP}
Let $G$ be a finitely generated nilpotent group equipped with a
generating $k$-tuple $(s_1,\dots, s_k)$. Let $\mu$ be as in 
{\em (\ref{def-mutrunc})} with $(\tilde{\alpha}_i)_1^k$, $L_i$ and $F_i$ be 
as in {\em (\ref{tsmhyp})}. Assume that {\em (\ref{decreasing})} holds. 
Assume that there exists an integer $M$ and  a
sequence $(i_j)_1^M\in \{1,\dots,k\}^M$ such that for each $r\ge 1$
there is a subset  $K(r)$ of $ G$ with the property that
\begin{equation}\label{hypK}
g\in K(r)\Longrightarrow g=\prod_1^Ms_{i_j}^{x_j}\;\;\mbox{ with  } |x_j|\le 
F_{i_j}(r).
\end{equation}
Then there exists a constant $C=C(\mu)$ such that
\begin{equation}\label{KPP}
\forall\, g\in K(r),\;\;\|f_g-f\|_2\le CM^2\; r \mathcal E_\mu(f,f).
\end{equation}
\end{theo}
\begin{proof} Because we assume (\ref{hypK}), the proof boils down 
to a collection of one dimensional inequalities, one for each 
of the measures $\mu_i$ on $\mathbb Z$ that appear in the definition
(\ref{def-mutrunc}) of $\mu$.  Indeed,  Lemma \ref{lem-1d} stated below 
shows that
there exists a constant $C$ such that, for each 
$i\in \{1,\dots,k\}$ and  $y\in \mathbb Z$ with $|y|\le F_{i}(r)$ we have
\begin{equation}\label{claim1d}
\|f_{s_i^y}-f\|^2_2\le C \, r\, \mathcal E _{\mu }(f,f)\end{equation}
for any finitely supported function $f$ on $G$. 
Together, (\ref{hypK}) and (\ref{claim1d})
imply (\ref{KPP}).  Since there exists a constant $C$ such that, 
for all $i\in \{1,\dots,k\}$, 
$$|y|\le F_i(r) \mbox{ implies }
\mathcal G_i(|y|)^{-1}|y|^2\le C r,$$ the claim (\ref{claim1d}) follows from Lemma \ref{lem-1d}. 
\end{proof}
\begin{lem} \label{lem-1d}
Let $\nu$ be a symmetric probability measure on $\mathbb Z$ 
satisfying  
$$\exists\, C_1,\;\;\forall\, 0\le m\le n,\;\;\nu(n)\le C_1\nu(m).$$ 
Let $G$ be a finitely generated group equipped with a distinguished element $s$.
Set $$\mathcal E_{s,\nu}(f,f)=\frac{1}{2}\sum_{x\in G,z\in \mathbb Z}
|f(xs^z)-f(x)|^2\nu(z)\;\mbox{ and }\;\;\mathcal G_\nu(m)=\sum_{|n|\le m}|n|^2\nu(n).$$
\begin{itemize}
\item[(i)] For any finitely supported function $f$ on $G$ we have
$$\forall\,y\in\mathbb Z,\;\;\|f_{s^y}-f\|_2^2\le C_\nu 
\left(\mathcal G_{\nu}(|y|)\right)^{-1}|y|^2\mathcal E_{s,\nu} (f,f).$$
\item[(ii)]  Further, 
for any two finitely supported functions $f,g$ we have%
\[
\forall x\in G,\ y\in 
\mathbb{Z}
,\ \left\vert f\ast g(xs^{y})-f\ast g(x)\right\vert ^{2}\leq C_{\nu }(%
\mathcal{G}_{\nu }(\left\vert y\right\vert ))^{-1}\left\vert y\right\vert
^{2}\mathcal{E}_{s,\nu }(f,f)\left\Vert g\right\Vert _{2}^{2}.
\]
\end{itemize}
\end{lem}
\begin{proof}[Proof of (i)]
For any pair of integers  $0<m\leq n,$ write 
$n=a_{m}m+b_{m}$ with $0\leq b_{m}<m$ and 
\begin{eqnarray*}
\left\Vert f-f_{s^{n}}\right\Vert _{2}^{2} &=&\sum_{x\in G}
(f(xs^{n})-f(x))^{2} \\
&\leq &2\sum_{x\in G }(f(xs^{a_{m}m})-f(x))^{2}+2\sum_{x\in G
}(f(xs^{b_{m}})-f(x))^{2} \\
&\leq &2a_{m}^{2}\sum_{x\in G }(f(xs^{m})-f(x))^{2}+2\sum_{x\in G
}(f(xs^{b_{m}})-f(x))^{2}.
\end{eqnarray*}
This yields
\begin{eqnarray*}
\| f-f_{s^{n}}\| _{2}^{2}
\left(\sum_{m=1}^n m^2\nu(m)\right) &\leq&
2\sum_{x\in G }\sum_{m=1}^n (f(xs^{m})-f(x))^{2} a_{m}^{2}m^2\nu(m)\\
&& +2\sum_{x\in G
}\sum_{m=1}^n (f(xs^{b_{m}})-f(x))^{2}m^2\nu(m).
\end{eqnarray*}
Next, observe that
\begin{eqnarray*}
\lefteqn{\sum_{x\in G }\sum_{m=1}
^n (f(xs^{m})-f(x))^{2}(a_{m}m)^{2}\nu (m)} \hspace{1in}&&\\
&\leq&  n^{2}\sum_{x\in G }\sum_{m=1}^n(f(xs^{m})-f(x))^{2}\nu(m)
\leq n^{2}\mathcal{E}_{s,\nu }(f,f). 
\end{eqnarray*}
Further, using the hypothesis that $\nu$ is essentially decreasing, i.e., 
$\nu(m)\le C_1\nu(b)$ is $0\le b\le m$, write
\begin{eqnarray*}
\lefteqn{\sum_{x\in G }\sum_{m=1}^n(f(xs^{b_{m}})-f(x))^{2}m^{2}\nu (m)}
\hspace{1in} &&\\
&=&\sum_{x\in G}\sum_{b=1}^{n/2}\sum_{m|n-b \atop  b<m\le n}
(f(xs^{b})-f(x))^{2}m^{2}\nu (m) \\
&\leq &C_{1}\sum_{x\in G }\sum_{b=1}^{n/2}\left(\sum_{
m|n-b \atop b<m\le n}m^{2}\right)
(f(xs^{b})-f(x))^{2}\nu (b).
\end{eqnarray*}%
As
$$\sum_{
m|n-b \atop b<m\le n}m^{2}\le (\sum_1^\infty i^{-2})n^2,$$
we obtain
$$\sum_{x\in G }\sum_{m=1}^n(f(xs^{b_{m}})-f(x))^{2}m^{2}\nu (m)
\le C_2n^2\mathcal E_{s,\nu}(f,f).$$
It follows that, for both $n>0$ and $n<0$,
\begin{eqnarray*}
\| f-f_{s^{n}}\| _{2}^{2} \left(\sum_{0<m\leq |n|}m^{2}\nu(m)\right)
 &\le & 2(1+C_2)n^2 \mathcal{E}_{s,\nu }(f,f).
\end{eqnarray*}
\end{proof} 
\begin{proof}[Proof of (ii)]
By Cauchy-Schwarz%
\begin{eqnarray*}
\left\vert f\ast g(xs^{y})-f\ast g(x)\right\vert&=&  \
\left\vert \sum_{z\in G}(f(z^{-1}xs^{y})-f(z^{-1}x))g(z)\right\vert  \\
&\leq &\left( \sum_{z\in G}(f(z^{-1}xs^{y})-f(z^{-1}x))^{2}\right) ^{\frac{1%
}{2}}\left( \sum_{z\in G}|g(z)|^2\right) ^{\frac{1}{2}} \\
&=&\left\Vert f-f_{s^{y}}\right\Vert _{2}\left\Vert g\right\Vert _{2}.
\end{eqnarray*}%
Applying part (i) to $\left\Vert f-f_{s^{y}}\right\Vert _{2}$ yields  the
desired inequality.
\end{proof}

\begin{rem} When $G=\mathbb Z$, Lemma \ref{lem-1d} provides an 
interesting and new pseudo-Poincar\'e inequality for probability measure $\nu$
satisfying (\ref{decreasing}) (i.e., which are essentially decreasing) in terms of the truncated second moment $\mathcal G_\nu$. Namely, 
assuming (\ref{decreasing}), we have 
$$\sum_{x\in \mathbb Z}|f(x+y)-f(x)|^2\le C_\nu \frac{|y|^2}{\mathcal G_\nu(|y|)}
\mathcal E_\nu(f,f)$$
where
$$\mathcal E_\nu(f,f)=\frac{1}{2}\sum_{x,z\in \mathbb Z}|f(x+z)-f(x)|^2\nu(z).$$
Together with the trivial fact that $\#\{y:|y|\le r\}=2r+1$, 
this pseudo-Poincar\'e inequality and Theorem \ref{th-Nash} 
provide a sharp Nash inequality satisfied 
by $\mathcal E_\nu$.
\end{rem}

\subsection{Assorted return probability upper bounds}
This section describes direct applications of Theorem \ref{th-coord3} 
together with Theorems \ref{th-Nash}-\ref{th-PP}. 
We use the notation introduced in Sections \ref{sec-wD} and \ref{sec-wF}.

\begin{theo}\label{th-up1} 
Let $G$ be a finitely generated nilpotent group equipped with a 
generating $k$-tuple $(s_1,\dots,s_k)$ and a $k$-tuple of positive reals
$a=(\alpha_1,\dots,\alpha_k)$. Let $\mathfrak w$ be the weight system which 
assigns weight $w_i=1/\tilde{\alpha}_i$ to $s_i$ where 
$\tilde{\alpha}_i=\min\{2,\alpha_i\}$.
Then
$$\mu_{S,a}^{(n)}(e)\le C_{S,a} n^{-D(S,\mathfrak w)}$$
where $D(S,\mathfrak w)=
\sum_h \bar{w}_h \mbox{ \em rank}(G^\mathfrak w_h/G^\mathfrak w_{h+1}).$
\end{theo}
\begin{proof} 
By Theorem \ref{th-coord3}, for each $r \ge 1$ we can find a subset 
$K(r)$ of $G$ such that $\#K(r)\ge r^{D(S,\mathfrak w)}$ and $g\in K(r)$ implies
$g=\prod_1^M s_{i_j}^{x_{j}}$ with $|x_i|\le  r^{w(s_{i_j})}$. 
The result then follows from Theorems \ref{th-Nash}-\ref{th-PP}
\end{proof}
\begin{rem} If all the $\alpha_i$'s are in $(0,2)$ or, more generally, if
$R^\mathfrak w_h>0$ implies $\bar{w}_h>1/2$, the upper bound given in Theorem \ref{th-up1}
is sharp. Indeed, we will prove a matching lower bound in the next section.

If all the $\alpha_i$'s are greater than $2$ the measure $\mu_{S,a}$ 
has finite second moment and 
$D(S,\mathfrak w)= \frac{1}{2}\sum h \mbox{ rank}(G_h/G_{h+1})$. In this case
the upper bound of Theorem \ref{th-up1} is also sharp. It coincides with the 
bound provided by Corollary \ref{cor-PhiG}.

We conjecture that this upper bound is sharp when $\alpha_i\neq 2$ 
for all $i\in \{1,\dots,k\}$ but we have not been able to prove 
this conjecture when there exists $i,j$ such that $\alpha_i<2$ and $\alpha_j>2$.
\end{rem}  

The next result shows that Theorem \ref{th-up1} is not always sharp when
some of the $\alpha_i$'s are equal to $2$. 

\begin{theo} \label{th-up2} 
Let $G$ be a finitely generated nilpotent group equipped with a 
generating $k$-tuple $(s_1,\dots,s_k)$ and a $k$-tuple of positive reals
$a=(\alpha_1,\dots,\alpha_k)\in (0,\infty]^k$. 
Let $\mathfrak w=\mathfrak w(a)$ be the two-dimensional weight system which assigns weight 
$w_i=(v_{i,1},v_{i,2})$ to $s_i$ where 
$$v_{i,1}=\frac{1}{\tilde{\alpha_i}},\;\;\tilde{\alpha}_i=\min\{2,\alpha_i\}$$
and
$$v_{i,2}= 0 \mbox{ unless } \alpha_i=2 \mbox{ in which case } v_{i,2}=1/2.$$
Then
$$\mu_{S,a}^{(n)}(e)\le C_{S,a} n^{-D_1(S,\mathfrak w)}
[\log (e+n)]^{-D_2(S,\mathfrak w)}$$
where 
$$D_i(S,\mathfrak w)=\sum_h \bar{v}_{h,i} 
\mbox{ \em rank}(G^\mathfrak w_h/G^\mathfrak w_{h+1}), \;\;\bar{w}_h=(\bar{v}_{h,1},\bar{v}_{h,2}).$$
\end{theo}
\begin{proof}The proof is the same as for 
Theorem \ref{th-up1} but uses a refined weight system and the 
associated weight function system  $\mathfrak F(a)$ 
where the function $F_c$ associated to a commutator of weight $v(c)=(v_1,v_2)$
is $F_c(r)=r^{v_1} [\log(e+r)]^{v_2}$.
\end{proof}
\begin{rem}\label{rem-th-up2} Referring to Theorem \ref{th-up2}, 
let $\Sigma$ be a sequence of formal commutators as in Theorem \ref{th-coord2F}
applied to $S, \mathfrak w, \mathfrak F (a)$.
Assume that 
for any $i$ such that
$s_i\in \mbox{core}(\mathfrak w,S, \Sigma)$, we have $\alpha_i=2$. 
Then $D_1(S,\mathfrak w)=D_2(S,\mathfrak w)=D(G)/2$ and
$$\mu_{S,a}^{(n)}(e)\le C_{S,a} [n \log n]^{-D(G)/2}.$$
\end{rem}

\begin{exa}   Let $G$ be the group of $4$ by $4$ unipotent 
upper-triangular matrices
$$G=\left\{\left(\begin{array}{cccc} 1 & x_{1,2} & x_{1,3} & x_{1,4}\\
0& 1& x_{2,3}& x_{2,4} \\
0& 0& 1& x_{3,4}\\
0& 0& 0& 1\end{array}\right): x_{i,j}\in \mathbb Z\right\}.$$
With obvious notation, let $X_{i,j}$ be the matrix in $G$ with a $1$ in position $i,j$ and all other non-diagonal entries equal to $0$. Consider the generating $4$-tuple
$$S=(s_1=X_{1,2},s_2=X_{2,3}, s_3=X_{3,4},s_4=X_{1,4}).$$
The non-trivial brackets are 
$$[X_{1,2},X_{2,3}]
=X_{1,3}, [X_{2,3},X_{3,4}]=X_{2,4}, [X_{1,2},X_{2,4}]=
[X_{1,3},X_{3,4}]=X_{1,4}.$$
Let $a=(1,2,5,1/3)$. 
The $2$-dimensional weight system $\mathfrak w$ is generated by
$w(s_1)=(1,0),w(s_2)=(\frac{1}{2},\frac{1}{2}),
w(s_3)=(\frac{1}{2},0), w(s_4)=(3,0)$. This implies
$$\textstyle 
w([X_{1,2},X_{2,3}])=(\frac{3}{2},\frac{1}{2}),
w([X_{2,3},X_{3,4}])= (1,\frac{1}{2}),$$
$$\textstyle 
w([X_{1,2},[X_{2,3},X_{3,4}]
])=(2,\frac{1}{2}),
w([[X_{1,2},X_{2,3}]
,X_{3,4}])=(2,\frac{1}{2}).$$
Ignoring (as we may) the weight values that would obviously lead to trivial 
quotients $G^\mathfrak w_h/G^\mathfrak w_{h+1}$, we have
$\bar{w}_1=(\frac{1}{2},0),
\bar{w}_2=(\frac{1}{2},\frac{1}{2})$, 
$\bar{w}_3=(1,0)$,
$\bar{w}_4=(1,\frac{1}{2})$, 
$\bar{w}_5=(\frac{3}{2},\frac{1}{2})$, $\bar{w}_6=(2,\frac{1}{2})$ 
and $\bar{w}_7=(3,0)$. Next we compute the  groups $G^\mathfrak w _i$. We have
\begin{eqnarray*}G^\mathfrak w_7=G^\mathfrak w_6 =<X_{1,4}> &\subset& 
G^\mathfrak w_5=
<X_{1,4},X_{1,3}>\\  
& \subset& G^\mathfrak w_4 
= <X_{1,4},X_{1,3},X_{2,4}>\\ &\subset& G^\mathfrak w_3
=<X_{1,4},X_{1,3},X_{2,4},X_{1,2}>\\
&\subset& G^\mathfrak w_2= <X_{1,4},X_{1,3},X_{2,4},X_{1,2},X_{2,3}>\\
&\subset& G^\mathfrak w_1= <X_{1,4},X_{1,3},X_{2,4},X_{1,2},X_{2,3},X_{3,4}>=
G.
\end{eqnarray*}
This gives
$$D_1(S,\mathfrak w)=\frac{1}{2}+
\frac{1}{2}+ 1+1+\frac{3}{2}+ 3=\frac{15}{2}$$ 
and 
$$D_2(S,\mathfrak w)=0+\frac{1}{2}+0+\frac{1}{2}+\frac{1}{2}+0=\frac{3}{2}.$$
We believe that the associated upper bound
$\mu_{S,a}^{(n)}(e)\le Cn^{-15/2}[\log n]^{-3/2}$ is sharp but, at this writing,
 we are not able to obtain a matching lower bound. 
\end{exa}

As a corollary of Theorem \ref{th-up2}, we can prove Theorem 
\ref{th-faster}. The bracket length $\ell(g)$ of an element of $G$ is defined 
just before Theorem \ref{th-faster}.
\begin{cor} Referring to {\em Theorem \ref{th-up2}}, 
assume that $S$ and $a$ are such that there exists $i\in \{1,\dots,k\}$
with the property that 
$$(\alpha_i,\ell(s_i))=(2,1) \mbox{ or } \;\;\alpha_i \ell(s_i)<2.$$
Then   
\begin{equation}
 \lim_{n\ra \infty} n^{D(G)/2}\mu_{S,a}^{(n)}(e)=0
\end{equation}
where $D(G)=\sum j\; \mbox{\em rank}(G_j/G_{j+1})$ where $G_j$ is the 
lower central series of $G$.
\end{cor}
\begin{proof} Pick $i_0$ among those $i\in \{1,\dots, k\}$ such that 
$(\alpha_i,\ell(s_i))=(2,1)$ or $\alpha_i \ell(s_i)<2$ so that  
$\alpha_{i_0}$ is smallest possible.  Let $\mathfrak w'=\mathfrak w(a)$ be the 
$2$-dimensional weight system introduced in Theorem \ref{th-up2} and 
let $\mathfrak F'=\mathfrak F(a)$ be the  weight function 
system appearing in the proof of Theorem \ref{th-up2}.  Let 
$\mathfrak w$ be the weight system that assigns weight $(1/2,0)$ to every 
$s_i\in S$ with weight function $F_{s_i}=(1+r)^{\frac{1}{2}}$.

If $\alpha_{i_0}<2/\ell(s_{i_0})$ then
by Theorem \ref{th-Vcomp} shows that  
$D_1(S,\mathfrak w')> D(S,\mathfrak w)=D(G)/2$. 
If $\alpha_{i_0}=2$ then we must have $\ell(s_{i_0})=1$.
This time, it follows that
$D_2(S,\mathfrak w')\ge 1/2> D_2(S,\mathfrak w)=0$. 
In both case, Theorem \ref{th-up2} show that  
$\mu_{S,a}^{(n)}(e)=o(n^{-D(G)/2})$ as desired.
\end{proof}

The next statement illustrates the use of a weight system $\mathfrak w$ 
and weight-functions system $\mathfrak F$ 
that are not tightly connected to each other
(including cases when the weight functions $F_c$ cannot be order in a 
useful way). 
\begin{theo} \label{th-up3} 
Let $G$ be a finitely generated nilpotent group equipped with a 
generating $k$-tuple $(s_1,\dots,s_k)$. Assume that
$\mu$ is a probability measure on $G$ of the form {\em(\ref{def-mutrunc})}
with 
$$\mu_i(n)= \kappa_i (1+|n|)^{-\alpha_i-1} \ell_i(|n|),\;\;1\le i\le k,$$
where each $\ell_i$ is a positive slowly varying function satisfying 
$\ell_i(t^b)\simeq \ell_i(t)$ for all $b>0$ and $\alpha_i\in (0,2)$.
Let $\mathfrak w$ be the power weight system associated with 
$a=(\alpha_1,\dots,\alpha_k)$ by setting $w_i=1/\alpha_i$. Let $(c_i)_1^t$ be a t-tuple of formal 
commutators such that for each $h$, the family $\{c_i: w(c_i)=\bar{w}_h\}$
projects to a linearly independent family in 
$G^\mathfrak w_h/G^\mathfrak w_{h+1}$.
Let $(s^{\pm 1}_{i_j})_{j=1}^N $ be the list of all the letters (with multiplicity) used in the build-words for the commutators $c_i$, $1\le i\le t$. Then
$$\mu^{(n)}(e)\le C n^{-D(S,\mathfrak w)} L(n)^{-1}$$
where
$$D(S,\mathfrak w)= \sum_h \bar{w}_h \mbox{\em rank}(G^\mathfrak w_h/G^\mathfrak w_{h+1})
\mbox{ and }\;\;
L(n)=\prod_1^N \ell_{i_j}(n)^{1/\alpha_{i_j}}.$$
\end{theo}
Note that this theorem does not offer one but many upper bounds. For each $n$, 
one can choose the commutator sequence $(c_i)_1^t$ so as to maximize 
the size of the resulting $L(n)$. 
\begin{exa} Consider the Heisenberg group
$$G=\left\{\left(\begin{array}{ccc} 1 & x & z \\
0& 1& y \\
 0& 0& 1\end{array}\right): x,y,z\in \mathbb Z\right\},$$
with generating $3$-tuple $S=(X,Y,Z)$ where $X$ is the matrix with 
$x=1,y=z=0$ and $Y,Z$ a defined similarly. Let $a=(\alpha_1,\alpha_2,
\alpha_3)\in (0,2)$
and let $\ell_1\equiv 1,\ell_2,\ell_3$ be slowly varying functions such that
$\ell_2 \le \ell_3 $ if and only if $n\in \cup_k [n_{2k},n_{2k+1}]$
for some increasing sequence $n_k$ tending to infinity. We 
also assume that $\ell_2,\ell_3$ satisfy $\ell_i(t^b)\simeq \ell_i(t)$ 
for all $b>0$. Applying Theorem \ref{th-up3}, we obtain:

\begin{itemize}
\item If $\frac{1}{\alpha_3}<\frac{1}{\alpha_1}+\frac{1}{\alpha_3}$ then we have
$$\mu^{(n)}(e)\le  C n^{-2(\frac{1}{\alpha_1}+\frac{1}{\alpha_2})} \ell_2(n)^{-\frac{2}{\alpha_2}}.
$$
\item If $\frac{1}{\alpha_3}>\frac{1}{\alpha_1}+\frac{1}{\alpha_3}$ then we have
$$\mu^{(n)}(e)\le  C n^{-\sum_1^3\frac{1}{\alpha_i}} 
\ell_2(n)^{-\frac{1}{\alpha_2}} \ell_3(n)^{-\frac{1}{\alpha_3}}.
$$
\item Finally, if $\frac{1}{\alpha_3}=\frac{1}{\alpha_1}+\frac{1}{\alpha_3}$, we have
$$\mu^{(n)}(e)\le  C n^{-\frac{2}{\alpha_3}}
\left\{\begin{array}{ll} \ell_2(n)^{-\frac{2}{\alpha_2}}& \mbox{ if } n\in \cup_k [n_{2k-1},n_{2k}]\\
 \ell_2(n)^{-\frac{1}{\alpha_2}}\ell_3(n)^{\frac{1}{\alpha_3}}
& \mbox{ if } n\in \cup_k [n_2k,n_{2k+1}].\\
\end{array}\right.$$
\end{itemize}
\end{exa}

\begin{exa}[continuation of Example {\ref{exa-illustr}-\ref{exa-illustr2}}]
Consider again the Heisenberg group with $S=(s_1=X,s_2=Y,s_3=Z)$.
Set $F_1(r)=r^{3/2}$, $F_2(r)=r^2\log (e+r), F_3(r)=r^\gamma$ with $\gamma> 3/2.$Let $\mu$ be the probability measure which assigns to $s_i^n$, $i=1,2,3$, $n\in \mathbb Z$ a probability proportional to
$\frac{1}{(1+|n|F^{-1}_i(|n|))}$. 
Namely, 
$$\mu(g)=\frac{1}{3}\sum_{i=1}^3\sum_{n\in \mathbb Z}
\mu_i(n)\mathbf 1_{s_i^n}(g),\;\; \mu_i(n)=\frac{c}{1+|n|F_i^{-1}(|n|)}.$$
Referring to the notation (\ref{def-Gi})(\ref{tsmhyp}), we have
\begin{eqnarray*}
\mathcal G_1(n)&\simeq & (1+n)^{2-(2/3)},\;\;
\tilde{\alpha}_1=2/3,\;\; L_1\equiv 1,\\
\mathcal G_2(n)&\simeq & (1+n)^{2-(1/2)}[\log(e+n)]^{-1/2},\;\;
\tilde{\alpha}_2=1/2, \;\;L_2(n)\simeq  [\log (e+n)]^{-1/2}\\
\mathcal G_3(n)&\simeq &(1+n)^{2-1/\gamma},\;\;
\tilde{\alpha}_3=1/\gamma, \;\;L_3\equiv 1.
\end{eqnarray*}
Apply Theorem \ref{th-up3} with $\alpha_i=\tilde{\alpha}_i$, $\ell_i=L_i$.
If $\gamma \in (3/2, 7/2]$, use the sequence of formal commutators
$(c_1=s_1,c_2=s_2,c_3=[s_1,s_2])$. If $\gamma>7/2$, 
use the sequence of formal commutators $(c_1=s_1,c_2=s_2,c_3=s_3)$ instead.
This gives
$$\mu^{(n)}(e)\le C\left\{\begin{array}{cc} (1+n)^{-7}[\log(e+ n)]^{-2}
&\mbox{ if } \gamma\in (3/2,7/2]\\
(1+n)^{-(7/2)-\gamma}[\log(e+ n)]^{-1}
&\mbox{ if } \gamma>7/2.\\
\end{array}\right.$$
Below, we will prove a matching lower bound.
\end{exa}

\section{Norm-radial measures and return probability lower bounds}\label{sec-low}
\setcounter{equation}{0}
The aim of this section is to provide lower bounds for the return probability
for the random walk driven by the measure $\mu_{S,a}$ on a nilpotent group $G$,
that is, lower bounds on $\mu_{S,a}^{(n)}(e)$. These lower bounds are 
obtained via comparison with appropriate norm-radial measures.

\subsection{Norm-radial measures} \label{radial}
A (proper) 
norm $\|\cdot\|$ on a countable 
group $G$ is a function $g\mapsto \|g\|\in [0,\infty)$
such that $\|g\|=0$ if and only if $g=e$, $\#\{g\|\le r\}$ is finite for all 
$r>0$, $\|g\|=\|g^{-1}\|$
and $\|g_1g_2\|\le \|g_1\|\|g_2\|$. If the triangle inequality is 
replaced by the weaker property that there exists $K$ such that
$\|g_1g_2\|\le K\|g_1\|\|g_2\|$, we say that $\|\cdot\|$ is a quasi-norm.

The associated left-invariant distance
is obtained by setting $d(g_1,g_2)=\|g_1^{-1}g_2\|$.  A norm is  
$\kappa$-geodesic
if for any element $g\in G$  there is a sequence 
$g_1,\dots, g_N$ with $N\le \kappa \|g\|$ such that  $\|g_i^{-1}g_{i+1}\|\le 
\kappa$. 

A simple observation is that any two $\kappa$-geodesic proper norms 
$\|\cdot\|_1,\|\cdot\|_2$ are comparable in the sense that there 
is a constant $C\in (0,\infty)$ such that
$$C^{-1}\|g\|_1\le\|g\|_2\le C\|g\|_1.$$ 

The word-length norm associated to any finite symmetric generating set is a 
proper $1$-geodesic norm. Most of the quasi-norms that we will consider below 
are not $\kappa$-geodesic. In general, they are not norms but only quasi-norms. 

\begin{theo} \label{th-BGK}
Let $G$ be a countable group. Let $\|\cdot\|$ be a norm on $G$ 
such that 
$$\forall r\ge 1, \;\;V(r)=\#\{g: \|g\|\le r\}\simeq r^D$$
for some $d>0$.
Fix $\gamma\in (0,2)$ and set
$$\nu_\gamma(g)= \frac{C_\gamma}{(1+\|g\|)^\gamma V(\|g\|)},\;\;
C_\gamma^{-1}=\sum _g\frac{1}{(1+\|g\|)^\gamma V(\|g\|)}.$$
Then we have
\begin{equation}\label{norm1}
\forall\,n\in \mathbb N,\;\;\nu_\gamma^{(n)}(e)
\simeq c n^{-D/\gamma}.\end{equation}
\end{theo}
\begin{rem} This is a subtle result in that, as stated, it depends very 
much on the fact that $\|\cdot\|$ is norm versus a quasi-norm. 
Indeed, the lower bound in (\ref{norm1}) 
is false if $\gamma\ge 2$ and the only thing that prevents us to apply the 
result to $\|\cdot\|^\theta$ with $\theta>1$ is that, in general,
$\|\cdot \|^\theta$ is only a quasi-norm when $\theta>1$.  
However, by Theorem \ref{th-PSC1}, (\ref{norm1}) 
holds true for any measure $\nu$ such that $\nu\simeq \nu_\gamma$.
 \end{rem}
\begin{rem}Definition \ref{defin-FQN} provides a great variety 
of examples of norms to which Theorem \ref{th-BGK} applies.
\end{rem}

\begin{proof} The probability of return $\nu_\gamma^{(n)}(e)$ 
behaves in the same way as the  probability of return of the associated 
the continuous time jump process. For the continuous time jump process, the result follows from \cite{BGK}. 
\end{proof}

\subsection{Comparisons between  $\mu_{S,a}$ and radial measures}

Let $G$ be a countable group. Let $\|\cdot\|$ be a quasi-norm on $G$. 
Set $$\forall r\ge 1, \;\;V(r)=\#\{g: \|g\|\le r\}.$$
Let $\phi:[0,\infty)\ra (0,\infty)$ be continuous. 
Consider the following hypotheses:
\begin{equation}\label{Ndoub}
\exists\,C,\;\;\forall\,r\ge 0,\;\;V(2r)\le CV(r);
\end{equation}
\begin{equation}\label{phidoub}
\exists \,C,\;\;
\forall\, \lambda\in (1/2,2),\;t\in (0,\infty), \;\;\phi(t)\le C \phi (\lambda t);
\end{equation}
and
\begin{equation}\label{phisum}
\sum_g\frac{1}{\phi(\|g\|)V(\|g\|)}<\infty.
\end{equation}

\begin{lem} \label{lem-compN1}
Assume {\em (\ref{Ndoub})-(\ref{phidoub})-(\ref{phisum})}.
For each $n\in \mathbb Z$, let  $g_n\in G$ and 
$\Lambda_n\subset G$ be such that:
\begin{enumerate}
\item $g\in \Lambda_n \Longrightarrow  \|g^{-1}g_n\|\le C \|g_n\| 
\mbox{ and } \|g\|\le C\|g_n\|$
\item $V(\|g_n\|)\le C n\# \Lambda_n$  
\item  $\forall\,g\in G,\;\;\#\{n: g\in \Lambda_n\}\le C$ and 
$\#\{n: g\in g_n^{-1}\Lambda_n\}\le C$.
\end{enumerate}
Then there is a constant $C_1$ such that
$$\sum_{n\in \mathbb Z}\sum_{x\in G}
\frac{|f(xg_n)-f(x)|^2}{(1+n)\phi(\|g_n\|)}\le
C_1 \sum_{x,g\in G}
\frac{|f(xg)-f(x)|^2 }{\phi(\|g\|)V(\|g\|)}.$$
\end{lem}
\begin{proof} Using 2,1 and 3 successively, write
\begin{eqnarray*}
\lefteqn{\sum_n\sum_x
\frac{|f(xg_n)-f(x)|^2}{(1+n)\phi(\|g_n\|)} \le
C \sum_n\sum_x\frac{|f(xg_n)-f(x)|^2\# \Lambda_n}{\phi(\|g_n\|)V(\|g_n\|)}}&&\\
&\le &
2C 
\sum_n \sum_{g\in\Lambda_n}\sum_x (
|f(xg_n)-f(xg)|^2+|f(xg)-f(x)|^2) \frac{1}{\phi(\|g_n\|)V(\|g_n\|)}\\
&\le &
C' 
\sum_n \sum_{g\in\Lambda_n}\sum_x \left(
\frac{|f(xg^{-1}g_n)-f(x)|^2 }{\phi(\|g^{-1}g_n\|)V(\|g^{-1}g_n\|)}
+
\frac{|f(xg)-f(x)|^2 }{\phi(\|g\|)V(\|g\|)}\right)\\
&\le &C'' 
\sum_{x,g} 
\frac{|f(xg)-f(x)|^2 }{\phi(\|g\|)V(\|g\|)}.
\end{eqnarray*}
\end{proof}
\begin{rem}\label{rem-sumphi}
Note that under the hypotheses of Lemma \ref{lem-compN1}, we have
$$\sum\frac{1}{(1+n)\phi(\|g_n\|)}<\infty.$$
\end{rem}
The next lemma will allow us to apply Lemma \ref{lem-compN1} in the context 
of Theorem \ref{th-coord2F}.  Assume that $G$ is a nilpotent group generated 
by the $k$ -tuple $(s_1,\dots,s_k)$. In addition, we are given a weight 
system $\mathfrak w$ and weight functions $F_c$  such that (\ref{F1})-(\ref{F2})
holds.  Observe that for any commutators 
$c,c'$, we have
\begin{equation} \label{Fsum} \forall\, r_1,r_2\ge 1,\;\;
F_{c'}\circ F^{-1}_c
(r_1+r_2)\simeq F_{c'}\circ F^{-1}_c(r_1)+
F_{c'}\circ F^{-1}_c(r_2).\end{equation}
Indeed, it follows from our hypotheses that  $F_{c'}\circ F^{-1}_c$ is an increasing doubling function.

\begin{lem} \label{lem-compN2}
Referring to the setting of {\em Theorem \ref{th-coord2F}},
fix $h\in \{1,\dots,q\}$, $i\in \{m_{h-1}+1,\dots m_{h-1}+R_h\}$ and an 
integer $u$. For each $n\in \mathbb Z$, let $z_n\in G^\mathfrak w_{h+1}$  
with  $\|z_n\|_{\mathfrak F,\mbox{\em \tiny com}}\le F_{c_1}\circ 
F_{c_i}^{-1}(n)$. 
Set  
$$g_n=\pi(c^{un}_i)z_n \in G$$
and 
$$\Lambda_n=\left\{g=\pi\left(\prod_1^q\prod_{m_{h-1}+1}^{m_{h-1}+R_h} c^{x_j}_j
\right) 
: |x_j|\le F_{c_j}\circ F_{c_i}^{-1}(n),\; x_i=\lfloor\frac{un}{2}\rfloor
.\right\}.$$
Then $(g_n)$ and $(\Lambda_n)$ satisfy the hypotheses 1,2 and 3 of 
{\em Lemma \ref{lem-compN1}}.
\end{lem}
\begin{proof} By Proposition \ref{pro-gn} and Theorem \ref{th-coord2F}, 
$\|g_n\|_{\mathfrak F,\mbox{\tiny com}}\simeq F_{c_1}\circ F^{-1}_{c_i}(n)$
and $g\in \Lambda_n$ implies 
$$ \|g\|_{\mathfrak F,\mbox{\tiny com}}\leq CF_{c_1}\circ F^{-1}_{c_i}(n),$$
so, Property 1 in Lemma \ref{lem-compN1} is satisfied. Property 2 also 
follows from Theorem \ref{th-coord2F} and the proof of Theorem \ref{th-VolQ}.

Suppose that $ g \in \Lambda_n \cap \Lambda_m$. Then, computing modulo 
$G^\mathfrak w_{h+1}$ and using the fact that $ 
[G^\mathfrak w_h,G^\mathfrak w_h]\subset G^\mathfrak w_{h+1}$
we obtain that  $\lfloor un/2\rfloor=\lfloor um/2\rfloor$.
Similarly, $g\in  g_n^{-1}\Lambda_n \cap g_m^{-1}\Lambda_m$ implies
 $n+\lfloor un/2\rfloor=m+\lfloor um/2\rfloor$. 
In both cases we must have $|n-m|\le 1$. This shows that Property 3 of Lemma \ref{lem-compN1} is satisfied. 
\end{proof}

The main result of this section is the following theorem.
\begin{theo} \label{th-compN}
Let $G$ be  a nilpotent group with generating the $k$-tuple 
$S=(s_1,\dots,s_k)$.  Let $I_{\mbox{\em \tiny tor}}=\{i\in\{1,\dots,k\}: 
s_i \mbox{ is torsion in } G\}$. Fix a weight 
system $\mathfrak w$ and a weight-function system  $\mathfrak F$  such that 
{\em (\ref{F1})-(\ref{F2})} are satisfied. Let 
$\|\cdot\|=\|\cdot\|_{\mathfrak F,\mbox{\em \tiny com}}$ 
be the associated quasi-norm introduced in {\em Definition \ref{defin-FQN}}.
For each $i\in \{1,\dots,k\}\setminus I_{\mbox{\em \tiny tor}}$, let 
$$h_i=j_\mathfrak w(s_i).$$
 Let  $\phi$ be such that 
{\em (\ref{phidoub})-(\ref{phisum})} are satisfied.

Let  $\mu$ be a probability measure on $G$ of the form
$$\mu(g)=\frac{1}{k}\sum_{j=1}^k\sum_{n\in \mathbb Z}\mu_i(n)
\mathbf 1_{s_i^n}(g)$$
where $\mu_i$ is an arbitrary symmetric probability measure on $\mathbb Z$ if 
$i\in I_{\mbox{\em \tiny tor}}$ and
$$\mu_i(n)= \frac{C_i}{(1+n)\phi(F_{c_1}\circ \mathbf F_{h_i}^{-1}(n))},
\;\; C_i^{-1}=\sum_n 
\frac{1}{(1+n)\phi(F_{c_1}\circ \mathbf F^{-1}_{h_i}(n))},$$
for $ i \in \{1,\dots, k\}\setminus I_{\mbox{\em \tiny tor}}.$
\if Assume that there exists $C_1$ such that
for each $i$ such that $s_i$ is not a torsion element in $G$  we have
$$\mathcal E_{s_i,\mu_i}(f,f)\le  C_1\sum_{g\in G}\sum_{n\in \mathbb Z}
\frac{|f(gs_i^n)-f(g)|^2}{n\phi(\|s_i^n\|)}.$$
\fi
Then there exists $C$ such that
$$\mathcal E_{\mu}(f,f)\le C \mathcal E_\nu (f,f)$$
where 
$$\nu(g)= \frac{C_\phi}{\phi(\|g\|)V(\|g\|)},\;\; 
C_\phi^{-1}=\sum_g\frac{1}{\phi(\|g\|)V(\|g\|)}.
$$ 
In particular, there are constants $c>0$ and $N$ such that 
$$\mu^{(2n)}(e)\ge  c \nu^{(2Nn)}(e).$$
\end{theo}
\begin{proof} Fix $i$ and write $s=s_i$.  By Definition \ref{def-jw},
either $s$ is a torsion 
element and $s^\kappa=e$ for some $\kappa$ or $j_\mathfrak w(s)=h<\infty$. 
In the second case we can find $\kappa $ such that
$$s^\kappa = \pi(\prod_{m_{h-1}+1}^{m_{h-1}+ \rho} c_i^{x_i})z, \;x_{m_{h-1}+\rho}\neq 0,\;\; z\in G^\mathfrak w_{h+1}.$$  
If $s$ is torsion, it is very easy to see that
$\mathcal E_{s,\mu_i}(f,f)\le C \mathcal \nu(f,f)$.
In the course of this proof, $C$ denotes a generic constant that may change 
from line to line. If $s$ is not torsion and 
$$s^\kappa = \pi(\prod_{m_{h-1}+1}^{m_{h-1}+ \rho} c_i^{x_i})z, 
\;x_{m_{h-1}+\rho}\neq 0,\;\; z\in G^\mathfrak w_{h+1},$$
set $F=F_{c_{m_{h-1}+1}}$ (we have $F\simeq F_{c_j}$, 
$j\in \{m_{h-1}+1,m_h\}$). Then, for any $n$, we have
$$s^{\kappa n} = \pi(\prod_{m_{h-1}+1}^{m_{h-1}+ \rho} c_i^{x_i n})z_n
\mbox{ with } \|z_n\|\le CF_{c_1}\circ F^{-1} (|n|),\;\; z_n\in G^\mathfrak w_{h+1}.$$  

Now, write $n= \kappa u_n+v_n$ with $|v_n|<\kappa$ and
$$\sum_g|f(gs^n)-f(g)|^2\le 2 (\sum_g|f(g s^{\kappa u_n})-f(g)|^2+
\sum_g|f(gs^{v_n})-f(g)|^2).$$
By Lemma \ref{lem-compN2} and Remark \ref{rem-sumphi}, the 
hypotheses of Theorem \ref{th-compN} imply that 
$\sum ((1+n)\phi(\|s^n\|))^{-1}<\infty$. Hence, 
it is is easy to check that
\begin{equation}\label{phi0}
\sum_g\sum_n\frac{|f(gs^{v_n})-f(g)|^2}{(1+n)\phi(\|s^n\|)}\le C \mathcal E_\nu(f,f).
\end{equation}
Consequently, it suffices to show that 
$$\sum_g\sum_n\frac{|f(gs^{\kappa u_n})-f(g)|^2}{(1+n) \phi(\|s^n\|)}
\le C \mathcal E_\nu(f,f).$$
We have  $\|s^n\|\simeq \|s^{\kappa u_n}\|\simeq F_{c_1}\circ F^{-1}(\kappa u_n)$. Hence
\begin{equation}\label{phi1}
\sum_g\sum_n\frac{|f(gs^{\kappa u_n})-f(g)|^2}{(1+n)\phi(\|s^n\|)} \le 
C\sum_g\sum_\ell \frac{|f(gs^{\kappa \ell})-f(g)|^2}{\ell \phi(F_{c_1}\circ
F^{-1}(\ell))}.
\end{equation}
Next, set $i_1=m_{h-1}+1,i_2=m_{h-1}+\rho$ and write
\begin{eqnarray*}
\lefteqn{\sum_g\sum_\ell |f(gs^{\kappa \ell})-f(g)|^2 }&&\\
&\le &\rho\left(\sum_g\sum_\ell  \sum_{i=i_1}^{i_2-1}
 |f(g\pi(c_i^{x_i \ell}))-f(g)|^2
+\sum_g\sum_\ell |f(g\pi(c_{i_2}^{x_{i_2} \ell})z_\ell)-f(g)|^2\right).
\end{eqnarray*}
By Lemmas \ref{lem-compN1}-\ref{lem-compN2}, for each  $i=i_1,\dots,i_2-1$,
we have
$$
\sum_g\sum_\ell 
 \frac{|f(g\pi(c_i^{x_i \ell}))-f(g)|^2}{(1+\ell) \phi(\|\pi(c_i^{x_i \ell})\|)}
\le C  \mathcal E_\nu (f,f)$$
and, since $z_\ell \in G^\mathfrak w _{h+1}$ and $\|z_\ell\|\le CF_{c_1}\circ 
F^{-1}(\ell)$, 
$$\sum_g\sum_\ell \frac{|f(g\pi(c_{i_2}^{x_{i_2} \ell})z_\ell)-f(g)|^2}{(1+\ell)
\phi(\|\pi(c_{i2}^{x_{i_2} \ell})z_\ell\|)} 
\le C  \mathcal E_\nu (f,f).$$ 
Further, for each $i=i_1,\dots,i_2$ with $x_i\neq 0$, we have
$$\|\pi(c_i^{x_i\ell})\|\simeq F_{c_1}\circ F^{-1}(\ell)$$
as well as $\|\pi(c_i^{x_{i_2}\ell})z_\ell\|\simeq F_{c_1}\circ F^{-1}(\ell)$.
Hence (\ref{phi1}) and the above estimates give 
$$
\sum_g\sum_n\frac{|f(gs^{\kappa u_n})-f(g)|^2}{(1+n)\phi(\|s^n\|)} \le C \mathcal E
_\nu(f,f).$$
Together with (\ref{phi0}), this gives 
$$ \sum_{g\in G}\sum_{n\in \mathbb Z}
\frac{|f(gs^n)-f(g)|^2}{(1+n)\phi(\|s^n\|)}\le C \mathcal E_\nu(f,f).$$
Since this holds true for each $s=s_i$, $i=1,\dots,k$, 
the desired result follows.
\end{proof}

\subsection{Assorted corollaries: return probability lower bounds}
In this section we use the  comparison with  norm-radial measures to obtain
explicit lower estimates on $\mu_{S,a}^{(n)}(e)$. 
The simplest and 
 most important result of this type is as follows.

\begin{theo}\label{th-low1} 
Let $G$ be a finitely generated nilpotent group equipped with a 
generating $k$-tuple $(s_1,\dots,s_k)$ and a $k$-tuple of positive reals
$a=(\alpha_1,\dots,\alpha_k) \in (0,2)^k$. 
Let $\mathfrak w$ be the weight system which 
assigns weight $w_i=1/\alpha_i$ to $s_i$.
Then
$$\mu_{S,a}^{(n)}(e)\ge c_{S,a} n^{-D(S,\mathfrak w)}$$
where $D(S,\mathfrak w)=\sum_h \bar{w}_h \mbox{ \em rank}(G^\mathfrak w_h/G^\mathfrak w_{h+1}).$
\end{theo}
\begin{rem}This lower bound matches precisely the upper bound given by 
Theorem \ref{th-up1}.  Thus, as stated in Theorems 
\ref{th-main1}-\ref{th-main2}, 
for any $a\in (0,2)^k$, 
$$\mu_{S,a}^{(n)}(e)\simeq n^{-D(S,\mathfrak w)}.$$
Note however that, in Theorems \ref{th-main1}-\ref{th-main2}, the constraints 
on the $\alpha_i$'s is weaker.  This more general case will be treated below.
\end{rem}

\begin{proof} Fix a sequence $\Sigma=(c_i)_1^t$ of commutators as in Theorem 
\ref{th-coord2F} and let $\|\cdot\|$ be the associated norm 
$\|\cdot\|=\|\cdot\|_{\Sigma}$ introduced in Definition \ref{defin-FQN}.
Note that, by Remark \ref{rem-norm}, 
$\|\cdot\|$ is indeed not only a quasi-norm but a norm.  
By hypothesis, $1/w(c_1)<2$.
Hence Theorem \ref{th-BGK}, together with Theorem \ref{th-VolQ},
shows that the norm-radial measure
$$\nu(g)= \frac{C}{(1+\|g\|)^{1/w(c_1)}V(\|g\|)}$$
satisfies 
\begin{equation}
\label{nucrux}
\nu^{(n)}(e)\ge c n^{- w(c_1)D(S,\mathfrak w)/w(c_1)}=c n^{-D(S,\mathfrak w)}.\end{equation} 
Theorem \ref{th-compN} produces a symmetric measure $\mu$ such that
$\mathcal E _\mu\le C\mathcal E_\nu$. This measure $\mu$ is given by
$$\mu(g)=\frac{1}{k}\sum_{j=1}^k\sum_{n\in \mathbb Z}\mu_i(n)
\mathbf 1_{s_i^n}(g)$$
where $\mu_i$ is an arbitrary symmetric probability measure on $\mathbb Z$ if 
$i\in I_{\mbox{\tiny tor}}$ and
$$\mu_i(n)= \frac{C_i}{(1+n) 
(1+F_{c_1}\circ \mathbf F_{h_i}^{-1}(n))^{1/w(c_1)}}$$
with $$ C_i^{-1}=\sum_n 
\frac{1}{(1+n) (1+F_{c_1}\circ \mathbf F^{-1}_{h_i}(n))^{1/w(c_1)}}$$ for $\; i 
\in \{1,\dots, k\}\setminus I_{\mbox{\tiny tor}}.$
In the latter case, we have 
$\mathbf F_{h_i}(t)= 
t^{\bar{w}_{h_i}}$ with $\bar{w}_{h_i}\ge w(s_i)=1/\alpha_i$ and
$F_{c_1}(t)= t^{w(c_1)}$. Hence
$$\mu_i(n)\simeq  \frac{C_i }{(1+n)^{1+1/\bar{w}_{h_i}}}
\ge \frac{C'_i}{(1+n)^{1+\alpha_i}}
.$$
It follows that if we pick $\mu_i$ to be given by $\mu_i(n)=c_i(1+n)^
{-(1+\alpha_i)}$ for $i\in  I_{\mbox{\tiny tor}}$, and $\mu_i=c_i(1+n)^{1+1/\bar{w}_{h_i}}$ if $i\in I\setminus I_{\mbox{\tiny tor}}$,
 we obtain a measure $\mu$
such that
$$\mathcal E_{\mu_{S,a}}\le C \mathcal E_\mu\le C'\mathcal E_
\nu.$$ By Theorem \ref{th-PSC1}, this implies that there are $c,N\in (0,\infty)$
such that
$$ \mu_{S,a}^{(2n)}(e)\ge c \nu ^{(2nN)}(e).$$
Thus the  lower bound stated in Theorem \ref{th-low1} 
follows from (\ref{nucrux}).
\end{proof}

The following theorem extends the range of applicability of the 
previous result. In particular, the statement is different but equivalent to 
the statement recorded in Theorem \ref{th-main2}. See also Theorem 
\ref{th-low33} below.

\begin{theo}\label{th-low2} 
Let $G$ be a finitely generated nilpotent group equipped with a 
generating $k$-tuple $(s_1,\dots,s_k)$ and a $k$-tuple of positive reals
$a=(\alpha_1,\dots,\alpha_k) \in (0,\infty]^k$. Set $\tilde{\alpha}_i=\min\{\alpha_i,2\}$. Let $\mathfrak w$ be the weight system which assigns weight 
$1/\tilde{\alpha}_i$ to $s_i\in S$. Let $\Sigma$ be a sequence 
of formal commutators as in {\em Theorem \ref{th-coord2F}}. Assume that
$w(s)>1/2$ for all $s\in \mbox{\em core}(\mathfrak w,S,\Sigma)$. 
Then
$$\mu_{S,a}^{(n)}(e)\simeq  n^{-D(S,\mathfrak w)}.$$
\end{theo}
\begin{proof} The upper bound follows from Theorem \ref{th-up1}. 
The lower bound is more subtle.  Consider any $s\in S$ such that 
$w(s)=1/2$ (i.e., $s=s_i$ with $\alpha_i\ge 2$). Observe that $1/2$ is the lowest possible value for weights in $\mathfrak w$ and that the hypothesis that
$w>1/2$ on $\mbox{core}(\mathfrak w,S,\Sigma)$ implies that
 $G^{\mathfrak w}_1/G^{\mathfrak w}_2$ is a torsion group. In particular,
this implies that $\bar{w}_{j_{\mathfrak w}(s)}>1/2=w(s)$. By Corollary \ref{cor-normeq},
the weight system $\mathfrak w'$ generated by 
$$w'(s)=\left\{\begin{array}{cc} w(s) & \mbox{ if } w(s)\neq 1/2 \\
\bar{w}_2 &\mbox{ if }  w(s)=1/2\end{array}\right.$$ is such that 
$w(s)\le w'(s)\le \bar{w}_{j_{\mathfrak w}(s)}$ for all $s\in S$ and $w'(s)>1/2$
for all $s\in S$.  Now, Theorem \ref{th-compN} gives the comparison
$\mathcal E_{\mu_{S,a}}\le C\mathcal E_\nu$ with 
$$\nu(g)\simeq \frac{1}{ (1+ \|g\|_{\Sigma,\mathfrak w})^{1/w_\Sigma}
V_{\Sigma,\mathfrak w}(\|g\|_{\Sigma,\mathfrak w})}.$$
However, since the minimum weight value $w_\Sigma$ may be equal to $1/2$, 
we cannot apply Theorem \ref{th-BGK} directly.
We proceed as follows. By the definition of $w'$ and 
Corollary \ref{cor-normeq}, we have
$$\forall \,g\in G,\;\;\|g\|_{\Sigma,\mathfrak w}^{1/w_\Sigma}\simeq 
\|g\|_{S,\mathfrak w'}^{1/w'_S}.$$
Note that this implies that
$$ V_{\Sigma,\mathfrak w}(\|g\|_{\Sigma,\mathfrak w}) =\#\{g'\in G:\|g'\|_{\Sigma,\mathfrak w}\le \|g\|_{\Sigma,\mathfrak w}\} \simeq  V_{S,\mathfrak w'}(\|g\|_{S,\mathfrak w'}). $$
Hence we have
$$\mathcal E_\nu \simeq \mathcal E_{\nu'}$$
where
$$\nu'(g)\simeq  
\frac{1}{ (1+ \|g\|_{S,\mathfrak w'})^{1/w'_S}
V_{S,\mathfrak w'}(\|g\|_{S,\mathfrak w'})}.$$
Now, since by construction $w'_S>1/2$, we can apply Theorem \ref{th-BGK}
which gives $(\nu')^{(n)}(e)\simeq n^{-D(S,\mathfrak w')}=n^{-D(S,\mathfrak
w)}$. Also, we have
$\mathcal E_{\mu_{S,a}}\le C \mathcal E_\nu \simeq \mathcal E_{\nu'}$.
Hence
$$\mu_{S,a}^{(n)}(e)\ge c n^{-D(S,\mathfrak w)}.$$
This ends the proof of Theorem \ref{th-low2}.\end{proof}

Our next results provides  a comparison  between the behaviors of two measures 
$\mu_{S,a}$ and $\mu_{S',a'}$. Compare to Corollary \ref{cor-PhiG} and 
Theorem \ref{th-faster} which treats comparison with $\mu_{S',a'}$ when 
$a'=(\alpha'_i)_1^{k'}\in (2,\infty]^{k'}$, a case that is excluded in 
Theorem \ref{th-low3}.  
\begin{theo}\label{th-low3}
Let $G$ be a finitely generated nilpotent group equipped with a 
generating $k$-tuple $(s_1,\dots,s_k)$ and a $k$-tuple of positive reals
$a=(\alpha_1,\dots,\alpha_k) \in (0,\infty]^k$. 
Set $\tilde{\alpha}_i=\min\{\alpha_i,2\}$. Let $\mathfrak w$ be the weight system which assigns weight 
$1/\tilde{\alpha}_i$ to $s_i\in S$.  Fix another weight system  
$\mathfrak w'=(w'_1,\dots,w'_k)$ with minimal weight $w'_S> 1/2$.
Let $\Sigma$ be a sequence 
of formal commutators as in {\em Theorem \ref{th-coord2F}} for $(S,\mathfrak w')$.
Assume that 
$w(s)\ge w'(s)$ for all $s\in \mbox{\em core}(\mathfrak w',S,\Sigma)$.
Then
$$\mu_{S,a}^{(n)}(e)=o( n^{-D(S,\mathfrak w')})$$
if and only if there exists $s\in S$ such that $w(s)> \bar{w}'_{j_{\mathfrak w'}(s)}$.
\end{theo}
\begin{proof}Apply Theorems \ref{th-up1} and \ref{th-low2} together with 
Corollary \ref{cor-normeq} and Theorem \ref{th-Vcomp}. 
\end{proof}

\begin{theo}\label{th-low32} 
Let $G$ be a finitely generated nilpotent group equipped with a 
generating $k$-tuple $(s_1,\dots,s_k)$ and a $k$-tuple of positive reals
$a=(\alpha_1,\dots,\alpha_k) \in (0,\infty]^k$.  Set 
$\tilde{\alpha}_i=\min\{\alpha_i,2\}$.
Let $\mathfrak w$ be the weight system which 
assigns weight $w_i=1/\tilde{\alpha}_i$ to $s_i$.
Then there exists $A\ge 0$ such that
$$\mu_{S,a}^{(n)}(e)\ge c_{S,a} n^{-D(S,\mathfrak w)} [\log n]^{-A}.$$

Further, let $\Sigma$ be as in {\em Theorem \ref{th-coord2F}} applied
to $(S,\mathfrak w)$ and assume that $\alpha_i=2$ for all $i\in \{1,\dots,k\}$
such that $s_i\in \mbox{\em core}(S,\mathfrak w,\Sigma)$. Then
$$\mu_{S,a}^{(n)}(e)\simeq  [n\log n]^{-D(G)/2}.$$
\end{theo}
\begin{proof} The proof of the general lower bound
is essentially the same as for Theorem \ref{th-low1}, 
except that we cannot rule out the possibility that 
$w(c_1)=1/2$. If $w(c_1)>1/2$ 
then the previous proof applies  and 
we obtain  $\mu_{S,a}^{(n)}(e)\ge cn^{-D(S,\mathfrak w)}$ which is better 
than the statement we need to prove. If $w(c_1)=1/2$ then we have a comparison
\begin{equation}\label{compmunu0}
\mathcal E_{\mu_{S,a}} \le C \mathcal E_\nu
\end{equation}
with 
$$\nu(g)= \frac{C}{(1+\|g\|)^{2}V(\|g\|)}.$$
To conclude, we need a lower bound on $\nu^{(n)}(e)$. This turns out to be  
rather subtle and difficult question in the present generality.
In \cite{SCZ0} we show that there exists $A\ge 0$ such that 
\begin{equation}
\label{nucrux2}
\nu^{(n)}(e)\ge c n^{-D(S,\mathfrak w)}[\log n]^{-A}.\end{equation} 
This proves the desired lower bound on $\mu_{S,a}^{(n)}(e)$.

When  $\alpha_i=2$ for all $i\in \mbox{core}(S,\mathfrak w,\Sigma)$, it follows 
that 
$$D(S,\mathfrak w)=G(G)/2\; \mbox{ and }\;\;\|g\|\simeq |g|_{S}$$ 
where $|g|_S$ denotes the usual word-length of $g$ over 
the symmetric generating set $\{s_i^{\pm 1}: 1\le i\le k\}$. 
Theorem \ref{th-up2} provides the  upper bound 
$$\mu_{S,a}^{(n)}(e)\le C 
[n\log n]^{-D(G)/2}.$$
For the lower bound,
by the Dirichlet form inequality (\ref{compmunu0}), 
it suffices to bound $\nu^{(n)}(e)$ from below. Using the fact that 
$\|g\|\simeq |g|_{S}$, we prove in \cite{SCZ0} that,
in this special case, (\ref{nucrux2})
holds with $A=D(G)/2$.  This provides the desired matching lower bounds 
$$\mu_{S,a}^{(n)}(e)\ge c 
[n\log n]^{-D(G)/2}.$$
\end{proof}

\begin{theo}
\label{th-low33} 
Let $G$ be a finitely generated nilpotent group equipped with a 
generating $k$-tuple $(s_1,\dots,s_k)$ and a $k$-tuple of positive reals
$a=(\alpha_1,\dots,\alpha_k) \in (0,\infty]^k$. 
Set $\tilde{\alpha}_i=\min\{\alpha_i,2\}$ and 
$w_i=1/\tilde{\alpha}_i$. Let $\mathfrak w$ be the associated weight system.  
Let $\Sigma$ be as in {\em Theorem 
\ref{th-coord2F}} applied to $(S,\mathfrak w)$.
Let 
$$\Theta= (\theta_1=s_{i_1},\dots ,\theta_\kappa=s_{i,\kappa})=
\mbox{\em core}(S,\mathfrak w,\Sigma).$$
Let $H$ be the subgroup of $G$ generated by $\Theta$.
Set $b=(\beta_1=\alpha_{i_1},
\cdots,\beta_\kappa=\alpha_{i_\kappa})$,  $\tilde{\beta}_i=\tilde{\alpha}_{i_j}$,
$v(\theta_i)=w(s_{i_j})$.
Let $\mathfrak v$ be the weight system associated to $v$ on $(H,\Theta)$, respectively. Then
$$D(\Theta,\mathfrak v)=D(S,\mathfrak w).$$
In particular, letting $e_H, e_G$ be the identity elements in $H$ and $G$,
respectively, we have:
\begin{itemize}
\item if $\alpha_i \in (0,2)$ for all $i$ 
such that $s_i\in \mbox{\em core}(S,\mathfrak w,\Sigma)$ then
$$\mu_{S,a}^{(n)}(e_G)\simeq \mu_{\Theta,b}^{(n)}(e_H)\simeq 
n^{-D(\Theta,\mathfrak v)}.$$
\item if $\alpha_i=2$ for all $i$ 
such that $s_i\in \mbox{\em core}(S,\mathfrak w,\Sigma)$ then
$$\mu_{S,a}^{(n)}(e_G)\simeq \mu_{\Theta,b}^{(n)}(e_H)\simeq 
[n\log n]^{-D(H)/2}.$$
\end{itemize} 
\end{theo}
\begin{rem} One can easily prove that $H$ is a subgroup of finite index in $G$.
 It is also easy to prove by the direct comparison  
techniques of \cite{PSCstab} that
$$\forall\,n,\;\;\mu_{S,a}^{(2Kn)}(e_G)\le C \mu_{\Theta,b}^{(2n)}(e_H)$$
for some integer $K$ and constant $C$ and for each $a=(\alpha_1,\dots,\alpha_k)$.
The converse inequality seems significantly harder to prove although we conjecture it does hold true.
\end{rem}
\begin{proof}
First we  observe that  $D(\Theta,\mathfrak v)\le D(S,\mathfrak w)$.
Indeed, this follows immediately from the obvious fact that
$$\{g\in H: \|g\|^{1/v_\Theta}_{\Theta,\mathfrak v}\le r\}\subset 
\{g\in G: \|g\|^{1/w_S}_{S,\mathfrak w}\le r\}.$$
To prove that $D(\Theta,\mathfrak v)\ge D(S,\mathfrak w)$, it is convenient to
introduce the generating $k$-tuple $S^*=(s^*_i)_1^k$ of $H$ such that
$s^*_{i,j}=s_{i_j}$ if $s_{i_j}=\theta_j\in \Theta$, and $s^*_{i_j}=e$ 
otherwise. Both $S$ and $S^*$ are equipped with the weight system $\mathfrak w$.
Obviously, the non-decreasing sequence of subgroups $(H^\mathfrak w_j)$
is a trivial refinement of the sequence $(H^\mathfrak v_j)$ in the sense that 
the two sequences differ only by insertion of some repetitions. For instance, 
$A,B,C$ may become $A,A, B,B,B,B, C$.  It follows that 
$D(\Theta,\mathfrak v)=D(S^*,\mathfrak w)$. The notational advantage is that 
the weight system $\mathfrak w$ with increasing weight-value sequence 
$\bar{w}_j$ is now shared by $S$ and $S^*$. We wish to prove that
$$\mbox{rank}(H^\mathfrak w_j/H^\mathfrak w_{j+1})\ge 
\mbox{rank}(G^\mathfrak w _j/G^\mathfrak w_{j+1}).$$

The (torsion free) rank of an abelian group can be computed as the 
cardinality of a maximal free subset. Set $R=R^\mathfrak w_j$ be the torsion 
free rank of $G^\mathfrak w_j/G^\mathfrak w_{j+1}$. Let
$(c_{m_{j-1}+1},\dots, c_{m_{j-1}+R})$ 
be the formal commutators given by Theorem 
\ref{th-coord2F}  which form a maximal free subset of
$G^\mathfrak w_j/G^\mathfrak w_{j+1}$.  
By definition of $\mbox{core}(S,\mathfrak w,\Sigma)$, 
the images of these formal commutators in $G$ belong to $H$. In fact, 
they clearly belong to $H^\mathfrak w_j\subset G^\mathfrak w_j $.
Now, we also have $H^\mathfrak w_{j+1}\subset G^\mathfrak w_{j+1}$.
Assume that
$\prod_{m_{j-1}+1}^{m_{j-1}+R}c_i^{x_i}=e $ in  
$H^\mathfrak w_j/H^\mathfrak w_{j+1}$. 
Then, a fortiori, this product is trivial
in 
$$H^\mathfrak w_jG^\mathfrak w_{j+1}/G^\mathfrak w_{j+1} \simeq 
H^\mathfrak w_j/(H^\mathfrak w_j\cap G^\mathfrak w_{j+1})$$
since $(H^\mathfrak w_j\cap G^\mathfrak w_{j+1})\subset H^\mathfrak w_{j+1}$. 
In particular, this product must be trivial in 
$G^\mathfrak w_j/G^\mathfrak w_{j+1}$. This implies that $x_i=0$ for all $i$ 
so that  $H^\mathfrak w_j/H^\mathfrak w_{j+1}$ admits a free subset of size $R$.
It follows that $\mbox{rank}(H^\mathfrak w_j/H^\mathfrak w_{j+1})\ge R$ 
as desired. 
\end{proof}

To state the final result of this section, we need some preparation.
Consider the class of measure $\mu$ of the form (\ref{def-mutrunc})
with 
\begin{equation}\label{slow-li}
\mu_i(n)= \kappa_i (1+|n|)^{-\alpha_i-1} \ell_i(|n|),\;\;1\le i\le k,
\end{equation}
where each $\ell_i$ is a positive slowly varying function satisfying 
$\ell_i(t^b)\simeq \ell_i(t)$ for all $b>0$ and $\alpha_i\in (0,2)$.
Consider the weight-function system $\mathfrak F$ generated by letting $F_i$ 
be the inverse function of $r\mapsto r^{\alpha_i}/\ell_i(r)$. 
Note that $F_i$ is regularly varying of order $1/\alpha_i$ and that
$F_i(r)\simeq [r\ell_i(r)]^{1/\alpha_i}$, $r\ge 1$, $i=1,\dots,k$. 
We make 
the fundamental assumption that the functions $F_i$ have the property that
for any $1\le i,j\le k$, either $F_i(r)\le C F_j(r)$ of $F_j(r)\le CF_i(r)$.
For instance, this is clearly the case if all $\alpha_i$ are distinct.
Without loss of generality, we can assume that there exists a 
multidimensional weight system $\mathfrak w$, say of dimension $d$,
with 
$$w_i=(v_i^1,\dots,v_i^d), \;\;v_i^1=1/\alpha_i, \;\;1\le i\le k,$$ and such that
$\mathfrak w $ and $\mathfrak F$ are compatible in the sense that 
(\ref{F1})-(\ref{F2}) hold true.  Separately, consider also 
the one-dimensional weight system $\mathfrak v$ generated by $v_i=1/\alpha_i$, 
$1\le i\le k$. Note that one can check that
$$D(S,\mathfrak v)=\sum_j \bar{v}_j
R^\mathfrak v_j =\sum_j \bar{v}^1_iR^\mathfrak w_j $$  
where, by definition, $\bar{w}_j=(\bar{v}^1_j,\dots,\bar{v}^d_j)$.
Fix $\alpha_0\in (0,2)$ such that 
$$\alpha_0>\max\{\alpha_i: 1\le i\le k\}$$
and $\alpha_0/\alpha_i \not\in \mathbb N$, $i=1,\dots, k$.
Observe that there are convex functions $K_i\ge 0$, $i=0,\dots,k$, such that $K_i(0)=0$ and
\begin{equation}\label{conv} 
\forall r\ge 1,\;\; F_i(r^{\alpha_0}) \simeq K_i(r).
\end{equation}  
Indeed,  $r\mapsto F_i(r^{\alpha_0})$ is regularly varying of index $\alpha_0/\alpha_i$ with $1<\alpha_0/\alpha_i\not\in \mathbb N$. By 
\cite[Theorems 1.8.2-1.8.3]{BGT} there are smooth
positive convex functions $\tilde{K}_i$ such that  $\tilde{K}_i(r)\sim 
F_i(r^{\alpha_0})$. If $\tilde{K}_i(0)>0$, it is easy to construct a convex function
$K_i:[0,\infty)\rightarrow [0,\infty)$ such that 
$K_i\simeq \widetilde{K}_i$ on $[1,\infty)$ and $K_i(0)=0$.

\begin{theo} \label{th-eq3} 
Let $G$ be a finitely generated nilpotent group equipped with a 
generating $k$-tuple $(s_1,\dots,s_k)$. Assume that
$\mu$ is a probability measure on $G$ of the form {\em (\ref{def-mutrunc})}
with $\mu_i$ as in {\em (\ref{slow-li})}. 
Let $\ell_i$, $F_i$, $\mathfrak F, \mathfrak w, \mathfrak v$ be as 
described above. Let $(c_i)_1^t$ be a $t$-tuple of formal commutators as in 
{\em Theorem \ref{th-coord2F}} applied to $G,S, \mathfrak w,\mathfrak F$. 
Let $(s^{\pm 1}_{i_j})_{j=1}^N $ be the list of all the letters 
(repeated according to multiplicity) used in the build-words for 
the commutators $c_i$ with
$i\in \bigcup_j \{m_{j-1}+1,\dots,m_{j-1}+R^\mathfrak w_j\}$. Then
$$\mu^{(n)}(e)\simeq n^{-D(S,\mathfrak v)} L(n)^{-1}$$
where
$$L(n)=\prod_1^N \ell_{i_j}(n)^{1/\alpha_{i_j}}.$$
\end{theo}
\begin{proof} The upper bounds follows immediately from Theorem \ref{th-up3}.
For the lower bound, it is technically convenient to adjoint to $S$
the dummy generator $s_0=e$ with associated weight function 
$F_0(r)=r^{\alpha_0}$.  Let $\mathfrak W_0$, $\mathfrak F_0$ 
we the weight systems induced by $S_0=(e,s_1,\dots,s_k)$, $F_0,F_1,\dots,F_k$.

Apply Theorem \ref{th-compN} to 
$G,S,\mathfrak w_0,\mathfrak F_0$ to obtain that
$\mathcal E_\mu \le C\mathcal E_\nu$ where
$$\nu(g)\simeq \frac{1}{\|g\|_{\mathfrak F_0,\mbox{\tiny com}}^{\alpha_0} 
V_{\mathfrak F_0,\mbox{\tiny com}}
(\|g\|_{\mathfrak F_0,\mbox{\tiny com}})}$$
with $V_{\mathfrak F_0,\mbox{\tiny com}}(r)= 
\#\{g\in G: \|g\|_{\mathfrak F_0,\mbox{\tiny com}}\le r\}.$
By construction, 
$$\nu(g)\simeq \frac{1}{\|g\| V(\|g\|)} $$
where $\|\cdot\|$ is the norm $\|\cdot\|_{\mathfrak K,\mbox{\tiny com}}$
based on the convex function $K_i \simeq F_i(r^{\alpha_0})$ provided by (\ref{conv}) and $V$ denotes the associated volume function. 
Indeed, by construction we have 
$\|\cdot\|\simeq \|\cdot\|_{\mathfrak F_0,\mbox{\tiny com}}^{\alpha_0}
$.
As $\|\cdot\|$ is a norm, 
an extension of Theorem \ref{th-BGK} obtained in \cite{SCZ0} and 
which allows volume growth 
of regular variation with positive index gives
$$\nu^{(n)}(e)\simeq  \frac{1}{V(n)}\simeq 
\frac{1}{V_{\mathfrak F_0,\mbox{\tiny com}}(n^{1/\alpha_0})}\simeq 
\frac{1}{\#Q(S_0,\mathfrak F_0,n)}\simeq \frac{1}{\#Q(S,\mathfrak F,n)}
. $$
Using the notation introduced in Theorem \ref{th-eq3}, we have
$$\#Q(S,\mathfrak F,r)\simeq n^{D(S,\mathfrak v)}L(n)$$ which yields 
the desired result.
\end{proof}

\subsection{Near diagonal lower bounds}
In this section we use Lemma \ref{lem-1d}(ii) to 
turn the sharp {\em on diagonal lower bounds}
of the previous section into {\em near diagonal lower bounds}.
The key tool is the following lemma.

\begin{lem}\label{lem-HSC}
Let $G$ be a finitely generated nilpotent group equipped with a 
generating $k$-tuple $(s_1,\dots,s_k)$ and a $k$-tuple of positive reals
$a=(\alpha_1,\dots,\alpha_k)\in (0,\infty]^k$. 
Let $\mathfrak w=\mathfrak w(a)$ be the two-dimensional weight system which assigns weight 
$w_i=(v_{i,1},v_{i,2})$ to $s_i$ where 
$$v_{i,1}=\frac{1}{\tilde{\alpha_i}},\;\;\tilde{\alpha}_i=\min\{2,\alpha_i\}$$
and
$$v_{i,2}= 0 \mbox{ unless } \alpha_i=2 \mbox{ in which case } v_{i,2}=1/2.$$ 
Let $\mathfrak F$ be the associated weight function system 
generated by 
$$F_i(r)=r^{v_{i,1}}[\log (1+r)]^{v_{i,2}},\;1\le i\le k.$$
Then
 \[
\ \left\vert \mu _{S,a}^{(2n+m)}(xg)-\mu _{S,a}^{
(2n+m)}(x)\right\vert \leq C \left( F_S^{-1}(\|g\|_{\Sigma,\mathfrak F}) /m
\right)^{1/2}\, \mu _{S,a}^{(2n)} (e).
\]
\end{lem}

\begin{proof}
By Theorem \ref{th-coord2F}, 
there is an integer $p=p(G,S,\mathfrak{w})$ such
that any $g$ with $F^{-1}_S(\left\Vert y\right\Vert _{S,\mathfrak F})=r$ can be
expressed as  
\[
g=\prod\limits_{j=1}^{p}s_{i_{j}}^{x_{j}}\mbox{ with }\left\vert
x_{j}\right\vert \leq C F_{i_j}(r).
\]%
Write $\mu _{S,a}^{
(2n+m)}=\mu _{S,a}^{(n+m)}\ast \mu _{S,a}^{(n)}$ and, for
each step $s_{i_j}^{x_j}$, apply
Lemma \ref{lem-1d}(ii) to obtain
\begin{eqnarray*}
\lefteqn{ 
\left\vert \mu _{S,a}^{(2n+m)}(zs_{i_{j}}^{x_{j}})-\mu _{S,a}^{
(2n+m)}(z)\right\vert  } \hspace{.5in}&&\\
&\le &C \mathcal G_{i_j}(|x_j|)^{-1/2}\left\vert x_{j}\right\vert 
\mathcal{E}_{\mu _{S,a}}(\mu _{S,a}^{(n+m)},\mu _{S,a}^{
(n+m)})^{1/2}\left\Vert \mu _{S,a}^{(n)}\right\Vert _{2} \\
&\leq &Cr^{1/2} \mathcal{E}_{\mu _{S,a}}(\mu _{S,a}^{
(n+m)},\mu _{S,a}^{ (n+m)})^{1/2}\left\Vert \mu _{S,a}^{
(n)}\right\Vert _{2}.
\end{eqnarray*}%
Here, according to Lemma \ref{lem-1d}, 
$\mathcal G_i(r)= r^{2-\widetilde{\alpha}_i}$ if $v_i,2=0$ and 
$\mathcal G_i(r)=\log (1+r)$ if 
$v_{i,2}=1/2$ (i.e., if $\alpha_i=2$). Hence, 
$s^2 /\mathcal G_i(s) \simeq  F_i^{-1}(s)$, which gives the last inequality.

By \cite[Lemma 3.2]{HSC}, we also have
\[
\mathcal{E}_{\mu _{S,a}}(\mu _{S,a}^{(n+m)},\mu _{S,a}^{
(n+m)})^{1/2}\leq C m^{-1/2}\left\Vert \mu _{S,a}^{(n)}\right\Vert _{2}
=C m^{-1/2} \mu _{S,a}^{(2n)}(e)^{1/2}.
\]
This gives the desired inequality.
\end{proof}

\begin{theo}\label{th-low2near} 
Let $G$ be a finitely generated nilpotent group equipped with a 
generating $k$-tuple $(s_1,\dots,s_k)$ and a $k$-tuple of positive reals
$a=(\alpha_1,\dots,\alpha_k) \in (0,\infty]^k$. 
Set $\tilde{\alpha}_i=\min\{\alpha_i,2\}$. 
Let $\mathfrak w$ be the weight system which assigns weight 
$1/\tilde{\alpha}_i$ to $s_i\in S$. Let $\Sigma$ be a sequence 
of formal commutators as in {\em Theorem \ref{th-coord2F}}. Assume that
$w(s)>1/2$ for all $s\in \mbox{\em core}(\mathfrak w,S,\Sigma)$. 
Then, there exists $\epsilon>0$ such that,
uniformly over the region 
$\{x\in G: \|x\|_{S,\mathfrak w}\le F_{S}(\epsilon n)\}$, we have
$$\mu_{S,a}^{(n)}(x)\simeq  n^{-D(S,\mathfrak w)}.$$
\end{theo}
\begin{proof} Theorem \ref{th-low2} gives 
$\mu_{S,a}^{(n)}(e)\simeq  n^{-D(S,\mathfrak w)}$. This, together with
Lemma \ref{lem-HSC}, yields the desired lower bound.
\end{proof}

\begin{theo}\label{th-low32near} 
Let $G$ be a finitely generated nilpotent group equipped with a 
generating $k$-tuple $(s_1,\dots,s_k)$ and a $k$-tuple of positive reals
$a=(\alpha_1,\dots,\alpha_k) \in (0,\infty]^k$.  Set 
$\tilde{\alpha}_i=\min\{\alpha_i,2\}$.
Let $\widetilde{\mathfrak w}$ be the weight system which 
assigns weight $\widetilde{w}_i=1/\tilde{\alpha}_i$ to $s_i$.
Let $\Sigma$ be as in {\em Theorem \ref{th-coord2F}} applied
to $(S,\widetilde{\mathfrak w})$ and assume that 
$\alpha_i=2$ for all $i\in \{1,\dots,k\}$
such that $s_i\in \mbox{\em core}(S,\widetilde{\mathfrak w},\Sigma)$. 
Then there exists $\epsilon>0$ such that, uniformly over the region
$$\{x\in G: |x|^2_S[\log |x|_S]^{-1}\le \epsilon n\},$$ we have
$$\mu_{S,a}^{(n)}(x)\simeq  [n\log n]^{-D(G)/2}.$$
\end{theo} 
\begin{proof} By Theorem \ref{th-low32}, we have
$\mu_{S,a}^{(n)}(e)\simeq  [n\log n]^{-D(G)/2}$. 
Let $\mathfrak w,\mathfrak F$ be the two dimensional weight system
and weight function system introduced above in Lemma \ref{lem-HSC}.
It follows from Theorems \ref{th-coord2F}-\ref{th-coord2Fimp} and Corollary
\ref{cor-normeq} that $F^{-1}_S(\|\cdot\|_{S,\mathfrak F})\simeq 
|\cdot|_S^2/\log |\cdot |_S$. The result follows. 

\end{proof}

\section{Proofs regarding approximate coordinate systems} \setcounter{equation}{0}
This section contains the proofs of the key 
results stated in 
Sections \ref{sec-wF}-\ref{sec-vol}, namely, 
Theorems \ref{th-coord2F}-\ref{th-coord3}.
Throughout this section, $G$ is a finitely generated nilpotent group equipped 
with a  generating $k$-tuple $(s_1,\dots,s_k)$. Formal commutators refer
to commutators on the alphabet $\{s_i^{\pm 1}: 1\le i\le k\}$.

\subsection{Proof of Theorem \ref{th-coord3} and assorted results}
Theorem \ref{th-coord3} is one of the keys to the random walk upper 
bounds of Section \ref{sec-rwub}.  It can be understood 
as providing a volume lower bound for the volume of certain balls together with 
some additional ``structural information'' on the balls in question. 

Fix a weight system $\mathfrak w$ and weight functions $F_c$
as in Theorem \ref{th-coord3}.
Let $G^\mathfrak w_h$ be the associated descending normal series in $G$.
By construction, $G^\mathfrak w_h$ is normal in $G$ and, for all $p,q,j$ 
such that $\bar{w}_p+\bar{w}_q\ge \bar{w}_j$, we have (see Section \ref{sec-main})
$$[G^\mathfrak w_p,G^\mathfrak w_q]\subset G^\mathfrak w_j.$$
It follows that the commutators map
$$ G^\mathfrak w_p\times G^\mathfrak w_q: (u,v)\mapsto [u,v]\in G^\mathfrak w_j$$
induces a group homomorphism 
$$ G^\mathfrak w_p/G^\mathfrak w_{p+1}\otimes G^\mathfrak w_q/G^\mathfrak w_{q+1}
\ra G^\mathfrak w_j/G^\mathfrak w_{j+1}.$$
This yields the following lemma.
\begin{lem}[Similar to {\cite[Lemma 3]{Bass}}] \label{lem-B}
Let $c$ be a formal commutator
of weight $\bar{w}_j$ and let $g_c$ be its image in $G$. 
There is an integer 
$\ell=\ell(c)\le 8^j$ and a sequence $(i_1,\dots,i_\ell)\in \{1,\dots,k\}^\ell$
such that, for any $r\ge 1$ and $n\in \mathbb Z$ satisfying $|n|\le F_c(r)$, 
we have
$$g_c^n= s_{i_1}^{n_{1}}s_{i_2}^{n_{2}}\cdots s_{i_\ell}^{n_{\ell}} \;\;
\mbox{ mod } G^\mathfrak w_{j+1}$$
for some $n_{i_j}\in \mathbb Z$ with $|n_{j}|\le F_{s_{i_j}}(r)$.
\end{lem}
\begin{proof} The proof is by induction on $j$. For $j=1$, $c$ must have 
length $1$ and $g_c^n=s_i^n$ for some $i\in \{1,\dots,k\}$. Assume the result
holds true for all $h<j$ and let $c$ be a commutator of weight $\bar{w}_j$.
Either $c$ has length $1$ and the result is trivial or $c=[u,v]$ where 
$u,v$ are commutators of weights $\bar{w}_p, \bar{w}_p$, $\bar{w}_p+\bar{w}_q=
\bar{w}_j$. Since $F_c=F_uF_v$, for all $|n|\le F_c(r)$ we can  write
$n= ab +d$ with $|a|,|d|\le F_u(r)$, $0\le d \le F_v(r)$.
Then
$$g_c^n=[u,v]^{ab}[u,v]^d =[ u^a,v^b][u^d,v] \mbox{ mod } G^\mathfrak w_{j+1}.$$
The desired result follows from the induction hypothesis.
\end{proof}
\begin{defin}Given $c$, $\ell=\ell(c)$ and $(i_1,\dots,i_\ell)$ as 
in Lemma \ref{lem-B}, for any $\mathbf x=(x_1,\dots,x_\ell)\in \mathbb Z^\ell$,
set
$$\mathbf g _c(\mathbf x)=  \mathbf g _c(x_1,\dots,x_\ell)=  
s_{i_1}^{x_{1}}s_{i_2}^{x_{2}}\cdots s_{i_\ell}^{x_{\ell}}\in G .$$
Set 
$$F^c_j=F_{s_{i_j}}=F_{i_j}, \;\;1\le j\le \ell.$$
\end{defin}

By Lemma \ref{lem-B}, if $w(c)=\bar{w}_j$ and $|n|\le F_c(r)$ then
$$ g^n_c=\mathbf g_c(\mathbf n(c)) \mbox{ mod } G^\mathfrak w_{j+1}$$ 
for some $\mathbf n(c)=(n_1(c),\dots, n_\ell(c))$ with 
$|n_{j}(c)|\le F_{s_{i_j}}(r)=F^c_j(r)$. 
 
\begin{theo} \label{th-coord31}
Let $c_1,\dots c_t$ be a sequence of formal 
commutators with non-decreasing $\mathfrak w$-weights and such that, 
for each $h$, the image in $G^\mathfrak w_h/G^\mathfrak w_{h+1}$
of the family $\{c_i: w(c_i)=\bar{w}_h\}$ is a linearly independent family.
Set
$$K(r)=\{g\in G: g=\prod_{i=1}^t\mathbf g_{c_i}(\mathbf x_i),\;\;
\mathbf x_i=(x^i_1,\dots,x^i_{\ell(c_i)})\in \mathbb Z^{\ell(c_i)},
|x^i_j|\le F^{c_i}_j(r)\}.$$ 
Then
$$\#K(r)\ge \prod_1^t (2F_{c_i}(r)+1)\ge \prod_{i=1}^t\prod_{j=1}^{\ell(c_i)} 
F^c_{j}(r).$$
\end{theo}
\begin{proof} For each $(y_i)_1^t\in \mathbb Z^t$ with $|y_i|\le F_{c_i}(r)$, 
let $\mathbf y_i=(y^i_j)_1^{\ell(c_i)}$, $1\le i\le t$,  be such that
$$g_{c_i}^{y_i}=\mathbf g_{c_i}(\mathbf y_i)\;\mbox{ mod } G^\mathfrak w_{j+1},\;\;
w(c_i)=\bar{w}_j,
\; 1\le i\le t.$$
Such a $(\mathbf y^i)_1^t$ is given by Lemma \ref{lem-B}. Assume that 
two sequences 
$(y_i)_1^t$ and $(\tilde{y}_i)_1^t$ are such that
$\prod_{i=1}^t\mathbf g_{c_i}(\mathbf y_i)=\prod_{i=1}^t\mathbf g_{c_i}
(\widetilde{\mathbf {y}}_i)$. 
Then by projecting on 
$G^\mathfrak w_1/G^\mathfrak w_2$ and using the assumed linear independence
of the collection of the $c_i$'s with  $w(c_i)=\bar{w}_1$ in
 $G^\mathfrak w_1/G^\mathfrak w_2$ and the fact that 
$g_{c_i}^{y_i}=\mathbf g_{c_i}(\mathbf y^i)$ in $G^\mathfrak w_1/G^\mathfrak w_2$,we find that $y_i= \tilde{y}_i$ for 
those $i$ with $w(c_i)=\bar{w}_1$. This implies that
$\mathbf y_1=\widetilde{\mathbf y}_1$.
Proceeding further up in the weight 
filtration shows that we must have $ \mathbf y_i=\tilde{\mathbf y}_i$ for all 
$1\le i\le t$. This shows that there are at least
$\prod_1^t(2F_{c_i}(r)+1)$ distinct elements in $K(r)$ which is the 
desired result.
\end{proof}

\begin{theo}\label{th-coord32} 
Fix a weight system $\mathfrak w$ and weight functions $F_c$
as in {\em Theorem \ref{th-coord3}}.
Let $b_1,\dots b_t$ be a sequence of elements in $G$. Assume that : 
\begin{enumerate}
\item For each $i=1,\dots,t$, there exists an integer $h(i)$ such that 
$b_i\in G^\mathfrak w_{h(i)}$ and $b_i$ 
is torsion free in $G^\mathfrak w_{h(i)}/G^\mathfrak w_{h(i)+1}$. Further,
for each $h$, the system $\{b_i: h(i)=h\}$ is free in 
$G^\mathfrak w_{h(i)}/G^\mathfrak w_{h(i)+1}$.

\item For each $i=1,\dots,t$, there exists and increasing function
$\widetilde{F}^i$, a positive integer $\ell(i)$ and a sequence
${j^i_1},\dots, j^i_{\ell(i)}$  such that, 
for any $r>0$ and any integer $n$ with $|n|\le \widetilde{F}^i(r)$,  
there exists $\mathbf n^i= (n^i_1,\dots, n^i_{\ell(i)})$
with $ |n^i_q|\le F_{j^i_q}(r)$ satisfying
$$b_i^{n} = \prod_{q=1}^{\ell(i)} s_{j^i_q}^{n^i_{q}} 
\;\;\;\mbox{\em mod} \;\; 
G^\mathfrak w_{h(i)+1}.$$
\end{enumerate}
For $\mathbf x=(x_1,\dots, x_{\ell(i)})\in \mathbb Z^{\ell(i)}$,
set
$\mathbf b _i(\mathbf x)= \prod_{q=1}^{\ell(i)} s_{j^i_q}^{x_q}\in G$
and 
$$K(r)=\{g\in G: g=\prod_{i=1}^t\mathbf b_i(\mathbf x_i),\;\;
\mathbf x_i=(x^i_1,\dots,x^i_{\ell(i)})\in \mathbb Z^{\ell(i)},
|x^i_q|\le F_{j^i_q}(r)\}.$$ 
Then
$$\#K(r)\ge \prod_1^t (2\widetilde{F}_{i}(r)+1).$$
\end{theo}
\begin{proof} This a straightforward generalization of Theorem \ref{th-coord31}.
Instead of considering commutators and their natural weight function $F_c$, we
consider arbitrary group elements $b$ with associated weight function 
$\widetilde{F}$ with the property that $b$ is free in 
$G^\mathfrak w_h/G^\mathfrak w_{h+1}$, for some $u,h$, and   
$b^{n}$, $|n|\le \widetilde{F}(r)$, can be express modulo $G^\mathfrak w_{h+1}$   as a 
fixed product of powers of generators 
with properly controlled exponents.  The proof is essentially the 
same as that of Theorem \ref{th-coord31}. Namely, for each 
$(y_i)_1^t\in \mathbb Z^t$ with $|y_i|\le \widetilde{F}^i(r)$, 
let $\mathbf y_i=(y^i_j)_1^{\ell(i)}$, $1\le i\le t$,  be such that
$$b_i^{u_i y_i}=\mathbf b_i(\mathbf y_i)\;\mbox{ mod } G^\mathfrak w_{h(i)+1},\;\;
\; 1\le i\le t.$$
Such a $(\mathbf y^i)_1^t$ exists by hypothesis. Assume that 
two sequences 
$(y_i)_1^t$ and $(\tilde{y}_i)_1^t$ are such that
$\prod_{i=1}^t\mathbf b^{i}(\mathbf y_i)=\prod_{i=1}^t\mathbf b^i(\widetilde{\mathbf y}_i)$. 
Then by projecting on 
$G^\mathfrak w_1/G^\mathfrak w_2$ and using the assumed freeness
of the collection of the $b_i$'s with  $h(i)=1$ in
 $G^\mathfrak w_1/G^\mathfrak w_2$ and the fact that 
$b_i^{u_iy_i}=\mathbf b^i(\mathbf y^i)$ in $G^\mathfrak w_1/G^\mathfrak w_2$,we find that $y_i= \tilde{y}_i$ for 
those $i$ with $h(i)=1$. This implies $\mathbf y_1=\mathbf{\widetilde{y}}_1.$
Proceeding further up in the weight 
filtration shows that we must have $ y_i=\tilde{y}_i$ for all 
$1\le i\le t$. This shows that there are at least
$\prod_1^t(2\widetilde{F}^i(r)+1)$ distinct elements in $K(r)$, as  
desired.
\end{proof} 

\begin{rem} Theorem \ref{th-coord32} allows for much more freedom than 
Theorem \ref{th-coord31}. This freedom is used in the proof of Theorem
\ref{th-Vcomp}.
\end{rem}

\subsection{Commutator collection on free nilpotent groups}

In this section, we prove the following weak version of Theorem \ref{th-coord2F}.

\begin{theo}\label{th-coord2FW}
Referring to the setting and notation of {\em Theorem \ref{th-coord2F}}, 
assume 
that {\em (\ref{F1})-(\ref{F2})} hold true.
Then there  exist an integer 
$t=t(G,S,\mathfrak w) ,$ 
a constant $C=C(G,S,\mathfrak w)\ge 1$, and 
a sequence $\Sigma$ of commutators (depending on $G,S, \mathfrak w$) 
$$c_1,\dots,c_t \mbox{ with  non-decreasing weights } w(c_1)\preceq  
\dots\preceq  w(c_t)$$ such that  
\begin{itemize}  
\item[(i)] For any $r>0$, if $g\in G$ can be expressed as a word $\omega$
over $\mathfrak C(S)^{\pm 1}$
with $\mbox{\em deg}_{c}(\omega)\le F_{c}(r)$ for all 
$c\in \mathfrak C(S)$ then
$g$ can be expressed in the form
$$g=\prod_{i=1}^t c_i^{x_i}
\mbox{ with }|x_i|\le  F_{c_i}(Cr) \mbox{ for all } i\in \{1,\dots, t\}.$$
\item[(ii)] There exist an integer $p=p(G,S,\mathfrak w)$ and 
$(i_j)_1^p\in\{1,\dots, k\}^p$  (also depending on $(G,S,\mathfrak w)$
such that, if $g$ can be expressed as a word $\omega$ 
over $\{c_i^{\pm 1}: 1\le i\le t\}$
with $\mbox{\em deg}_{c_i}(\omega)\le F_{c_i}(r)$ for some $r>0$ then
$g$ can be expressed in the form
$$g=\prod_{j=1}^p s_{i_j}^{x_j}
\mbox{ with } |x_j|\le F_{{i_j}}(Cr).$$
\end{itemize}
\end{theo}

\begin{rem} Note that it must be the case that, for any $j$, the image of 
$\{c_i: w(c_i)=\bar{w}_j\}$ 
in $G^\mathfrak w_j/G^\mathfrak w_{j+1}$
generates $G^\mathfrak w_j/G^\mathfrak w_{j+1}$. The key difference with 
Theorem \ref{th-coord2F} is that  Theorem \ref{th-coord2FW} does not identify
a maximal subset of $\{c_i: w(c_i)=\bar{w}_j\}$ that is free in
$G^\mathfrak w_j/G^\mathfrak w_{j+1}$.  
\end{rem}

The proof of Theorem \ref{th-coord2FW} requires a number of steps.
The first observation is that it is enough to prove 
Theorem \ref{th-coord2FW} in the case of the free nilpotent group 
$N(k,\ell)$ on $k$ generators $s_1,\dots, s_k$  and of nilpotency class
$\ell$. Indeed,
once Theorem \ref{th-coord2FW} is proved on $N(k,\ell)$, 
the same statement holds on any nilpotent $G$ of nilpotency class $\ell$
equipped with a generating k-tuple $S$  
via the canonical projection from $N(k,\ell)$ 
to $G$ (by definition, the canonical projection is the group homeomorphism
from $N(k,\ell)$ onto $G$ which sends the canonical 
$k$ generators of $N(k,\ell)$ to the given $k$ generators of $G$). 

\begin{nota}
For the rest of this section, we assume that $G=N(k,\ell)$ is the  
free nilpotent group 
$N(k,\ell)$ equipped with its canonical generating set $S=(s_1,\dots,s_k)$
and the multidimensional weight-system $\mathfrak w$ generated by the
$(w_1,\dots,w_k)$.
Without loss of generality, we assume that the commutator set $\mathfrak C(S)$
is equipped with a total order $\prec$ such that the function
$$w: \mathfrak C(S)\ni c\mapsto w(c)\in (0,\infty)\times \mathbb R^{d-1}$$
associated with the given weight system $\mathfrak w$ 
is non-decreasing. Hence, $c\prec c'$ implies  $w(c)\preceq w(c')$. 
In addition, we let $\mathfrak F$ be a weight function system that is compatible with $\mathfrak w$ in the sense that (\ref{F1})-(\ref{F2}) hold true.
\end{nota}

\begin{nota} Recall that $\mbox{deg}_c(\omega)$ denotes the number of occurrences 
of $c^{\pm 1}$ in the word $\omega$ over $\mathfrak C(S)$. Similarly, we define 
$\mbox{deg}^*_c(\omega)$ to be the number of occurrences of $c$ 
minus the number of occurrences of $c^{-1}$ in a word over $\mathfrak C(S)$. 
\end{nota}

On $\mathfrak C(S)$, consider  the map $J$ such that $J(s_i^{\pm 1})=s_i^{\mp 1}$
and $J([a,b])=[b,a]$. Abusing notation, we also write $J(c)=c^{-1}$. Note that 
$J^2$ is the identity.
Restrict $J$ to $\mathfrak C^*(S)=\{c: J(c)\neq c\}$ (where $J(c)=c$ is 
understood as equality as formal commutator so that $J(s_i)\neq s_i$
and $J([a,b])=[a,b]$ if and only if $a=b$).  
Let $\mathfrak C_+^*$ be the set of representative of $\mathfrak C^*(S)/J$
given by $ c\in \mathfrak C_+^*(S)$ if and only if $c=s_i$ or $c=[a,b]$ 
with $a\succ b$. 
  
It is convenient to enumerate all formal commutators in  
$\mathfrak C^*_+(S,\ell)$
and write
$$\mathfrak C^*_+(S,\ell)=\{c_1,\dots,c_t\},\;\;t=\#\mathfrak C^*_+(S,\ell).$$ 
Since $\ell$ is fixed throughout, we write 
$$\mathfrak C^*_+(S)=\mathfrak C_+^*(S,\ell).$$ 
Note that, a priori, 
this list contains commutators that are trivial in $N(k,\ell)$. This does 
not matter although these formal commutators can be omitted if desired. 
Let us describe the basic collecting process on $N_{k,\ell}$.\vspace{.1in}

\noindent{\bf Commutator collecting algorithm}\vspace{.05in}

\begin{itemize}
\item Given a word $\omega =c^{\epsilon_{i_1}}_{i_{1}}c^{\epsilon_{i_2}}_{i_{2}}...c^{\epsilon_{i_m}}_{i_{m}}$ in $\mathfrak{C}^*_+%
(S)\cup \mathfrak{C}^*_+(S)^{-1},$ first identify the commutator of lowest
order with respect to $\prec $, say it is commutator $%
c_{i_{j}}$, mark all the contributions of $c_{i_{j}}$ to $\omega $ from left
to right in order$:$ $\{y_{1},...,y_{q}\},$ $y_{j}\in \{
c_{i_{j}}^{\pm 1}\}.$

\item Starting with $y_{1}$, move $y_{1},...,y_{q}$ to the left one by
one by successive commutation. Note that every time $c_{i_{j}}$ jumps backward
over a
commutator $c$, the jump produces the sequence  
$...c_{i_{j}}c[c,c_{i_{j}}]...$. It follows that all
commutators that are created in this process belong to $\mathfrak C_+^*(S)$ 
and have weight 
$\succeq 2w(c_{i_{j}})\succ w(c_{i_j}).$ 

\item After $y_{1},...,y_{q}$ have been moved to the left, we obtain a word $%
y_{1}...y_{q}\omega'$ with the same image as $\omega$, and where
$\omega'$ is a word in commutators $\succ c_{i_{j}}.$ 

\item Apply the previous steps to $\omega'$, producing $\omega''$ and continue  
until the process terminates after at most $ \# \mathfrak C_+^*(S)$ steps.
\end{itemize}

This proves the following weak version of M. Hall basis theorem 
\cite[Theorem 11.2.3]{MHall} 
(in Hall's more sophisticated version, only the so called ``basic'' commutators are used 
and this results in a unique representation of any element of $N(k,\ell)$). 
\begin{pro}
Any element $g\in N(k,\ell)$ has a representation $%
$g=c_{1}^{x_{1}}c_{2}^{x_{2}}...c_{t}^{x_{t}},\;\;x_i\in \mathbb Z.$$
\end{pro}

Next we want to have some control over $\{x_{i},$ $1\leq i\leq t\}.$ Let's
start with a simple binomial counting lemma adapted from \cite[page 173]{MHall} 
and  \cite{Tits}. We will use the following notation. For any two commutators 
$c_j\succ c_i$, let $C_{n-1}(i,j)$ be the sets of all commutators 
$c\in \mathfrak C^*_+(S)$
such that there exist  $\epsilon_0,\dots,\epsilon_ n
\in \{-1,1\}$ such that $c_j^{\epsilon_n}=
[\cdots [c^{\epsilon_0}, c_i^{\epsilon_1}],\dots,c_i^{\epsilon_{n-1}}]$ 
(as formal commutators in $\mathfrak C(S)$).

\begin{lem} Consider a word $%
\omega $ in $\{c_j:c_j\succeq c_{i}\}^{\pm 1}$. 
Let $m=\deg _{c_{i}}\omega $,
and let $\{y_{1},...,y_{m}\}$, $y_{j}\in \{$ $c_{i}^{\pm 1}\}$,  
be the left to right contribution of $c_{i}$ to $\omega$. 
For $0\leq q\leq m,$ there is a
word $\omega _{q}$ in $\{c_j:c_j\succeq c_{i}\}^{\pm 1}$ which starts with
$y_{1}...y_{q}$, whose left to right contribution of $c_i^{\pm 1}$
is $y_1,\dots, y_m$, and in which, for all $c_{j}\succ c_{i}$, 
\begin{eqnarray*}
\deg _{c_{j}}(\omega _{q}) &\leq &\deg _{c_{j}}(\omega
)+q\sum_{c\in C_1(i,j)}\deg _{c}(\omega )
+\binom{q}{2}%
\sum_{c\in C_2(i,j)}\deg _{c}(\omega ) \\
&&+...+\binom{q}{\ell}
\sum_{c\in C_\ell(i,j)}\deg _{c}(\omega )
\end{eqnarray*}%
Further, if $c^{\prime }$ denotes the lowest
commutator in $\omega $ with $c^{\prime }\succ c_{i}$ then contributions of
commutators $c$ with $w(c)\prec w(c^{\prime })+w(c_{i})$ remain unchanged in $\omega_q$.
\end{lem}
\begin{rem}
Note that, after we move all contributions of 
$c_{i}$ to $\omega$ to the left, we obtain a word $\omega_m$ with same image 
as $\omega$ of the form
\[
\omega _{m}=c_{i}^{x}\omega _{m}^{\prime } 
\]
where $x=\deg _{c_{i}}^{\ast }(\omega )$, $\omega _{m}^{\prime }$ is a
word in  $[\mathfrak C_+^*(S)\cap\{c\succ c_{i}\}]^{\pm 1}$, and in which the
contributions of
commutators $c$ with $w(c)\prec w(c^{\prime })+w(c_{i})$ remain the same than in $%
\omega .$
\end{rem}

\begin{proof} The proof is by  induction on $q$. 
It holds trivially for $q=0.$ The induction hypothesis
gives us a word $\omega _{q-1}$ with%
\begin{eqnarray*}
\deg _{c_{j}}(\omega _{q-1}) &\leq &\deg _{c_{j}}(\omega
)
+(q-1)\sum_{c\in C_1(i.j)}\deg _{c}(\omega )
+\binom{q-1}{2%
}\sum_{c\in C_2(i,j)}\deg _{c}(\omega ) \\
&&+...+\binom{q-1}{\ell}\sum_{c\in C_\ell(i,j)}\deg _{c}(\omega ).
\end{eqnarray*}%
Now, we move $y_{q}$ to the left as in the collecting process by successive
commutations. To keep track of contribution of $c_{j},$ notice 
that a new contribution of $c_j$ is produced  only if $%
y_{q}$ jumps over a commutator $c^{\pm 1}$ such that $[c^{\pm 1},y_{q}]=
c^{\pm 1}_{j}$. 
Further, $w([c^{\pm 1},y_{q}])=w(c)+w(c_{i})\succeq  w(c')+w(c_{i})$. 
Hence, $c_j$ must satisfies $w(c_j)\succeq w(c')+w(c_i)$.
Therefore we eventually get a word $%
\omega _{q}$ in $[\mathfrak C^*_+(S)\cap \{c\succeq c_{i}\}]^{\pm 1}$ with $\pi
(\omega _{q})=\pi(\omega )$, in which the left to right contribution 
of $c_{i}$ is the same as in $\omega$, which 
starts with $y_{1}...y_{q}$, and such that%
\[
\deg _{c_{j}}(\omega _{q})\leq \deg _{c_{j}}(\omega
_{q-1})+\sum_{c\in C_1(i,j)}\deg _{c}(\omega _{q-1}). 
\]%
Using the induction hypothesis on $\omega _{q-1}$ and the fact that
all brackets of length at least $\ell+1$ drop out,
\begin{eqnarray*}
\sum_{c\in C_1(i,j)}\deg _{c}(\omega _{q-1})&=&
\sum_{c=c_\alpha\in C_2(i,j)} 
\sum_{p=0}^{\ell} \binom{q-1}{p}\sum_{\tilde{c}\in C_p(i,\alpha)}
\deg _{\widetilde{c}}(\omega)\\
&\le & 
\sum_{p=1}^{\ell} \binom{q-1}{p-1}\sum_{\widetilde{c}\in C_p(i,j)}
\deg _{\widetilde{c}}(\omega).
\end{eqnarray*}
Hence, we have
\begin{eqnarray*}
\deg _{c_{j}}(\omega _{q})&\leq &\deg _{c_{j}}(\omega
_{q-1})+\sum_{c\in C_2(i,j)}\deg _{c}(\omega _{q-1})\\
&\le & \sum_{p=0}^{\ell}\left( \binom{q-1}{p}+ \binom{q-1}{p-1}\right)
\sum_{\widetilde{c}\in C_p(i,j)}
\deg _{\widetilde{c}}(\omega)\\
&= & \sum_{p=0}^{\ell}\binom{q}{p}\sum_{\widetilde{c}\in C_p(i,j)}
\deg _{\widetilde{c}}(\omega).
\end{eqnarray*}
\end{proof}

\begin{lem} \label{lem-bullet2}
\bigskip\ There exists a constant $C>0$ such that for any word $\omega $ in $%
[\mathfrak{C}^*_+(S)\cap \{c\succeq c_{i}\}]^{\pm 1}$ 
with $\deg _{c}\omega \leq
F_{c}(d)$ for all $c\succeq c_{i}$, there exists a word $\omega'$ in 
$[\mathfrak{C}_+^*(S)\cap \{c\succeq c_{i}\}]^{\pm 1}$ in collected form:%
\[
\omega'=\prod\limits_{j=i}^{t}c_{j}^{x_{j}} 
\]%
such that $\pi(\omega')=\pi(\omega )$, $x_{j}=\deg
_{c_{j}}^{\ast }\omega $ for those $j$ such that $w(c_{j})\prec 2w(c_{i})$
and $\left\vert x_{j}\right\vert \leq F_{c_{j}}(Cd)$ for all $i\leq j\leq t.$
\end{lem}

\begin{proof} The proof is by backward induction on $i$.
For $i=t$, the statement holds trivially since commutators with $c\succeq c_t$
commute.

Suppose the assertion holds for $i+1.$  Consider a word $\omega $ on
$[\mathfrak C_+^*(S)\cap \{c\succeq c_{i}\}]^{\pm 1}$ as in the lemma. 
Let  $\{y_{1},...,y_{q}\}$ be the
contribution of $c_{i}$ to $\omega$, $q=\deg _{c_{i}}\omega .$
The previous lemma yields  $\omega _{q}=y_{1}...y_{q}\omega'_q,$ 
where  $\omega'_q$ is
a word in $[\mathfrak C^*_+(S)\cap \{c\succeq c_{i+1}\}]^{\pm 1}$. 
From the hypothesis on the degrees of $\omega ,$\ 
$$
\deg _{c_{j}}(\omega _{k}) \leq  \sum_{p=0}^{\ell}
\binom{k}{p} \sum_{c\in C_p(i,j)}F_{c}(d)
$$
From definition of weight functions, if $c\in C_p(i,j)$ then $%
F_{c}F_{c_{i}}^{p}=F_{c_{j}}.$ Further,$ \#C_p(i,j)\le t=\#\mathfrak C^*_+(S)$
and $q=\deg _{c_{i}}\omega
\leq F_{c_{i}}(d)$.
Therefore, we obtain 
\begin{eqnarray*}
\deg _{c_{j}}(\omega _{q}) &\leq & t F_{c_j}(d)\left( \sum_{p=0}^{\ell}
\binom{q}{p} F_{c_i}(d)^{-p}\right)\\
&\leq & t F_{c_j}(d)\left( \sum_{p=0}^{\ell}
q^p F_{c_i}(d)^{-p}\right)\le 
t(1+\ell) F_{c_{j}}(d).
\end{eqnarray*}
By assumption (\ref{F1}), there exists a constant $C_{1}$ such
that $$t(1+\ell)F_{c}(d)\leq F_{c}(C_{1}d)$$ for all 
$c$ and $d\geq 1.$
\end{proof}

Lemma \ref{lem-bullet2} with $i=1$ proves Theorem \ref{th-coord2FW}(i).
Next we work on improving  Theorem \ref{th-coord2FW}(i) in the special 
case of the free nilpotent group $N(k,\ell)$. This 
improvement will be instrumental in proving Theorem \ref{th-coord2FW}(ii).
It is based on the following important Lemma.

\begin{lem}\label{lem-free}
For each $j$, 
$N(k,\ell)^\mathfrak w_j/N(k,\ell)^\mathfrak w_{j+1}$ is a finitely generated
free abelian group.  
\end{lem}
\begin{proof} The proof is by a backward induction on $\ell$. 
If $\ell=1$, $N(k,1)$ is the free abelian 
group on $k$ generators and the desired result holds by inspection.
Let $g\in N(k,\ell)^\mathfrak w_j$ 
such that $g\notin N(k,\ell)^\mathfrak w_{j+1}$. Let $N_\ell=N(k,\ell)_\ell$
be the center of $N(k,\ell)$ 
(i.e., the subgroup generated by commutators of length $\ell$).   
Assume first that $g\in N(k,\ell)^\mathfrak w_{j+1} N_\ell$. 
Since  
$$N(k,\ell)^\mathfrak w_{j+1}N_\ell /N(k,\ell)^\mathfrak w_{j+1}\simeq
 N_\ell/[N(k,\ell)^\mathfrak w_{j+1}\cap N_\ell],$$
and $N(k,\ell)^\mathfrak w_{j+1}\cap N_\ell$ 
is generated by the basic commutators of weight $\bar{w}_j$ and length $\ell$,
$N_\ell/[N(k,\ell)^\mathfrak w_{j+1}\cap N_\ell]$ is torsion free.
It thus follows that $g$ is not torsion in $N(k,\ell)^\mathfrak w_j/
N(k,\ell)^\mathfrak w_{j+1}$.   

Now, consider the case when $g\not\in N(k,\ell)^\mathfrak w_jN_\ell$.
Let $g'$ be the projection of $g$  in $N(K,\ell)/N_\ell=N(k,\ell-1)$.
Clearly $g'\in N(k,\ell-1)^\mathfrak w_j$ and 
$g'\not\in N(k,\ell-1)^\mathfrak w_{j+1}$ because 
the inverse image of $N(k,\ell-1)^\mathfrak w_{j+1}$ under this projection
is $N(k,\ell)^\mathfrak w_{j+1}N_\ell$.  Further,
$$N(k,\ell)^\mathfrak w_j N_\ell/N(k,\ell)^\mathfrak w_{j+1}N_\ell  
\simeq N(k,\ell-1)^\mathfrak w_j/N(k,\ell-1)^\mathfrak w_{j+1}.$$
By the induction hypothesis, $g'$ is not torsion in 
$N(k,\ell-1)^\mathfrak w_j/N(k,\ell-1)^\mathfrak w_{j+1}.$
It follows that $g$ is not torsion in $N(k,\ell)^\mathfrak w_j/
N(k,\ell)^\mathfrak w_{j+1}.$\end{proof}

Next, let $(b_i)_1^\tau$ be a sequence of elements 
of $\mathfrak C_+^*(S)$ such that 
$\{b_i: w(b_i)=\bar{w}_j\}$ projects to a basis of 
$N(k,\ell)^\mathfrak w_j/N(k,\ell)^\mathfrak w_{j+1}$.
Let $R^\mathfrak w_j$ be the rank of this torsion free abelian group and set
$m'_j= \sum_1^j R_i^\mathfrak w$ so that $\tau=m'_{j_*}$. 
Set also $m_j=\max\{i: w(c_i)=\bar{w}_j\}$. Without loss of generality, 
we can assume that our ordering  on $\mathfrak C^*_+(S)$ is such that
$$(b_i)_{m'_{j-1}+1}^{m'_j}=(c_j)_{m_{j-1}+1}^{m_{j-1}+R^\mathfrak w_j}.$$

\begin{lem}\label{lem-Free} Referring o the above setup and notation, 
there exists a constant $C>0$ such that for any word $\omega $ in 
$\{c_i:w(c_i)\succeq \overline{w}_{h}\}^{\pm 1}$ with $\deg _{c_{j}}\omega
\leq F_{c_{j}}(d)$ for all $j,$ there is a word $\omega _h$ 
\[
\omega_h=\prod\limits_{j=m'_{h-1}+1}^\tau  b_{j}^{x_{j}} 
\]%
such that $\pi (\omega _{h})=\pi (\omega )$ and $\left\vert x_{j}\right\vert
\leq CF_{c_{j}}(Cd)$, $m'_{h-1}+1\le j\le m'_h$. 
\end{lem}

\begin{proof} The proof is by
backward induction on $h.$
When $h=j_*,$ $N(k,\ell)_{j_*}^{\mathfrak{w}}$ is abelian and 
this is just linear algebra.

For a word $\omega$ as in the lemma, Lemma \ref{lem-bullet2} gives
a word 
$$\omega'=\prod_{i\ge m_{h-1}+1}c_i^{x_i}, \;\;|x_i|\le F_{c_i}(Cd)$$
with the same image as $\omega$. Set 
$$I_1(h)=\{m_{h-1}+1,\dots, m_{h-1}+R^\mathfrak w_h\},\;\;I_2(h)=\{m_{h-1}+R^\mathfrak w_h+1,\dots, m_h\}$$
For $i\in I_2(h)$, $c_i$ has the same image than
$$\prod_{j\in I_1(h)} c_j ^{z_{j,i}} v_i$$
with $v_i$ a word in $\{c_p: w(c_p)\succeq \bar{w}_{h+1}\}^{\pm 1}$.
Hence
$$\omega''= \prod_{j\in I_1(h)} c_j^{x_j} \prod_{i\in I_2(h)}
\left(\prod_{j\in I_1(h)} c_j^{z_{i,j}} v_i\right)^{x_i} \prod_{p> m_h} c_p^{x_p}
$$ 
has the same image than $\omega$. Applying Lemma \ref{lem-bullet2} to this word
$\omega''$ gives
$$\omega'_h=
\prod_{j\in I_1(h)} c_j^{x_j +\sum_{i\in I_2{h}}z_{i,j}x_i} 
\prod_{p> m_h} c_p^{x'_p}$$
with the same image  than $\omega''$ and  $|x'_p|\le F_{c_p}(Cd)$ for $p>m_{h}$.
Further, since $F_{c_i}\simeq F_{c_j}\simeq \mathbf F_h$, 
for $i\in I_1(h), j\in I_2(h)$, we have 
$$|x_j+\sum_{i\in I_2(h)} z_{i,j}x_i|\le F_{c_j}(Cd).$$
Applying the induction hypothesis to rewrite
$\prod_{p>m_h}c_p^{x'_p}$ finishes the proof.\end{proof}

\begin{theo}\label{th-coord2FF} Assume that the free nilpotent group $N(k,\ell)$
is equipped  with its canonical generating $k$-tuple $S=(s_1,\dots,s_k)$ and 
a weight system $\mathfrak w$ and weight-function system $\mathfrak F$  such 
that {\em (\ref{F1})-(\ref{F2})} hold true. 
Let $b_i$, $1\le i\le \tau$,  be a sequence of elements of $C^*_+(S)$ with
$w(b_i)\preceq w(b_{i+1})$, $1\le i\le \tau-1$ and 
such that, for each $j$, $\{b_i: w(b_i)=\bar{w}_j\}$ is a basis of
the free abelian group $N(k,\ell)^\mathfrak w_j/N(k,\ell)^\mathfrak w_{j+1}$. 
Then 
\begin{itemize}
\item[(i)] Any element $g\in N(k,\ell)$ can be expressed uniquely in the form
$$g=\prod_{i=1}^\tau b_i^{x_i}, \;\;x_i\in \mathbb Z, i\in \{1,\dots, \tau\}.$$
Further, 
$$F_{S}^{-1}(\|g\|_{\mathfrak C(S),\mathfrak F})\simeq 
\max_{1\le i\le \tau}\{F_{b_i}^{-1}(|x_i|)\}.$$
\item[(ii)]  There exist an integer $p$ and $(i_j)_1^p\in\{1,\dots, k\}^p$  
such that any $g\in N(k,\ell)$ with $\|g\|_{\mathfrak C(S),\mathfrak F}\le F_S(r)$, $r>0$,
can be expressed in the form
$$g=\prod_{j=1}^p s_{i_j}^{y_j}
\mbox{ with } |y_j|\le F_{{i_j}}(Cr),\;j\in \{1,\dots,p\}.$$ \end{itemize}
\end{theo}
\begin{rem} This result is a strong version of Theorem \ref{th-coord2F}
in the special case when $G=N(k,\ell)$.
\end{rem}
\begin{proof}[Proof of (i)]
The first assertion follows from Lemma \ref{lem-Free}. Uniqueness is clear if one
 considers the projections 
of $g$ onto the successive free abelian groups 
$N(k,\ell)^\mathfrak w_j/N(k,\ell)^\mathfrak w_{j+1}.$
\end{proof}

The proof of the the second assertion requires some preparation.
Given a
commutator $c$ with length $m\leq \ell,$ let $\sigma =\sigma _{1}...\sigma _{m}$
be the formal word on the alphabet $S$ obtained from $c$
by removing brackets and inverses. 
For $\overrightarrow{a}%
=(a_{1},...,a_{\ell})\in 
\mathbb{Z}
^{\ell},$ $\Theta (\overrightarrow{a},c)$ is defined as the expression we get
by substituting in $c$ each $\sigma _{i}$ by $\sigma _{i}^{a_{i}}$, while
keeping all the brackets and signs unchanged.
For example, if $c=[[s_{i_{1}},s_{i_{2}}^{-1}],s^{-1}_{i_{3}}],$ and $%
\overrightarrow{a}=(a_{1},a_2,a_{3},0,...,0),$ we have $$\Theta (%
\overrightarrow{a}%
,c)=[[s_{i_{1}}^{a_{1}},s_{i_{2}}^{-a_{2}}],s_{i_{3}}^{-a_{3}}].$$

\begin{lem}\label{lem-Theta}
For a commutator $c$ with length $m\leq \ell,$ let $\sigma =\sigma
_{1}...\sigma _{m}$ be the formal word associated with it. Suppose $%
a_{1},...,a_{m}\in 
\mathbb{Z}
$ are such that $\left\vert a_{j}\right\vert \leq F_{\sigma _{j}}(d)$ for
all $1\leq j\leq m$, $d>0$. Set $\overrightarrow{a}=(a_{1},...,a_{m},0,...,0)\in $ $%
\mathbb{Z}
^{\ell}$ and consider the element  $u\in N(k,\ell)$ such that 
$$u c^{a_{1}...a_{k}}=\Theta (%
\overrightarrow{a},c).$$ Then $u$ can be represented by a word $\omega $ 
on  $\{c_j: w(c_{j})\succ w(c)\}^{\pm 1}$ with  $\deg _{c_{j}}(\omega
)\leq F_{c_{j}}(Cd)$ for all $c_{j}$ with $w(c_j)\succ w(c)$.
\end{lem}

\begin{proof} The proof is by 
induction on the length $m$ of the commutator $c$. When $m=1$, 
the statement is trivial.

Suppose the statement is true for commutators of length $\leq m-1.$ Let $c$
be a commutator with length $m,$ say $c=[f_1,f_2],$
 where $f_{1}, f_{2}$ are commutators of length $m_{1},$ $m_{2}<m$. 
Write $\overrightarrow{a}_{1}=(a_{1},...,a_{m_{1}},0,...,0)$
and $\overrightarrow{a}_{2}=(a_{m_{1}+1},...,a_{m_{1}+m_{2}},0,...,0)$, then
by definition 
\[
\Theta (\overrightarrow{a},c)=[\Theta (\overrightarrow{a}_{1},f_{1}),
\Theta (\overrightarrow{a}_{2},f_{2})]. 
\]%
By the induction hypothesis, 
$$\Theta (\overrightarrow{a}%
_{1},f_{1})=u_{1}f_{1}^{a_{1}...a_{m_{1}}},\;\;\Theta (\overrightarrow{a}%
_{2},f_{2})=u_{2}f_{2}^{a_{m_{1}+1}...a_{m_{1}+m_{2}}}$$ 
where $u_{1}$ can
be represented by a word $\omega _{1}$ in commutators $c_{p}$ with $%
w(c_{p})\succ w(f_1)$ and $\deg _{c_{p}}(\omega )\leq F_{c_{p}}(Cd)$. Similarly, $%
u_{2}$ can be represented by a word $\omega_2 $ in commutators $c_{p}$ with $%
w(c_{p})\succ w(f_2)$ and $\deg _{c_{p}}(\omega )\leq F_{c_{p}}(Cd).$

Suppose $w(f_1)=\overline{w}_{h_1},$ $w(f_{2})=\overline{w}_{h_2},$ and $%
w([f_{1},f_{2}])=\overline{w}_{h}$. By the natural group homomorphism%
\[
N_{h_1}^{\mathfrak{w}}/N_{h_1+1}^{\mathfrak{w}}\otimes N_{h_2}^{\mathfrak{w}%
}/N_{h_2+1}^{\mathfrak{w}}\rightarrow N_{h}^{\mathfrak{w}}/N_{h+1}^{\mathfrak{w%
}}, 
\]%
we have that 
\begin{eqnarray*}
\lbrack \Theta (\overrightarrow{a}_{1},f_{1}),\Theta (%
\overrightarrow{a}_{2},f_{2})] &\equiv &[f_{1}^{a_{1}...a_{m_{1}}},
f_{2}^{a_{m_{1}+1}...a_{m_{1}+m_{2}}}]%
\mbox{ mod }N_{h+1}^{\mathfrak{w}} \\
&\equiv &[f_{1},f_{2}]^{a_{1}...a_{m_{1}+m_{2}}}\mbox{ }
\mbox{mod }N_{h+1}^{\mathfrak{w}}\\
&\equiv & c^{a_{1}...a_{m}}\mbox{ }\mbox{mod }N_{h+1}^{\mathfrak{w}}.
\end{eqnarray*}%
Therefore $u=\Theta (\overrightarrow{a},c)c^{-a_{1}...a_{m}}\in $ $N_{h+1}^{%
\mathfrak{w}},$ and since%
\[
u=[u_{1}f_{1}^{a_{1}...a_{k_{1}}}
,u_{2}f_{2}^{a_{k_{1}+1}...a_{k_{1}+k_{2}}}]c^{-a_{1}...a_{k}}, 
\]%
it can be represented by a word $\omega $ such that $\deg
_{c_i}\omega \leq 5F_{c_i}(Cd)$ for all $i.$ 
Then by Theorem \ref{th-coord2FF}(i), we have 
\[
u=\prod\limits_{j: w(b_j)\succeq \bar{w}_h} b_{j}^{x_{j}}. 
\]%
with $\left\vert x_{j}\right\vert \leq F_{b_{j}}(C^{\prime }d)$.
\end{proof}

\begin{lem} \label{lem-cprods}
For any $h$, there exist  constants $M_{h}>0$ and $C_h>0$ 
such that, for any $c\in \mathfrak C^*_+(S)$ with $w(c)\succeq \bar{w}_h$, 
there a integer  $p=p(c)$ with $0\le p\le M_h$ and 
a $p$-tuple $(i_1,\dots,i_p)\in\{1,\dots,k\}^{p}$, such that 
for any $x\in \mathbb Z$ with  $|x|\le F_c(d)$, $d>0$,
we have 
\[
c^x=s_{i_{1}}^{x_{1}}s_{i_{2}}^{x_{2}}...s_{i_{p}}^{x_{p}} \mbox{ with } x_{j}
\in  \mathbb{Z}
,\;\;| x_{j}| \leq F_{i_{j}}(Cd), j=1,\dots, p.
\]%
\end{lem}
\begin{proof} The proof is by 
backward induction on $h$. When $h=j_*$ and  $c$ is a commutator with 
$w(c)=\overline{w}_{j^*},$ let $\sigma
=\sigma _{1}...\sigma _{m}$, $\sigma_i\in \{s_1,\dots, s_k\}$ 
be the formal word associated with $c$ 
(by forgetting brackets and inverses). Write%
\[
x=a_{0}\prod\limits_{1\leq j\leq m}\left\lfloor F_{\sigma
_{j}}(d)\right\rfloor +a_{1}\prod\limits_{2\leq j\leq m}\left\lfloor
F_{\sigma _{j}}(d)\right\rfloor +...+a_{m-1}\left\lfloor
F_{\sigma _{m}}(d)\right\rfloor +a_{m} 
\]%
with $a_{j}\in 
\mathbb{Z}
$, $|a_0|\le C$  and $\left\vert a_{j}\right\vert \leq F_{\sigma_{j}}(d).$ Write 
\[
\overrightarrow{a}_{0}=(a_{0}\left\lfloor F_{\sigma _{1}}(d)\right\rfloor
,\left\lfloor F_{\sigma _{2}}(d)\right\rfloor ...,\left\lfloor F_{\sigma
_{m}}(d)\right\rfloor ), 
\]
\[
\overrightarrow{a}_{j}=(\underbrace{1,...,1}_{j-1},a_{j},\left\lfloor
F_{\sigma _{j+1}}(d)\right\rfloor ,...,\left\lfloor F_{\sigma
_{m}}(d)\right\rfloor ), 
\]%
then%
\[
c^{x}\equiv \Theta (\overrightarrow{a}_{1},c)...\Theta (\overrightarrow{a}%
_{k},c)\mbox{ }\mbox{mod }N(k,\ell)_{j_*+1}^{\mathfrak{w}}. 
\]%
Since $N(k,\ell)_{j_*+1}^{\mathfrak{w}}=\{e\},$ we actually have equality. 
Unraveling the
brackets in $\Theta (\overrightarrow{a}_{j},c)$ we get an expression in
the powers of the generators satisfying the desired conditions.

Suppose the claim holds for $h+1.$ Given a commutator $c$ with $w(c)=\bar{w}_h$,
let again $\sigma_1,\dots \sigma_m$ ($m$ depends on  $c$)
be the formal word on the generators associated with $c$. For $x\in  \mathbb Z$, 
$|x|\le F_c(d)$, decompose $x$ as above and use Lemma \ref{lem-Theta} to write
\[
c^{x}=u_{0}^{-1}\Theta (\overrightarrow{a}_{0},c)...u_{m}^{-1}%
\Theta (\overrightarrow{a}_{m},c) 
\]%
where $u_{i}\in N(k,\ell)^\mathfrak w_{h+1}$ 
can be represented by a word $\omega _{i}$  with 
$\deg _{c_{j}}\upsilon _{i}\leq F_{c_{j}}(Cd)$
for all $j.$  By Lemma \ref{lem-Free}, $u_i$ can also be represented
in the form  $\prod_{j\ge h+1} b_j^{y_{i,j}}$ with $|y_{i,j}|\le F_{b_j}(Cd)$.
Applying the induction hypothesis to each terms of these products
we can now write $c^{x}$ in the desired form
$c^x=s_{i_{1}}^{x_{1}}s_{i_{2}}^{%
x_{2}}...s_{i_{p}}^{x_{p}}$.  \end{proof}

\begin{proof}[Proof of Assertion (ii) in Theorem \ref{th-coord2FF}]
By Theorem \ref{th-coord2FF}(i), any $g\in N(k,\ell)$ with
$\|g\|_{S,\mathfrak F}\le F_S^{-1}(r)$, $r>0$, as a unique 
representation of the form
$g= \prod_1^\tau b_j^{x_j}$ with $|x_j|\le F_{b_j}(Cr)$.
Applying Lemma \ref{lem-cprods} with  $c=b_j, x=x_j$ for each $j=1,\dots,\tau$
produces a sequence  $((i_n)_1^p$ (independent of $g$) and a sequence
$(x'_n)\in \mathbb Z^p$ (depending on $g$)  with $|x'_n|\le F_{s_{i_n}}(Cr)$ for all $n\in \{1,\dots,p\}$ and 
such that 
$$g= \prod _1^p s_{i_n}^{x'_n}.$$
\end{proof}

\subsection{End of the proof of Theorem \ref{th-coord2F}} 
In order to finish the proof of Theorem \ref{th-coord2F} for a general 
finitely  generated nilpotent group $G$, we simply need to improve upon 
Theorem \ref{th-coord2FW}(i). 
Namely, Theorem \ref{th-coord2FW}(i) provide a decomposition of any element 
$g$ with $\|f\|_{\mathfrak C(S),\mathfrak F}\le F_S(r)$ in the form
$$g=\prod_1^t c_i^{x_i},\;\;|x_i|\le F_{c_i}(Cr).$$
Here $(c_i)_1^t$ is an enumeration of $\mathfrak C_+^*(S)$ so that 
$w(c_i) \preceq w(c_{i+1})$.

Now, let $(b_i)_1^\tau$ be a collection of formal commutators  with
$w(b_i)\preceq w(b_{i+1})$. For $j\in \{1,\dots, j_*\}$, let
$$m_j= \max\{i: w(b_i)= \bar{w}_j\}.$$ 
Clearly, $w(b_i)=\bar{w}_j$ if and only if $m_{j-1}+1\le i\le m_j$.
Recall that $R^\mathfrak w_j$ is the torsion free rank of the abelian group
$G^\mathfrak w_j/G^\mathfrak w_{j+1}$. 
We make two natural assumptions on the sequence $(b_i)$:
\begin{itemize}
\item[(A1)] For each $j$, $\{b_i': m_{j-1}<i\le m_j\}$ generates $G^\mathfrak w_j$ modulo $G^{\mathfrak w}_{j+1}$.
\item[(A2)] For each $j$, $\{b_i': m_{j-1}<i\le m_{j-1}+R^\mathfrak w_j\}$ 
is free in  $G^\mathfrak w_j/G^{\mathfrak w}_{j+1}$. 
\end{itemize}
Note that, since $R^\mathfrak w_j$ is the torsion free rank of $G^\mathfrak w_j/G^\mathfrak w_{j+1}$, (A2) implies that (the image of) 
$\{b_i': m_{j-1}<i\le m_{j-1}+R^\mathfrak w_j\}$ 
generates a subgroup of finite index in $G^\mathfrak w_j/G^\mathfrak w_{j+1}$. 

\begin{lem}\label{lem-bullet3} Referring to the notion introduce above, 
assume that $(b_i)_1^\tau$ satisfies {\em (A1)}. Then there exists $C\in (0,\infty)$ such that, for any $h=1,\dots, j_*$, any $g\in G$ that
can be written in the form
$$g=\prod_{i:w(c_i)\succeq \bar{w}_h} c_i^{x_i},\;\;|x_i|\le F_{c_i}(r)$$
can also be written in the from
$$g=\prod_{i:w(b_i)\succeq \bar{w}_h} b_i^{y_i},\;\;|x_i|\le F_{b_i}(Cr).$$
\end{lem}
\begin{proof}The proof is by backward induction on $h$ and is similar to the proof of Lemma \ref{lem-Free}. The details are omitted.
\end{proof}

\begin{pro} Assume that, for each $j$, the image of 
$$\{b_i: m_{j-1}+1\le i\le m_{j-1}+R_j\}$$ 
in $G^\mathfrak w_j/G^\mathfrak w_{j+1}$
generates a subgroup of finite index in
$G^\mathfrak w_j/G^\mathfrak w_{j+1}$. Then there exists a constant $C>0$ 
such that for any word $\omega $ in 
$\{b_i: w(b_i)\succeq \bar{w}_h\}^{\pm 1}$ with $\deg _{b_{i}}\omega
\leq F_{b_{i}}(r)$ for all $i,$ there is a word $\omega'$ of the form
\[
\omega'=\prod_{i=m_{h-1}+1}^{\tau}b_{i}^{x_{i}}\]
with 
\[|x_i| \le \left\{\begin{array}{cc} F_{b_i}(Cr) &\mbox{ for } m_{j-1}+1\le i\le m_{j-1}+R^\mathfrak w_j\\
C& \mbox{ for } m_{j-1}+R^\mathfrak w_j+1\le i\le m_j\end{array}  \right. 
\]%
for 
$j\in \{h,\dots, j_*\}$ 
and such that $\pi (\omega')=\pi (\omega )$. 
\end{pro}

\begin{proof} The proof is by 
backward induction on $h$. When $h=j_*$,$G_{j_*}^{\mathfrak{w}}$ is abelian
and the desired result holds.

In general, let  $\omega$ as in the proposition. By an application
of Lemmas \ref{lem-bullet2}-\ref{lem-bullet3}, we obtain a word
$\omega_{1}=\prod_{j=m_{h-1}+1}^{t}b_{j}^{x_{j}}$
with $|x_j|\le F_{b_j}(Cr)$ for all $j\ge m_{h-1}+1$ and such that 
$\pi(\omega)=\pi(\omega_1)$.

By hypothesis, the images of the commutators $
b_{j}, m_{h-1}+1\leq j\leq m_{h-1}+R_{h}^{\mathfrak{w}}$,
generates a subgroup of
finite index in $G_{h}^{\mathfrak{w}}/G_{h+1}^{\mathfrak{w}}$. 
Let $N_{h}$
denote the index. Then for $m_{h-1}+R_{h}^{\mathfrak{w}}+1\leq j\leq m_{h},$
there exists $a_{1}^{(j)},...,a_{R_{h}^{\mathfrak{w}}}^{(j)}\in 
\mathbb{Z}
$ such that 
$$b_{j}^{N_{h}}=
b_{m_{h-1}+1}^{a_{1}^{(j)}}...b_{m_{h-1}+R_{h}^{\mathfrak{w}}}^{a_{R_{h}^{%
\mathfrak{w}}}^{(j)}} \mbox{ mod }G_{h+1}^{\mathfrak{w}},$$ 
 that is $$\pi
(b_{j}^{N_{h}})=\pi (b_{m_{h-1}+1}^{a_{1}^{(j)}}...b_{m_{h-1}+R_{h}^{%
\mathfrak{w}}}^{a_{R_{h}^{\mathfrak{w}}}^{(j)}}v_{j}),$$ where $v_{j}$ is a
word in $\{c_i:w(c)\succeq \overline{w}_{h+1}\}^{\pm 1}.$
In
$$\omega_{1}=\prod_{j=m_{h-1}+1}^{t}b_{j}^{x_{j}},$$
for each $j\in \{m_{h-1}+R_{h}^{\mathfrak{w}}+1,\dots, m_{h}\},$ 
write $x_{j}=z_{j}N_{h}+y_{j}$ with $0\leq y_{j}<N_{h}$ and
replace $b_{j}^{N_{h}}$ by
the word 
$$\omega _{j}=b_{m_{h-1}+1}^{a_{1}^{(j)}}...b_{m_{h-1}+R_{h}^{%
\mathfrak{w}}}^{a_{R_{h}^{\mathfrak{w}}}^{(j)}}v_{j}.$$
This produce a new word 
\[
\omega _{1}^{\prime }=\prod\limits_{j=m_{h-1}+1}^{m_{h-1}+R_{h}^{\mathfrak{w%
}}}b_{j}^{x_{j}}\cdot \prod\limits_{j=m_{h-1}+1+R_{h}^{\mathfrak{w}%
}}^{m_{h}}\omega _{j}^{z_{j}}b_{j}^{y_{j}}\cdot
\prod\limits_{j=m_{h}+1}^{t}b_{j}^{x_{j}} 
\]%
satisfying $\pi (\omega _{1}^{\prime })=\pi (\omega _{1})$. For $%
m_{h-1}+1\leq j\leq m_{h-1}+R_{h}^{\mathfrak{w}},$ 
\[
\deg _{b_{j}}\omega _{1}^{\prime }\leq \left\vert x_{j}\right\vert
+\sum_{m_{h-1}+R_{h}^{\mathfrak{w}}+1\leq i\leq m_{h}}|
a_{j-m_{h-1}}^{(i)}| \left\vert x_{i}\right\vert , 
\]%
By hypothesis, $\deg _{b_{j}}\omega \leq F_{b_{j}}(Cd)\leq \mathbf{F}%
_{h}(C_1d)$ for all $m_{h-1}+1\leq j\leq m_{h}$ and
$$\max\{
|a_{n}^{(i)}|: m_{h-1}+R_{h}^{\mathfrak{w}%
}+1\leq i\leq m_{h}, 1\le n\le R^\mathfrak w_h\} = C_h<\infty.$$  
Hence, for $m_{h-1}+1\leq j\leq m_{h-1}+R_{h}^{\mathfrak{w}}$, we obtain 
$$
\deg _{b_{j}}\omega _{1}^{\prime } \leq 
C_1 (m_{h}-m_{h-1})
\mathbf{F}_{h}(Cd)  \leq \mathbf{F}_{h}(C_{2}d).$$
For $%
m_{h-1}+R_{h}^{\mathfrak{w}}+1\leq j\leq m_{h}$, $\deg _{b_{j}}\omega \leq
N_{h}$.  Finally, for any $c\in \{c_i: 1\le i\le t\}$ with
 $w(c)\succ \bar{w}_h$, we have $F_{c}\succ \mathbf{F}_{h}$ and 
\begin{eqnarray*}
\deg _{c}\omega _{1}^{\prime } &\leq &\deg _{c}\omega
_{1}+\sum_{m_{h-1}+R_{h}^{\mathfrak{w}}+1\leq k\leq m_{h}}\left\vert
z_{k}\right\vert \deg _{c}v_{k} \\
&\leq &F_{c}(C_{3}d).
\end{eqnarray*}%
Applying Lemmas \ref{lem-bullet2}-\ref{lem-bullet3} 
to $\omega _{1}^{\prime },$ we obtain a
word $\omega'$ with $\pi(\omega)=\pi(\omega')$ and 
\[
\omega _{2}=\prod\limits_{j=m_{h-1}+1}^{m_{h-1}+R_{h}^{\mathfrak{w}}}b_{j}^{%
\widetilde{x_{j}}} \prod\limits_{j=m_{h-1}+1+R_{h}^{\mathfrak{w}%
}}^{m_{h}}b_{j}^{y_{j}} \prod\limits_{j>m_h}b_{j}^{\widetilde{%
x_{j}}} 
\]%
where $\left\vert \widetilde{x_{j}}\right\vert \leq \mathbf{F}_{h}(C_{1}d)$
for $m_{h-1}+1\leq j\leq m_{h-1}+R_{h}^{\mathfrak{w}};$ $0\leq y_{j}<N_{h}$
for $m_{h-1}+R_{h}^{\mathfrak{w}}+1\leq j\leq m_{h}$, and $\left\vert 
\widetilde{x}_{j}\right\vert \leq F_{c_{j}}(C_{2}^{\prime }d)$ for all $%
j>m_{h}.$ Now, apply the induction hypothesis to 
$\prod\limits_{j=m_{h}+1}^{t}b_{j}^{%
\widetilde{x_{j}}},$ to obtain the desired conclusion.
\end{proof}

We end with the following simple improvement of the last statement in 
Theorem \ref{th-coord2F}. The p[roof is a simple combination of the 
previous proposition together with Lemma \ref{lem-cprods}.

\begin{theo}\label{th-coord2Fimp}
Let $G$ be a nilpotent group equipped with a generating $k$-tuple 
$S=(s_1,\dots,s_k)$. Let $\mathfrak w$, $\mathfrak F$ 
be weight and weight-function systems on $S$ satisfying 
{\em (\ref{F1})-(\ref{F2})}.  
Let $\Sigma=(c_1,\dots,c_t)$ be a tuple of formal commutators in 
$\mathfrak C(S)$ with  non-decreasing weights  $w(c_1)\preceq 
\dots\preceq  w(c_t)$. Let $m_j$, $j=0,\dots, j_*$ be defined by
$$\{c_i: w(c_i)=\bar{w}_j\}
= \{c_i: m_{j-1}<i\le m_j\}.$$
Assume  that (the image of) $\{c_i: w(c_i)=\bar{w}_j\}$
generates $G^\mathfrak w_j$ modulo $G^\mathfrak w_{j+1}$ and
that $ \{c_i: m_{j-1}<i\le m_{j-1} +R^\mathfrak w_j\}$
is free in $G^\mathfrak w_j/G^\mathfrak w_{j+1}$. 

There exist an integer $p=p(G,S,\mathfrak F)$, 
a constant $C=C(G,S,\mathfrak F)$ and a sequence
$(i_1,\dots,i_p)\in \{1,\dots, k\}^p$ such that 
if $g$ can be expressed as a word $\omega$ 
over $\mathfrak C(S)$
with $\mbox{\em deg}_{c}(\omega)\le F_{c}(r)$ for some $r\ge 1$  and all $c\in \mathfrak C(S)$ then
$g$ can be expressed in the form
$$g=\prod_{j=1}^p s_{i_j}^{x_j}
\mbox{ with } |x_j|\le C \left\{\begin{array}{ll} F_{{i_j}}(r) & \mbox{ if }
s_{i_j}\in \mbox{\em core}(S,\mathfrak w,\Sigma)\\
1 & \mbox{ if } s_{i_j}\not\in \mbox{\em core}(S,\mathfrak w,\Sigma).
\end{array} \right.$$
\end{theo}

\bibliographystyle{plain}
\def\cprime{$'$}

\end{document}